\documentclass[a4paper,10pt]{article}
\usepackage{amssymb}
\usepackage{amsmath, esint}
\usepackage{amsfonts}
\usepackage{a4wide}
\usepackage{amssymb}
\usepackage{amsthm}
\usepackage{color, xcolor}
\theoremstyle{plain}
\begingroup
\newtheorem{theorem}{Theorem}
\newtheorem{lemma}[theorem]{Lemma}
\newtheorem{proposition}[theorem]{Proposition}
\newtheorem{corollary}[theorem]{Corollary}

\newtheorem{question}[theorem]{Question}
\newtheorem{definition}[theorem]{Definition}

\newtheorem{rmk}[theorem]{Remark}

\endgroup

\theoremstyle{remark}
\begingroup
\endgroup

\mathchardef\emptyset="001F

\numberwithin{theorem}{section}
\numberwithin{equation}{section}

\newcommand{\op}[1]{{\rm{#1}}}

\newcommand{\be}{\begin{equation}}
\newcommand{\ee}{\end{equation}}
\newcommand{\R}{\mathbb{R}}

\def\lf{\left}
\def\rg{\right}

\newcommand{\verti}[1]{\ensuremath{\left\lvert #1 \right\rvert}}
\newcommand{\vertii}[1]{\ensuremath{\left\lVert #1 \right\rVert}}
\newcommand{\gauges}[1]{\ensuremath{(#1,G)}}
\newcommand{\connforms}[1]{\ensuremath{(#1, \wedge^1#1\otimes\mathfrak{g})}}
\newcommand{\curvforms}[1]{\ensuremath{(#1, \wedge^2#1\otimes\mathfrak{g})}}
\newcommand{\connformsr}[2]{\ensuremath{(#1, \wedge^1#2\otimes\mathfrak{g})}}

\begin{document}
\title{The Space of Weak Connections in High Dimensions}
\author{ Mircea Petrache\footnote{Pontificia Universidad Catolica de Chile, Facultad de Matematicas, Av. Vicuna Mackenna 4860, Santiago, 6904441, Chile.} and Tristan Rivi\`ere\footnote{Department of Mathematics, ETH Z\"urich, R\"amistrasse 101, CH-8092 Z\"urich, Switzerland.} }
\date{ }
\maketitle
\begin{abstract}
The space of Sobolev connections, as it has been introduced for studying the variation of Yang-Mills Lagrangian in the critical dimension $4$, happens not to be weakly sequentially complete in dimension larger than $4$. This is a major obstruction for studying the variations of this important Lagrangian in high dimensions.
The present paper generalizes the result \cite{PRym} valid in $5$ dimensions to arbitrary dimension and introduces a space of so called ''weak connections'' for which we prove the weak sequential closure under Yang-Mills energy control. We also establish a strong approximation property of any weak connection by smooth connections away from codimension 5 polyhedral sets. This last property is used in a subsequent work in preparation \cite{PRepsi} for establishing the partial regularity property for general stationary Yang-Mills weak connections.
\end{abstract}

\section{Introduction}
%
Motivated by geometric applications of first importance, the analysis of Yang-Mills energy up to the conformal dimension $4$ (the dimension at which the Lagrangian is invariant under dilations) has known a fast and spectacular development in the late 70's early 80's. To that purpose, the space of 
{\it Sobolev Connections} of a given smooth bundle have been introduced and studied (see \cite{Uhl1}, \cite{Uhl2}, \cite{FrUh}). This space enjoys a {\it sequential almost weak closure property} under Yang-Mills Energy control assumptions. ``Almost'' in the sense that, from a sequence of uniformly bounded Yang-Mills energy Sobolev connections of a given bundle, one can extract a subsequence such that it converges weakly modulo gauge transformations away from finitely many points to a limiting Sobolev connection, but on a possibly different smooth bundle.

\medskip

In \cite{PRym} (see also \cite{Riv1}) the two authors reformulated this classical ``sequential almost weak closure property'' into an exact ``sequential weak closure property'' in the following way. Let $G$ be a compact Lie group and 
$(M^n,h)$  a compact riemannian manifold. Introduce the space of so called {\it Sobolev connections} defined by
\[
{\mathfrak A}_G(M^n):=\lf\{ 
\begin{array}{l}
A\in L^2(\wedge^1M^n,{\mathfrak g})\ ; \ \int_{M^n}|dA+A\wedge A|_h^2\ \mathrm{d}vol_h<+\infty\\[5mm]
\mbox{ locally }\exists\ g\in W^{1,2}\quad\mbox{ s.t. }\quad A^g:=g^{-1} dg+g^{-1}\, A\, g\in W^{1,2}
\end{array} 
\rg\}
\]
then we have proved the following result.
\begin{theorem}[Compactness in dimensions $\le 4$ \cite{Riv1}]\label{th-VII.1}
For $n\le 4$ the space ${\mathfrak A}_G(M^n)$ is weakly sequentially closed below any given Yang-Mills energy level: precisely
For any $A^k\in  {\mathfrak A}_G(M^n)$ satisfying
\[
\limsup_{k\rightarrow+\infty}YM(A^k)=\int_{M^n}|dA^k+A^k\wedge A^k|_h^2\ \mathrm{d}vol_h<+\infty
\]
there exists a subsequence $A^{k'}$ and a Sobolev connection $A^\infty\in  {\mathfrak A}_G(M^n)$ such that
\[
d^2_\mathrm{conn}(A^{k'},A^\infty):=\inf_{g\in W^{1,2}(M^n,G)}\int_{M^n}|A^{k'}-(A^\infty)^g|_h^2\ \mathrm{d}vol_h\longrightarrow 0
\]
moreover
\[
YM(A^\infty)\le\liminf_{k'\rightarrow 0} YM(A^{k'})\quad.
\]
\end{theorem}
\begin{rmk}\label{rm-VII.1}
Observe that the space ${\mathfrak A}_G(M^n)$ contains for instance global $L^2$ one forms taking values into the Lie algebra ${\mathfrak g}$ that correspond to smooth connections of sore Yang-Mills energy of  a sequence of such smooth connections is uniformly bounded, we can extract a subsequence converging weakly to a Sobolev connection and  corresponding 
possibly to \underbar{another} $G-$bundle. This possibility of ''jumping'' from one bundle to another is encoded in the definition of ${\mathfrak A}_G(M^n)$.
\end{rmk}
Because of this weak closure property the space ${\mathfrak A}_G(M^n)$ is the ad-hoc space for studying the variations of Yang-Mills energy in dimension less or equal than 4. This is however not the case in higher dimension. We have the following proposition.
\begin{proposition}[Sobolev connections in dimension $>4$ \cite{Riv1}]\label{pr-VII.1}
For $n> 4$ the space ${\mathfrak A}_{SU(2)}(M^n)$ is \underbar{ not} weakly sequentially closed below any given Yang-Mills energy level: namely
there exist $A^k\in  {\mathfrak A}_{SU(2)}(M^n)$ satisfying
\[
\limsup_{k\rightarrow+\infty}YM(A^k)=\int_{M^n}|dA^k+A^k\wedge A^k|^2\ \mathrm{d}vol_h<+\infty
\]
and a Sobolev connection $A^\infty\in  L^2$ such that
\begin{equation}\label{donaldson_conn}
d^2_\mathrm{conn}(A^{k'},A^\infty):=\inf_{g\in W^{1,2}(M^n,SU(2))}\int_{M^n}|A^{k'}-(A^\infty)^g|_h^2\ \mathrm{d}vol_h\longrightarrow 0
\end{equation}
but such that in every neighborhood $U$ of every point of $M^n$ there is \underbar{no} $g$ such that $(A^\infty)^g\in W^{1,2}(U)$.
\end{proposition}
In the search of the suitable formulation for variational problems involving Yang-Mills Lagrangians in higher dimensions, \cite{tian} and \cite{taotian} introduced the class of so-called \emph{admissible connections}, which are connections that are smooth outside a rectifiable set of codimension $4$. We note here that even assuming that our conections locally coincide with a Sobolev connection outside a rectifiable set of codimension $4$ still \emph{does not allow} to recover the weak closure under Yang-Mills energy control denied by Proposition \ref{pr-VII.1}.

\medskip

The main purpose of the present work is to propose a space of so called {\bf weak connections}, which extends the space of Sobolev connections ${\mathfrak A}_G(M^n)$ in the case $n>4$ and enjoying a {\it sequential weak closure property} under the control of Yang-Mills energy. This space also includes the class of admissible connections from \cite{tian, taotian}. In line with \cite{PRym} where the case $n=5$ is presented, we introduce the following definition.
\begin{definition}[Weak connections]\label{def:weakconn}
Let $G$ be a compact Lie group and 
$(M^n,h)$  a compact riemanian manifold. For $n\le 4$ the space of {\it weak connections} ${\mathcal A}_G(M^n)$ is defined to coincide with the space of  {\it Sobolev connections} defined by
\[
{\mathcal A}_G(M^4)={\mathfrak A}_G(M^4):=\lf\{ 
\begin{array}{l}
A\in L^2(\wedge^1M^n,{\mathfrak g})\ ; \ \int_{M^n}|dA+A\wedge A|_h^2\ \mathrm{d}vol_h<+\infty\\[5mm]
\mbox{ locally }\exists\ g\in W^{1,2}\quad\mbox{ s.t. }\quad A^g\in W^{1,2}
\end{array} 
\rg\}
\]
For $n>4$ we define the space of {\it weak connections} ${\mathcal A}_G(M^n)$ to be
\[
{\mathcal A_G}( M^n):=\lf\{ 
\begin{array}{l}
A\in L^2(M^n,\wedge^1TM\otimes{\mathfrak g})\ ; \ \int_{M^n}|dA+A\wedge A|_h^2\ \mathrm{d}vol_h<+\infty\\[5mm]
\forall f\in C^\infty(M^n,\mathbb R^{n-4}),\ \mbox{ a.e. }y\in \mathsf{Reg}(f),\  \iota_{f^{-1}(y)}^\ast A\in {\mathcal A}_G(f^{-1}(y))
\end{array} 
\rg\}\,,
\]
where $\mathsf{Reg}(f)\subset\mathbb R^{n-4}$ denotes the regular values of $f$, and for a submanifold $\Sigma\subset M^n$ we denote by $\iota_{\Sigma}^\ast A$ is the restriction of the 1-form $A$ to $\Sigma$.
\end{definition}
We next introduce a gauge-invariant pseudo-distances $\delta$ between weak connection forms $A,A'\in \mathcal A_G(M^n)$, fitting to the above definition, by setting
\begin{eqnarray}\label{def_delta}
\delta^2_{\mathrm{conn}}(A,A')&:=&\sup_{f\in C^\infty(M^n,\mathbb R^{n-4})}\inf_{g:M^n\to G}\int_{M^n}\left|\left(dg + Ag-gA'\right)\wedge f^*\omega\right|_h^2\frac{\mathrm{d}vol_h}{\left|f^*\omega\right|_h}\\
&=&\sup_{f\in C^\infty(M^n,\mathbb R^{n-4})}\inf_{g:M^n\to G}\int_{\mathbb R^{n-4}}\mathrm{d}\omega\int_{f^{-1}(y)}\left| i^*_{f^{-1}(y)}\left(dg + Ag-gA'\right)\right|_h^2\mathrm{d}\sigma_h,\qquad \label{coarea_dist}
\end{eqnarray}
where the infimum is taken over all measurable $g$, and we use the notations $\mathrm{d}\omega=\mathrm{d}\omega_{n-4}=dy_1\wedge\cdots\wedge dy_{n-4}$ for the volume form of $\mathbb R^{n-4}$ and $\mathrm{d}\sigma_h$ for the $4$-dimensional surface element of $f^{-1}(y)$ for $y\in \mathsf{Reg}(f)$. The above integrals are well-defined because almost all values of $f$ are regular by Sard's theorem. In order to justify the good-definition of the expressions in \eqref{def_delta}, \eqref{coarea_dist} under the low regularity assumption on $g$, we note that $dg$ can be interpreted as a distribution, and testing it against $f^*\omega$ is a well-defined operation. Because $G$ is bounded, the terms $Ag,gA'$ are the product of an $L^\infty$ and an $L^2$ function, and thus their wedge with $f^*\omega$ is also well-defined. Finally, we can pass from \eqref{def_delta} to \eqref{coarea_dist} via the co-area formula.

\medskip

Note that the integrand in \eqref{def_delta},\eqref{coarea_dist} is finite precisely if there exists a (measurable on $M^n$) gauge $g$ which \emph{restricted to a.e. levelsets} $f^{-1}(y)$ is in $W^{1,2}(f^{-1}(y),G)$, which is fitting to Definition \ref{def:weakconn}.

\medskip

In order to see that our defintion \eqref{def_delta} is a higher-dimensional extension of the Donaldson distance $d$ given in \eqref{donaldson_conn} and used in dimension $n\le 4$, we note the following alternative expression\footnote{For two functions $f,g: X\to \mathbb R$ we use the notation $f\asymp g$ if there exists $C>0$ such that for all $x\in X$ there holds $C^{-1}f(x)\le g(x)\le Cf(x)$.} of \eqref{donaldson_conn}, paralleling \eqref{def_delta} (for the proof, see Lemma \ref{lem:equiv}):
\begin{equation}\label{def_donaldson_conn}
d^2_{\mathrm{conn}}(A,A')\asymp\inf_{g:M^n\to G}\sup_{f\in C^\infty(M^n,\mathbb R^{n-4})}\int_{M^n}\left| \left(dg + Ag-gA'\right)\wedge f^*\omega\right|_h^2\frac{\mathrm{d}vol_h}{\left|f^*\omega\right|_h}.
\end{equation}
We will use the notation $A\simeq A'$ whenever $A, A'\in\mathcal A_G(M^n)$ are related by a gauge map $g\in W^{1,2}(M^n,G)$, i.e. $g^{-1}dg+g^{-1}Ag=A'$. Then (see \cite{isobehigher} for example) the so-defined relation $\simeq$ is an equivalence relation, and we may consider the quotient space $\mathcal A_G(M^n)/\simeq$ formed of equivalence classes 
\begin{equation}\label{eqclass}
[A]:=\{A'\in \mathcal A_G(M^n): A'\simeq A\}.
\end{equation}
For \emph{curvature forms} $F,F'\in L^2(\wedge^2M^n,\mathfrak{g})$ we also obtain the well-known equivalence relation $\simeq$ (denoted by abuse of notation by the same symbol, because as it turns out $A\simeq A'$ is equivalent to $F_A\simeq F_{A'}$), according to which $F\simeq F'$ if there exists a measurable $g:M^n\to G$ such that $g^{-1}Fg=F'$. On $\mathfrak g$-valued forms we use the canonical conjugation-invariant norm, under which if $F\simeq F'$ then pointwise a.e. there holds $|F|=|F'|$. This means that, as is well-known, the Yang-Mills energy is constant on each equivalence class $[A]$.

\medskip

Between curvature forms $F,F'\in L^2(\wedge^2 M^n,\mathfrak{g})$ we introduce (and by abuse of notation, we use the same notation as for connection forms), pseudo-distances defined by the following formulas, where again the infimum is taken over measurable $g:M^n\to G$:
\begin{eqnarray}
d^2_{\mathrm{curv}}(F,F')&:=&\inf_{g:M^n\to G}\int_{M^n}\left|g^{-1}Fg-F'\right|_h^2\mathrm{d}vol_h\nonumber\\[3mm]
&\asymp&\inf_{g:M^n\to G}\sup_{f\in C^\infty(M^n,\mathbb R^{n-4})}\int_{M^n}\left| \left(g^{-1}Fg-F'\right)\wedge f^*\omega\right|_h^2\frac{\mathrm{d}vol_h}{\left|f^*\omega\right|_h},\label{def_donaldson_curv}
\end{eqnarray}
and
\begin{equation}
\delta^2_{\mathrm{curv}}(F,F'):=\sup_{f\in C^\infty(M^n,\mathbb R^{n-4})}\inf_{g:M^n\to G}\int_{M^n}\left| \left(g^{-1}Fg-F'\right)\wedge f^*\omega\right|_h^2\frac{\mathrm{d}vol_h}{\left|f^*\omega\right|_h},\label{def_delta_curv}
\end{equation}
The evident motivation for introducing the Donaldson pseudo-distance $d$ from \eqref{donaldson_conn} or \eqref{def_donaldson_conn} between connection forms (and respectively, the distance \eqref{def_donaldson_curv} between curvature forms), is that it induces, arguably, \emph{the simplest possible} geometric distance on equivalence classes of connections (resp. curvatures). 

\medskip

We have the following results concerning the above distances, proved in Appendix \ref{app:distances}:
\begin{proposition}\label{prop:distances}
The following hold:
\begin{enumerate}
\item For $G=SU(2)$ there hols $d_\mathrm{curv}\asymp\delta_\mathrm{curv}$. 
\item For any $G$ we have for weak connections $A,B$ that $d_\mathrm{curv}(F_A,F_B)=0$ if and only if $d_\mathrm{conn}(A,B)=0$ and $\delta_\mathrm{curv}(F_A,F_B)=0$ if and only if $\delta_\mathrm{conn}(A,B)=0$.
\end{enumerate}
\end{proposition}

The above proves that the pseudo-distance $\delta_\mathrm{conn}$ from \eqref{def_delta} (respectively $\delta_\mathrm{curv}$ from \eqref{def_delta_curv}) induces a distance on equivalence classes $[A]$ as in \eqref{eqclass} (resp. on equivalence classes of curvature forms) in the case $G=SU(2)$. In the case of other simply connected Lie groups $G$, we are tempted to conjecture that this continues to hold:
\begin{question}\label{q:equiv_general}
 Is it true that for general simply-connected compact Lie groups $G$ there holds $d_\mathrm{curv}\asymp\delta_\mathrm{curv}$ and $d_\mathrm{conn}\asymp\delta_\mathrm{conn}$?
\end{question}
\begin{rmk}\label{q:equival_general}
When comparing two pseudo-distances $\delta_1,\delta_2$ over a space $X$ we can compare them at increasing levels of precision, and ask:
\begin{enumerate}
 \item whether $\delta_1=0\Leftrightarrow \delta_2=0$, i.e. if they generate the same metric space;
 \item whether the topology induced by $\delta_1$ is the same as the topology induced by $\delta_2$;
 \item whether $\delta_1\asymp\delta_2$.
\end{enumerate}
While it seems plausible that Question \ref{q:equiv_general} holds and the strong statement $d_\mathrm{conn}\asymp\delta_\mathrm{conn}$ is valid for general Lie groups $G$, outside the case $G=SU(2)$ we don't even know a proof (or counterexample) of the assertion that, given two constant $\mathfrak g$-valued $2$-forms $F,F'\in \wedge^2\mathbb R^n\otimes\mathfrak g$, the following conditions are equivalent: 
\begin{enumerate}
\item[(a)] For all $4$-dimensional subspaces $H\in Gr(n,4)$ there exists $g_H\in G$ such that the restrictions to $H$ are conjugated by $g_H$, i.e. $g_H^{-1}i^*_HFg_H=i^*_HF$.
\item[(b)] There exists $g\in G$ such that $g^{-1}Fg=F'$.
\end{enumerate}
The non-trivial part of the question is to prove that (a) implies (b), as the reverse implication follows by restriction.
\end{rmk}
The important property of interest for us is the fact that the pseudo-distance $\delta_\mathrm{conn}$ \emph{metrizes the weak convergence} under a Yang-Mills energy control, and we have:
\begin{theorem}[Compactness of weak connections in dimension $>4$]\label{thm:weakclosure}
The space ${\mathcal A}_{G}(M^n)$ is weakly sequentially closed below any Yang-Mills energy level. More precisely, let $A^j\in {\mathcal A}_{G}(M^n)$ such that
\[
\limsup_{j\rightarrow +\infty}\int_{M^n}|dA^j+A^j\wedge A^j|^2\ d\mathrm{Vol}_{M^n}<+\infty
\]
then there exists a subsequence $j'$ and  $A^\infty\in {\mathcal A}_{G}(M^n)$ such that for the equivalence classes of connections we have 
\[
\delta_{\mathrm{conn}}(A_j,A)\longrightarrow 0.
\]
\end{theorem}
This result can be interpreted as a nonlinear version of \emph{Rellich-Kondrachov's compactness theorem}, in which the linear operation of taking the gradient of a Sobolev $W^{1,2}$-regular function, is now replaced by the nonlinear operation of taking the curvature of a weak connection.

\medskip

The proof of the sequential weak closure property of ${\mathcal A}_{G}(M^n)$ uses a strong approximation result of an elements in  ${\mathcal A}_{G}(M^n)$ 
by smooth connections away from polyhedral codimension 5 singularities, whose space $\mathcal R^\infty(M^n)$ is precisely defined in Section \ref{sec:strapp} (see \eqref{Rinf}). 
\begin{theorem}\label{thm:strongapprox_intro}
If $A\in\widetilde{\mathcal A}_G(M^n)$ then there exists a sequence of connection forms $A_j\in \mathcal R^\infty(M^n)$ and a sequence of gauge changes $g_j\in W^{1,2}(M^n,G)$ such that if $F_j:=dA_j+A_j\wedge A_j$ then as $j\to\infty$ there holds
\begin{equation}\label{strapp_intro}
\|g_j^{-1}dg_j+g_j^{-1}A_j g_j - A\|_{L^2(M^n)}\to 0, \quad \|g_j^{-1}F_jg_j - F_A\|_{L^2(M^n)}\to 0\ .
\end{equation}
\end{theorem}
With the definitions \eqref{def_donaldson_conn} and \eqref{def_donaldson_curv}, the existence of $g_j$ such that \eqref{strapp_intro} holds is equivalent to saying that $d(A_j,A)\to 0$. 

\medskip

The main consequence of Theorem \ref{thm:strongapprox_intro} in the study of partial regularity for stationary Yang-Mills connections, is that it paves the way to apply the partial regularity results of \cite{MR}, a step which we plan to take in the paper \cite{PRepsi} in general dimension. 

In \cite{MR} it was proved that if $F$ has small Morrey norm condition on $F$ and $A$ is approximated by smooth connections with Morrey norm control rather than by elements of $\mathcal R^\infty(M^n)$ as in Theorem \ref{thm:strongapprox_intro}, then it is then possible to extract controlled Coulomb gauges of $A$ which allow to prove sharp $\epsilon$-regularity estimates. In \cite{PRepsi}, by the approximation procedure leading to Theorem \ref{thm:strongapprox_intro} we prove precisely such approximating smooth connections, and thus \cite{MR} directly leads to the sharp partial regularity result for weak connections. This was done in dimension $5$ in \cite{PRym}. As a consequence of these results, we have in a unified variational framework in all dimensions the compactness result of Theorem \ref{thm:weakclosure} and the partial regularity theory extending \cite{taotian} and \cite{MR}.

\medskip

A further consequence of Theorem \ref{thm:strongapprox_intro} we also obtain the following important property of weak connections:
\begin{proposition}[Bianchi identity for weak curvatures]\label{Bianchi}
 Assume that $A, F$ are the $L^2$ curvature and connection forms corresponding to a weak connection class $[A]\in\mathcal A_{G}(\mathbb R^5)$. Then the equation
\begin{equation}\label{bianchi}
 d_AF:=dF+[F,A]=0
\end{equation}
holds in the sense of distributions.
\end{proposition}
\subsection{Structure of the proof}\label{sec:structproof}
In Section \ref{sec:app1} we construct controlled gauges on the sphere by extending Uhlenbeck's method \cite{Uhl2} to a new setting. Using the outcome of Section \ref{sec:app1}, in Section \ref{sec:ext1gc} we prove a result which allows to bootstrap to higher dimensions the local regularity of connection forms. In Section \ref{sec:goodbad} we introduce a local version of Definition \ref{def:weakconn}, defining the space $\widetilde{\mathcal A}_G([-1,1]^n)$; we then present a setup for the proof of our approximation result, by separating regions of concentration and non-concentration of energy. Then in Section \ref{sec:strapp}, the result of Section \ref{sec:ext1gc} is fit to the setting prepared in Section \ref{sec:goodbad}, allowing to prove the strong approximation from Theorem \ref{thm:strongapprox_intro} on the space $\widetilde{\mathcal A}_G([-1,1]^n)$ (see Theorem \ref{thm:strongapprox}). In Section \ref{sec:strcpt}, again in the local space $\widetilde{\mathcal A}_G([-1,1]^n)$, we prove that the Yang-Mills energy forms a $2$-weak gradient structure for the connections in the appropriate metric, which allows to prove a local version (see Theorem \ref{thm:weakclosure}) of our compactness of Theorem \ref{thm:wclos}. In Section \ref{sec:global}, based on the outcomes of Sections \ref{sec:strapp} and \ref{sec:strcpt}, we complete the manifold-case of the proofs of Theorems \ref{thm:strongapprox_intro} and \ref{thm:weakclosure}. Finally, in Appendix \ref{app:distances} we prove that $\delta$ induces a distance on $\mathcal A_G(M^n)/\simeq$ for $G=SU(2)$, prove equivalence of different definitions of our distance, and briefly show how Question \ref{q:equival_general} is connected with other related questions.

\subsection*{List of notations}
%
\begin{itemize}
\item $G$: a compact Lie group
\item $\mathfrak g$: the Lie algebra of $G$
\item $n$: the dimension of the base space of our bundles.
\item $(M^n,h)$: a compact Riemannian manifold. Sometimes the Riemannian metric $h$ is omitted in the notation.
\item $A$: a $\mathfrak g$-valued $1$-form.
\item $F$: a $\mathfrak g$-valued $2$-form.
\item $g$: a map into the Lie group $G$, which can be interpreted as a ``singular gauge change'' of singular bundles.
\item $A^g$: the new expression $A^g=g^{-1}dg+g^{-1}Ag$ of the connection form after the gauge transformation $g$.
\item $\mathfrak A_G(M^n)$: the space of Sobolev connections which locally in some $W^{1,2}$-gauge are $W^{1,2}$. Used here only for $n\le 4$.
\item $\mathcal A_G(M^n)$: the space of weak connections on singular bundles. It coincides with $\mathfrak A_G(M^n)$ for $n\le 4$, but in $n=4$ we have the alternative definition based on $L^4$-spaces rather than $W^{1,2}$-spaces \eqref{ag4d}.
\item $\widetilde{\mathcal A}_{G}([-1,1]^n)$: the local model of the space $\mathcal A_G(M^n)$, described in Definition \ref{def:locmod}
\end{itemize}
%
\section{Controlled gauges on the $n$-sphere}\label{sec:app1}
%
In this section we follow the overall structure of the argument from \cite{Uhl2} to prove the following result:
\begin{theorem}\label{coulstrange}
Let $\pi:L^n(\mathbb S^n, \wedge^1T\mathbb S^n\otimes\mathfrak{g})\to L^n(\mathbb S^n, \wedge^1T\mathbb S^n\otimes\mathfrak{g})$ be a linear operator which is bounded on $L^p$ for $p\in[n,n+\epsilon_\pi[$ for some $\epsilon_\pi>0$, which satisfies $\pi\circ \pi=\pi$ and such that for $\omega\in L^1(\mathbb S^n, \wedge^1T\mathbb S^n\otimes\mathfrak{g})$ there holds
\begin{subequations}\label{newproppi}
\begin{equation}\label{newproppi1}
\|(id-\pi)\omega\|_{L^\infty(\mathbb S^n)}\le C_\pi \|\omega\|_{L^1(\mathbb S^n)}
\end{equation}
and
\begin{equation}\label{newproppi2}
d((id-\pi)\omega)=0.
\end{equation}
\end{subequations}
There exist constants $\epsilon_0, C$ with the following properties. If $A\in L^n(\mathbb S^n, \wedge^1T\mathbb S^n\otimes\mathfrak{g})$ is a connection form over $\mathbb S^n$ such that together with the corresponding curvature form $F$ satisfies
\[
\|F\|_{L^{n/2}(\mathbb S^n)}+\|A\|_{L^2(\mathbb S^n)}\leq \epsilon_0
\]
then there exists a gauge transformation $g\in W^{1,n}(\mathbb S^n, G)$ such that
\begin{subequations}\label{proppi-th}
\begin{equation}
g^{-1}dg\in \op{Im}(\pi),\label{proppi-th1}
\end{equation}
\begin{equation}
d^*_{\mathbb S^n}\left(\pi\left(A^g\right)\right)=0,\label{proppi-th2}
\end{equation}
\begin{equation}
\|A^g\|_{L^n(\mathbb S^n)}\leq C(\|F\|_{L^{n/2}(\mathbb S^n)}+ \|A\|_{L^2(\mathbb S^n)})\ .\label{proppi-th3}
\end{equation}
\end{subequations}
\end{theorem}
The proof consists in studying the case where the integrability exponents $n/2, n$ are replaced by $p/2, p$ for $p>n$ first, and then obtaining the $p=n/2$ case as a limit. Note that for $p>n$ the space $W^{1,p}(\mathbb S^n, G)$ embeds continuously in $C^0(\mathbb S^n, G)$, thus gauges $g$ of small $W^{1,p}$-norm will be expressible as $g=\op{exp}(v)$ for some $v\in W^{1,p}(\mathbb S^n, \mathfrak{g})$, due to the local invertibility of the exponential map $\op{exp}:G\to\mathfrak{g}$.

\medskip

We then consider the space
\begin{subequations}
\begin{equation}\label{Ep}
 E_p:=\left\{v\in W^{1,p}(\mathbb S^n, \mathfrak{g}):\:\int_{\mathbb S^n}v= 0,\ dv\in\op{Im}(\pi)\right\}
\end{equation}
where $x_k$ are the ambient coordinate functions relative to the canonical immersion $\mathbb S^n\to \mathbb R^{n+1}$. In case $p>n$ the Banach space $E_p$ is, by the above considerations, the local model of the Banach manifold
\begin{equation}\label{Mp}
 M_p:=\left\{ g\in W^{1,p}(\mathbb S^n, G):\: \text{\eqref{proppi-th1} holds}\right\}\ .
\end{equation}
We then consider the sets
\begin{equation}\label{Uep}
 \mathcal U^\epsilon_p:=\left\{A\in L^p(\mathbb S^n,\wedge^1T\mathbb S^n\otimes \mathfrak{g}):\:\|F_A\|_{L^{p/2}(\mathbb S^n)}\leq \epsilon_0\right\}
\end{equation}
and their subsets
\begin{equation}\label{Vecp}
 \mathcal V^{\epsilon,C_p}_p:=\left\{
\begin{array}{c}
A\in\mathcal U^\epsilon_p:\:\exists g\in M_n\text{ s.t. \eqref{proppi-th2} holds},\\[3mm]
\|A^g\|_{L^q}\leq C_q\|F\|_{L^{q/2}}\text{ for }q=p,n\\[3mm]
\text{and }\|F\|_{L^{n/2}}+\|A\|_{L^2}<\epsilon
\end{array}
\right\}\ .
\end{equation}
\end{subequations}
%
\subsection{Constructing modified Coulomb gauges: proof of Theorem \ref{coulstrange}}
%
%
Like in \cite{Uhl2} we prove theorem \ref{coulstrange} by showing that if $\epsilon_0>0$ is small enough then for $p\geq n$ we may find $C_p$ such that
\begin{equation}\label{ueqv}
 \mathcal V^{\epsilon_0,C_p}_p=\mathcal U^{\epsilon_0}_p\ .
\end{equation}
We are interested in \eqref{ueqv} just for $p=n$ but we use the cases $p>n$ in the proof: we successively prove the following statements.
\begin{enumerate}
\item[1.] $\mathcal U^\epsilon_p$ is path-connected.
\item[2.] For $p\geq n$ the set $\mathcal V^{\epsilon, C_p}_p$ is closed in $L^p(\mathbb S^n, \wedge^1T\mathbb S^n\otimes\mathfrak{g})$.
\item[3.] For $p>n$ there exists $C_p, \epsilon_0$ such that the set $\mathcal V^{\epsilon_0,C_p}_p$ is open relative to $\mathcal U^{\epsilon_0}_p$. In particular \eqref{ueqv} is true for $p>n$.
\item[4.] There exists $K$ such that if $g\in M_p,\:\|A^g\|_{L^n}\leq K$ and 
\[
d^*_{\mathbb S^n}(\pi(A^g))=0,\quad \|F\|_{L^{n/2}}+\|A\|_{L^2}<\epsilon_0
\]
then 
\[
\|A^g\|_{L^n}\leq C_n(\|F\|_{L^{n/2}}+\|A\|_{L^2})\ .
\]
\item[5.] The case $p=n$ of \eqref{ueqv} follows from the case $p>n$.
\end{enumerate}
%
\subsubsection*{Proof of step 1}
%
Fix $p\geq n, \epsilon, A\in\mathcal U^\epsilon_p$. We observe that $0\in \mathcal U^\epsilon_p$. Moreover the connection forms $A_t(x):=tA(tx)$ for $t\in[0,1]$ all belong to $\mathcal U^\epsilon_p$ as well, like in \cite{Uhl2}.
%
\subsubsection*{Proof of step 2}
%
Let $A_k\in \mathcal V^{\epsilon,C_p}_p$ be a sequence of connection forms converging in $W^{1,p}$ to $A$. Consider the gauges $g_k$ as in the definition \eqref{Vecp} of $\mathcal V^{\epsilon,C_p}_p$. We may assume that the  $A_k^{g_k}$ have a weak $W^{1,p}$-limit $\tilde A$. The bounds and equation in \eqref{Vecp} are preserved under weak limit thus we finish if we prove that there exists a gauge $g\in M_p$ such that $\tilde A=A^g$. Writing 
\begin{equation}\label{gauge-step2}
dg_k=g_k A_k^{g_k} - A_kg_k,
\end{equation}
as $G\subset\mathbb R^N$ is bounded it follows that $\|dg_k\|_{L^p}\lesssim \|A_k^{g_k}\|_{L^p}+\|A_k\|_{L^p}$, thus up to extracting a subsequence we have $g_k\stackrel{W^{1,p}}{\rightharpoonup}g$. Thus we may pass to the limit equation \eqref{gauge-step2} and we obtain indeed $\tilde A=A^g$ and also $g\in M_p$.
%
\subsubsection*{Proof of step 3}
%
Fix $p>n$ and let $A\in \mathcal V^{\epsilon,C_p}_p$. Consider the following data:
\begin{eqnarray*}
 g&\in&M_p\ ,\\
 \eta&\in&L^p(\mathbb S^n, \wedge^1T\mathbb S^n\otimes\mathfrak{g}),\\
V_p&:=&d^*_{\mathbb S^n}\left(\pi\left(L^p(\mathbb S^n, \wedge^1T\mathbb S^n\otimes\mathfrak g)\right)\right)\ .
\end{eqnarray*}
Consider the function of such $g,\eta$, with values in $V_p\subset W^{-1,p}(\mathbb S^n, \mathfrak g)$ defined as follows:
\[
 N_A(g,\eta):=d^*_{\mathbb S^n}\left(\pi\left(g^{-1}dg + g^{-1}(A+\eta)g\right)\right)=d^*_{\mathbb S^n}\left(g^{-1}dg + \pi\left(g^{-1}(A+\eta)g\right)\right),
\]
where we used the fact that $\pi=id$ on $\op{Im}(\pi)$.

\medskip

Note that $N_A(id,0)=0$ and $N_A$ is $C^1$. We want to apply the implicit function theorem in order to solve in $g$ the equation $N_A(g,\eta)=0$ for $g$ in a $W^{1,p}$-neighborhood of $id\in M_p$. The implicit function theorem will imply also that the dependence of $g$ on $\eta$ is continuous. Note that there holds $\op{exp}(tv)^{\pm 1}= 1\pm tv+O(t^2)$ as $t\to 0$. Using this and the fact that $E_p$ is the tangent space to $M_p$ at $id$ we find the linearization of $N_A$ at $(id,0)$ in the first variable: 
\begin{eqnarray*}
 H_A(v)&:=&\partial_gN_A(id,0)[v]\\
&=&\left.\frac{\partial}{\partial t}\right|_{t=0}\left[d^*_{\mathbb S^n}\left(\pi\left((\op{exp} (tv))^{-1}d\op{exp}( tv) + \op{exp}(tv)^{-1}(A+\eta)\op{exp}(tv)\right)\right)\right]\\
 &=& d^*_{\mathbb S^n}\left(dv +\pi([A,v])\right)\\
 &=& \Delta_{\mathbb S^n}v + d^*_{\mathbb S^n}\left(\pi([A,v])\right)\ .
\end{eqnarray*}
To prove invertibility of $H_A: E_p\to V_p$ we note that for all $f\in W^{-1,p}$, we have that $\Delta_{\mathbb S^n} v  = f$ has a unique solution $v\in W^{1,p}$ of zero average, thus $\Delta_{\mathbb S^n}:W^{1,p}\cap\{\int v=0\}\to W^{-1,p}$ is bijective, and moreover in this case $\|dv\|_{L^p}\leq C_p \|f\|_{W^{-1,p}}$ by elliptic estimates. In order to impose the extra constraints coming from $\pi$, note that if $f=d^*_{\mathbb S^n}\pi \alpha$ then we may rewrite the equation $\Delta_{\mathbb S^n}v= d^*_{\mathbb S^n}\pi \alpha$ in the weak form 
\begin{equation}\label{eq-step3}
\langle dv, d\varphi\rangle =\langle \pi\alpha, d\varphi\rangle, \quad \forall\varphi\in W^{1,\frac{p}{p-1}}(\mathbb S^n, \mathfrak g).
\end{equation}
If $d\varphi$ represents a functional on $L^p$ which vanishes on $\op{Im}\pi$, then \eqref{eq-step3} gives $\langle dv, d\varphi\rangle=0$. As $\op{Im}(d)\subset L^p$ is a closed subspace, this implies that $dv\in \op{Im}(\pi)$, thus $v\in E_p$. Therefore the restriction of the Laplacean $\Delta_{\mathbb S^n}|_{E_p}:E_p\to V_p$ is invertible with bounded inverse. To conclude that $H_A$ is invertible with bounded inverse as well, we show that for $A\in\mathcal V^{\epsilon,C_p}_p$ we have that 
\[
v\mapsto d^*_{\mathbb S^n}\pi([A,v]): E_p\to V_p
\]
is a bounded linear operator of small operator norm, where the norms we consider are the ones induced by $W^{1,p}$ on $E_p$ and by $W^{-1,p}$ on $V_p$, respectively. Indeed for $v\in E_p$ and for all test functions $\varphi\in W^{1,p'}$ we have
\begin{eqnarray*}
|\langle d^*_{\mathbb S^n}\pi([A,v]), \varphi\rangle|&=&|\langle \pi([A,v]), d\varphi\rangle|\le\|d\varphi\|_{L^{p'}}\|\pi([A,v])\|_{L^p}\\[3mm]
&\le&C_{\pi,p}\|[A,v]\|_{L^p}\|\varphi\|_{W^{1,p'}}\le C_{\pi,p}\|A\|_{L^n}\|v\|_{L^{p^*}}\|\varphi\|_{W^{1,p'}}\ ,
\end{eqnarray*}
which implies the desired estimate. If $\|A\|_{L^n}$ is small enough (depending only on $p,\pi$), we thus obtain that $H_A$ is invertible. In particular this condition holds for all $A\in\mathcal V^{\epsilon_0,C_p}_p$ for $C_p, \epsilon_0$ small enough.
%
\subsubsection*{Proof of step 4}
%
For $g\in M_p$ and using $\pi\circ \pi=\pi$, we have $\pi(g^{-1}dg)=g^{-1}dg$,
thus $d_{\mathbb S^n}^*((id-\pi)A^g)=d_{\mathbb S^n}^*(g^{-1}Ag)$. Given this and \eqref{newproppi1}, we find 
\begin{equation}\label{bd1}
\|A^g\|_{L^n}\le \|\pi A^g\|_{L^n} + \|(id-\pi)(g^{-1}Ag)\|_{L^n}\le \|\pi A^g\|_{L^n} + C_\pi\|A\|_{L^2}.
\end{equation}
Then using \eqref{newproppi2} as well as the property $d_{\mathbb S^n}^*(\pi A^g)=0$, by Hodge and Sobolev inequalities we have
\begin{equation}\label{bd2}
\|\pi A^g\|_{L^n}\le \|\nabla \pi A^g\|_{L^{n/2}}\le \|d\pi A^g\|_{L^{n/2}} + \|d_{\mathbb S^n}^*(\pi A^g)\|_{L^{n/2}}=\|dA^g\|_{L^{n/2}}.
\end{equation}
Now just use the equation $F= dA+A\wedge A$ noting that pointwise a.e. there holds $|F_{A^g}|=|F_A|$, and then by triangle and H\"older inequalities, via \eqref{bd1} and \eqref{bd2} we conclude:
\begin{eqnarray*}
\|A^g\|_{L^n} &\le&\|\pi A^g\|_{L^n} + C_\pi\|A\|_{L^2}\le  C_n\|dA^g\|_{L^{n/2}}+C_\pi\|A\|_{L^2}\\[3mm]
 &\le& C_n\|F\|_{L^{n/2}}+C_n\|A^g\|_{L^n}^2+C_\pi\|A\|_{L^2} \ .
\end{eqnarray*}
If $\|A^g\|_{L^n}\leq K$ small enough then the second term above is estimated by $KC_n\|A^g\|_{L^n}$ which can then be absorbed to the left side of the inequality, giving the desired estimate.
%
\subsubsection*{Proof of step 5}
%
We may approximate $A\in \mathcal U^{\epsilon_0}_n$ by smooth $A_k$ in $L^n$ norm such that $F_{A_k}\to F_A$ in $L^{n/2}$ as well. In particular there holds $A_k\in L^p$ for all $p>n$. We may obtain that $A_k\in \mathcal U^{\epsilon_0}_p= \mathcal V^{\epsilon_0,C_p}_p, p>n$ and in particular we find $g_k\in M_p$ such that
\[
\|A_k^{g_k}\|_{L^n}\lesssim\|A_k\|_{L^n}\lesssim \|F_k\|_{L^{2/n}}\lesssim\epsilon_0\ ,
\]
where the constants depend only on the exponents $p$ and $n$. By possibly diminishing $\epsilon_0$ we thus achieve $\|A_k^{g_k}\|_{L^n}\leq K$ for all $k$. By the closure result of Step 2 for $p=n$ we thus obtain that the same estimate holds for $A$ and for some gauge $g\in M_n$ and by Step 4 we conclude that $A\in \mathcal V^{\epsilon_0,K}_p$, as desired. $\square$
%
\section{Replacement of nonabelian curvatures on Lipschitz domains in $n$ dimensions}\label{sec:ext1gc}
%
In this section we prove the extension result which will help to define our approximating connections.

\medskip

We consider the scale $r=1$. Theorem \ref{thm:extension} will later be used on all faces of skeleta of a cubeulation, even on the ones of higher codimension, but by a slight abuse of notation we denote still by ``$n$'' the dimension of our domains. The ambient dimension will be called $N$.

\medskip

Theorem \ref{thm:extension} has the role of translating the $L^n$-smallness condition \eqref{Fgood} on $F$ on the $n$-dimensional boundary to a similar smallness (with a slightly worse, but still small, constant) \eqref{Fbetter} in the $n+1$ dimensional interior, which will allow the theorem to be used iteratively. At each step of the iteration we use Theorem \ref{coulstrange}, which requests $A$ to already be $L^n$-integrable. For $n=4$ this is a consequence of Defnition \ref{def:weakconn}, which describes the class of weak connections. In higher dimension $n>4$ we achieve it by iteratve extension on the skeleta of dimension $5\le k\le n$, as a consequence of Theorem \ref{thm:extension} itself. The qualitative improvement \eqref{FbetterA} in $A$ is crucial for applying Theorem \ref{coulstrange} at the next step. 
\begin{theorem}\label{thm:extension}
Let $D\subset \mathbb R^N$ be homeomorphic to the $n+1$ dimensional ball $\mathbb B^{n+1}$ via map $\Psi_D: \mathbb R^{n+1}\supset\mathbb B^{n+1}_1\to D$. Assume that $\Psi_D$ is also bi-Lipschitz with bi-Lipschitz constant $C_D$.
 
\medskip

Let $F\in L^2(D,\wedge^2\mathbb R^N\otimes\mathfrak g)$ and $A\in L^2(D,\wedge^1\mathbb R^N\otimes\mathfrak g)$ (where the underlying measure is the Hausdorff $(n+1)$-dimensional measure on $D$) be such that in the sense of distributions $i^*_DF=di^*_DA+i^*_DA\wedge i^*_DA$ on $D$. Fix also constant forms $\overline{F}\in\wedge^2\mathbb R^N\otimes\mathfrak g$ and $\overline{A}\in\wedge^1\mathbb R^N\otimes\mathfrak g$. 
 
\medskip

Assume that 
\begin{equation}\label{Agood}
\exists\, g\in W^{1,2}(\partial D,G)\quad\text{ such that }\quad i^*_{\partial D}A^g\in L^n(\partial D, \wedge^1 T D\otimes \mathfrak g).
\end{equation}
Then there exists a constant $\epsilon_0>0$ depenging only on $n,C_D,G$ such that if in the above gauge there holds
\begin{equation}\label{Fgood}
 \int_{\partial D}|i^*_{\partial D}F|^{n/2}< \epsilon_0,\quad  \int_{\partial D}|i^*_{\partial D}A^g|^n< \epsilon_0,\quad\verti{\overline{F}}<\epsilon_0,\quad|\overline{A}|^2<\epsilon_0,
\end{equation}
then there exists $\hat A\in L^2(D,\wedge^1TD\otimes\mathfrak g)$ and $\hat g\in W^{1,2}(D,G)$ such that $\hat A^{\hat g}$ is Lipschitz in the interior of $D$, $i^*_{\partial D}\hat A=i^*_{\partial D}A$, $\hat g|_{\partial D}=g$ and the distribution $F_{\hat A}=d\hat A+\hat A\wedge \hat A$ is represented by an element of $L^2(D,\wedge^2TD\otimes\mathfrak g)$, and we have the approximation bounds
\begin{subequations}\label{bounds_FA}
\begin{multline}\label{AA}
 \|d\hat A + \hat A\wedge \hat A - \overline{F}\|_{L^2(D)}\lesssim  \|i_{\partial D}^\ast (F-\overline{F})\|_{L^2(\partial D)}\\[3mm]
+\ \verti{\overline{F}}\|i_{\partial D}^\ast (A - \overline{A})\|_{L^2(\partial D)} + \|i^*_{\partial D}A^g\|_{L^n(\partial D)}^2 +\ \|i^*_{\partial D}F\|_{L^{n/2}(\partial D)}^2 + \verti{\overline{F}}^2
\end{multline}
and
\begin{equation}\label{AAA}
 \|\hat A - \overline{A}\|_{L^2(D)}\lesssim\|i_{\partial D}^\ast (A -\overline{A})\|_{L^2(\partial D)}\ .
\end{equation}
Moreover we have the improved regularity
\begin{equation}\label{Fbetter}
\|d\hat A + \hat A\wedge \hat A \|_{L^{\frac{n+1}{2}}(D)}\lesssim\ \|i^*_{\partial D}A^g\|_{L^n(\partial D)}+\|i^*_{\partial D}F\|_{L^{n/2}(\partial D)} +\ |\overline{F}|
\end{equation}
and
\begin{equation}\label{FbetterA}
\hat A^{\hat g}\in L^{n+1}(D, \wedge^1 TD\otimes \mathfrak g).
\end{equation}
\end{subequations}
\end{theorem}
\begin{proof}
\textbf{Preparation.}\textit{ Transfer of the informations to the sphere.} We do all our estimates on $\mathbb S^n, \mathbb B^{n+1}$, and we use the bi-lipschitz homeomorphism $\Psi_D$ to link this model case with the case of general $D$. Indeed we note that for measurable $L^p$-integrable $k$-forms $\omega$ on $D$ and $\omega'$ on $\partial D$ there holds, with $C_\phi$ denoting here the Lipschitz constant of a function $\phi$,
\begin{eqnarray}\label{psipresnorm}
C_{\Psi_D^{-1}}^{-k}C_{\Psi_D}^{-\frac{n}{p}}\|\omega'\|_{L^p(\partial D)}\le&\|\Psi_D^*\omega'\|_{L^p(\mathbb S^n)}&\le C_{\Psi_D}^kC_{\Psi_D^{-1}}^{\frac{n}{p}}\|\omega'\|_{L^p(\partial D)},\nonumber\\
C_{\Psi_D^{-1}}^{-k}C_{\Psi_D}^{-\frac{n+1}{p}}\|\omega\|_{L^p(D)}\le&\|\Psi_D^*\omega\|_{L^p(\mathbb B^{n+1})}&\le C_{\Psi_D}^kC_{\Psi_D^{-1}}^{\frac{n+1}{p}}\|\omega\|_{L^p(D)}.
\end{eqnarray}
Thus the bounds obtained on $\mathbb S^n$ and $\mathbb B^{n+1}$ which we obtain below are comparable to those on general $D$.

\medskip 

Let $\pi_{nc,D}:L^p(\partial D, \wedge^1T\partial D\otimes\mathfrak{g})\to L^p(\partial D, \wedge^1T\partial D\otimes\mathfrak{g})$ be define by
\[
\pi_{nc,D}(A) = A - \sum_{j=1}^N \left(\int_{\partial D} \langle A, i^*_{\partial D}dx_j\rangle d\mathcal H^n\right)i^*_{\partial D}dx_j\ .
\]
In other words, $\pi_{nc,D}$ removes from $A$ to the $L^2$-components parallel to $\op{Span}\{i^*_{\partial D}dx_k, k=1,\ldots,N\}$. In particular for any $\overline{A}$ as in the thesis of our theorem, then we have
\begin{eqnarray}\label{comparepibara}
\|\pi_{nc,D}(i^*_{\partial D}A)\|_{L^2(\partial D)}&=&\inf\{\|i^*_{\partial D}(A - B)\|_{L^2(\partial D)}: B\in\wedge^1\mathbb R^N\otimes\mathfrak g\}\nonumber\\[3mm]
 &\le& \|i^*_{\partial D}(A - \overline{A})\|_{L^2(\partial D)}.
\end{eqnarray}
We are now going to define an involution $\pi_D:L^p(\mathbb S^n, \wedge^1T\mathbb S^n\otimes\mathfrak{g})\to L^p(\mathbb S^n, \wedge^1T\mathbb S^n\otimes\mathfrak{g})$ like in Theorem \ref{coulstrange} and related to $\pi_{nc, D}$ by the following property:
\begin{equation}\label{compatpi}
\pi_D(\Psi_D^*i_{\partial D}^*A)=\Psi^*_D\left(\pi_{nc,D}(i^*_{\partial D}A)\right) \quad\text{ for all }\quad A\in L^p(\mathbb R^N, \wedge^1\mathbb R^N\otimes\mathfrak{g}).
\end{equation}
This involution can be written explicitly:
\begin{equation}\label{explicitpi}
\pi_D(A)= A - \sum_{j=1}^N\left(\int_{\mathbb S^n} \langle A, \Psi_D^*i^*_{\partial D}dx_j\rangle d\mathcal H^n\right)\Psi_D^*i^*_{\partial D}dx_j\ .
\end{equation}
This clearly implies \eqref{newproppi2}. Then \eqref{newproppi1} holds with $\pi_D$ in the place of $\pi$ because $\Psi_D\circ i_{\partial D}$ is Lipschitz and $dx_j$ is smooth. In fact, as $\Psi_D$ is Lipschitz, from \eqref{explicitpi} we also obtain for $1\le p\le\infty$ that
\begin{equation}\label{bound1-pi}
\|(1-\pi_D)(i_{\mathbb S^n}^*\Psi_D^*A)\|_{L^\infty(\mathbb S^n)}\le C_{N,p} \|i_{\mathbb S^n}^*\Psi_D^*A\|_{L^p(\mathbb S^n)}\le C_{N,p}C_{\Psi_D}\|i_{\partial D}^*A\|_{L^p(\partial D)}\ .
\end{equation} 
%
\textbf{Step 1.}\textit{ Coulomb gauge on the boundary.} Denote $\Psi_D^*A:=A_D$. Let $g$ be the change of gauge $g$ given by Theorem \ref{coulstrange} applied with $A$ replaced by $i^*_{\mathbb S^n}\Psi^*_DA=i^*_{\mathbb S^n}A_D$. Then $g$ satisfies
\begin{subequations}\label{A-1}
\begin{align}
&\pi_D(g^{-1}dg)=g^{-1}dg,\label{A-1-1}\\[3mm]
& d^\ast_{\mathbb S^n}\pi_D(i^*_{\mathbb S^n}A_D^g)=0\label{A-1-2}
\end{align}
and
\begin{equation}\label{A-1-3}
\|i^*_{\mathbb S^n}A_D^g\|_{L^n(\mathbb S^n)}\le C_n\left(\|i^*_{\mathbb S^n}F_{A_D}\|_{L^{n/2}(\mathbb S^n)} +\|i^*_{\mathbb S^n}A_D\|_{L^2(\mathbb S^n)}\right).
\end{equation}
\end{subequations}
From $i^*_{\mathbb S^n}A_D^g=g^{-1}dg+g^{-1}i^*_{\mathbb S^n}A_Dg$ we obtain using the $\pi_D$-Coulomb condition on $A^g$ (where we identify $1$-forms and vector fields using the metric)
\begin{eqnarray*}
 \Delta_{\mathbb S^n}g&=&dg\cdot i^*_{\mathbb S^n}A_D^g +(g-id)\,d^*_{\mathbb S^n}i^*_{\mathbb S^n}A_D^g- d^*_{\mathbb S^n}(\pi_D(i^*_{\mathbb S^n}A_D)g)\\[3mm]
 && - d^*_{\mathbb S^n}[(1-\pi_D)A_D\,(g-id)]+ d^*_{\mathbb S^n}[(1-\pi_D)(i_{\mathbb S^n}^*(A_D^g-A_D))]\ .
\end{eqnarray*}
The last term can be simplified using \eqref{A-1-1}:
\begin{equation}\label{simplifyg}
d^*_{\mathbb S^n}[(1-\pi_D)(i_{\mathbb S^n}^*(A_D^g-A_D))]=d^*_{\mathbb S^n}[(1-\pi_D)(i_{\mathbb S^n}^*(g^{-1}A_D g-A_D))]\ .
\end{equation}
If $\bar g$ is the average of $g$ on $\mathbb S^n$ taken in $\mathbb R^{n+1}$, then by the mean value formula there exists $x\in\mathbb S^n$ such that $|g(x) - \bar g|\leq C\|g-\bar g\|_{L^2}$. Up to changing $g$ to $gg_0$ where $g_0$ is a constant rotation, we may also assume $g(x)=id$. Using the embedding $W^{1,2}\to L^{\frac{2n}{n-2}}$ valid in $n$ dimensions, \eqref{simplifyg}, \eqref{bound1-pi} and the H\"older inequalities we deduce (estimating the $W^{-1,2}$-norm by duality with $W^{1,2}$):
\begin{align}
 \|dg\|_{L^2(\mathbb S^n)}&\lesssim \|\Delta_{S^n}g\|_{W^{-1,2}(\mathbb S^n)} \lesssim\|dg\|_{L^2(\mathbb S^n)}\|i^*_{\mathbb S^n}A_D^g\|_{L^n(\mathbb S^n)}\nonumber\\[3mm]
 & + \|g-id\|_{L^{\frac{2n}{n-2}}(\mathbb S^n)}\|i^*_{\mathbb S^n}A_D^g\|_{L^n(\mathbb S^n)}+ \|\pi_D(i^*_{\mathbb S^n}A_D)\|_{L^2(\mathbb S^n)} \nonumber\\[3mm]
 &+\|g-id\|_{L^2(\mathbb S^n)}\|(1-\pi_D)(i^*_{\mathbb S^n}A_D)\|_{L^\infty(\mathbb S^n)}\ .\label{dgbound}
\end{align}
By using Sobolev inequality $ \|g-id\|_{L^{\frac{2n}{n-2}}(\mathbb S^n)}\lesssim \|dg\|_{L^2(\mathbb S^n)}$, the boundedness $|g(x)-id|\le \op{diam}(G)$, \eqref{A-1-1} and $\|i^*_{\partial D}\overline{A}\|_{L^\infty(\partial D)}\lesssim|\overline{A}|<\epsilon_0$, we absorb the terms not containing $\pi_D(A_D)$ from the right hand side of \eqref{dgbound} to the left hand side. Using \eqref{bound1-pi}, for $\epsilon_0>0$ small enough we reabsorb also the $(1-\pi_D)$-term. By \eqref{comparepibara}, \eqref{compatpi} we obtain
\begin{equation}\label{A-2}
\|dg\|_{L^2(\mathbb S^n)}\leq C\|\pi_D(i_{\mathbb S^n}^*A_D)\|_{L^2(\mathbb S^n)}\le C \|i^*_{\partial D}(A-\overline{A})\|_{L^2(\partial D)}\ .
\end{equation}
%
\textbf{Step 2.}\textit{ Estimates on $i_{\mathbb S^n}^*\Psi_D*\overline{F}$.} By \eqref{A-2}, Poincar\`e's inequality and the Lipschitz bounds \eqref{psipresnorm}, we obtain
\begin{multline}\label{f-1}
\|g^{-1}(i_{\mathbb S^n}^\ast\Psi_D^*\overline{F}) g-i_{\mathbb S^n}^\ast\Psi_D^*\overline{F}\|_{L^2(\mathbb S^n)}\lesssim \|i_{\mathbb S^n}^\ast\Psi_D^*\overline{F}\|_{L^{\infty}(\mathbb S^n)} \|g-id\|_{L^2(\mathbb S^n)}\\[3mm]
\lesssim\|i_{\partial D}^\ast\overline F\|_{L^\infty(\partial D)}\|i_{\partial D}^\ast (A - \overline{A})\|_{L^2(\partial D)}\ .
\end{multline}
Since equation $F_{A^g}=g^{-1}\, F\, g$ is invariant under pullback and the norm on $\mathfrak{g}$-valued forms is invariant under conjugation, using then \eqref{A-1} and the triangular inequality we obtain
\begin{equation}\label{A-7}
\begin{array}{rcl}
\|i_{\mathbb S^n}^\ast (F_{A_D^g}-\Psi_D^*\overline{F})\|_{L^2(\mathbb S^n)} &\lesssim&  \|i_{\partial D}^\ast\overline F\|_{L^\infty(\partial D)}\|i_{\partial D}^\ast (A - \overline{A})\|_{L^2(\partial D)}\\[3mm]
&&+\|i_{\partial D}^\ast (F-\overline{F})\|_{L^2(\partial D)}\ .
\end{array}\
\end{equation}
Using \eqref{A-1-3} and the bounded embedding $L^4\to L^n$ for $n\ge 4$ to bound $\vertii{A_D^g\wedge A_D^g}_{L^2}\le \vertii{A^g_D}_{L^4}^2$, from \eqref{A-7} and \eqref{psipresnorm} we deduce
\begin{equation}\label{A-9}
\begin{array}{rcl}
\|i_{\mathbb S^n}^\ast (dA_D^g-\Psi_D^*\overline{F})\|_{L^2(\mathbb S^n)} &\lesssim&\|i_{\partial D}^\ast\overline F\|_{L^\infty(\partial D)}\|i_{\partial D}^\ast (A - \overline{A})\|_{L^2(\partial D)}\\[3mm]
&&+\|i_{\partial D}^\ast (F-\overline{F})\|_{L^2(\partial D)}+\|i^*_{\mathbb S^n}F\|_{L^{n/2}(\mathbb S^n)}^2 \\
&\lesssim&\verti{\overline F}\|i_{\partial D}^\ast (A - \overline{A})\|_{L^2(\partial D)}\\[3mm]
&&+\|i_{\partial D}^\ast (F-\overline{F})\|_{L^2(\partial D)}+\|i^*_{\partial D}F\|_{L^{n/2}(\partial D)}^2\ .
\end{array}
\end{equation}
%
\textbf{Step 3.} \textit{Extension to the interior.} Define the following form belonging to $C^\infty(\mathbb R^N, \wedge^1\mathbb R^N\otimes\mathfrak{g})$:
\begin{equation}
\label{A-14}
B:=\sum_{i<j}\overline{F_{ij}}\ \frac{x_i\,dx_j-x_j\, dx_i}{2}\ .
\end{equation}
For an 1-form $\eta\in L^n(\mathbb S^n, \wedge^1T\mathbb S^n\otimes \mathfrak g)$ we then consider the minimization problem
\begin{equation}\label{A-10}
\min\left\{\begin{array}{c}\int_{\mathbb B^{n+1}}|d(C-\Psi_D^*B)|^2+|d^{\ast_{{\mathbb R}^{n+1}}}(C-\Psi^*_DB)|^2\\[3mm]
\text{s. t. } C\in W^{\frac{1}{n},n}(\mathbb B^{n+1}, \wedge^1\mathbb B^{n+1}\otimes \mathfrak g),\ i_{\mathbb S^n}^\ast C=\eta
\end{array}\right\}\ .
\end{equation}
A classical argument (and the fact that $d\Psi_D^*B=\Psi^*_DdB=\Psi_D^*\overline{F}$) shows that the solution $\tilde\eta$ to \eqref{A-10} is uniquely given by
\begin{equation}\label{eta}
\left\{
\begin{array}{l}
 d^{\ast_{{\mathbb R}^{n+1}}}(\widetilde{\eta}-\Psi^*_D B)=0\quad\quad\text{ in } \mathbb B^{n+1}\ ,\\[3mm]
 d^{\ast_{{\mathbb R}^{n+1}}}(d\widetilde{\eta} - \Psi_D^*\overline{F})=0\quad\quad\text{ in } \mathbb B^{n+1}\ ,\\[3mm]
 i_{\mathbb S^n}^\ast \widetilde{\eta}=\eta\quad\quad\text{ on } \partial \mathbb B^{n+1}\ ,
\end{array}
\right.
\end{equation}
and one has
\begin{equation}\label{A-13}
\|\widetilde{\eta}-\Psi_D^*B\|_{L^{n+1}(\mathbb B^{n+1})}\le C\,\|\widetilde{\eta}-\Psi_D^*B\|_{W^{\frac{1}{n},n}(\mathbb B^{n+1})}\le C\ \|\eta-i_{\mathbb S^n}^*\Psi_D^*B\|_{L^n(\mathbb S^n)}\ .
\end{equation}
Note that $\Psi_D^*B$ is the solution to \eqref{A-10} for the choice $\eta=i_{\mathbb S^n}^\ast\Psi_D^* B$. If we choose $\eta=\pi_D(i^*_{\mathbb S^n}A_D^g)$ in \eqref{A-10} then we find the extension $\widetilde{\pi_D(i^*_{\mathbb S^n}A_D^g)}$. Next, for $A\in L^1(D,\wedge^1TD\otimes\mathfrak g)$ we denote
\begin{equation}\label{ovelinea}
 \overline{A}:=\sum_{k=1}^{n+1}\Psi_D^*dx_k \frac{1}{C_n}\int_{\mathbb S^n}\langle i^*_{\mathbb S^n}\Psi_D^*A,\Psi_D^*i^*_{\mathbb S^n}dx_k\rangle,
\end{equation}
where $C_n$ is a normalization constant such that $i^*_{\mathbb S^n}\overline{i^*_{\mathbb S^n}\Psi_D^*\alpha}=i^*_{\mathbb S^n}\Psi_D^*\alpha$ for any constant form $\alpha$. With notation \eqref{ovelinea} we define
\[
\widetilde{A^g_D}:=\widetilde{\pi(i^*_{\mathbb S^n}A_D^g)}+\ \overline{A^g}\ .
\]
%
\textbf{Step 4.} \textit{Estimates on the extended curvatures.} Note that $d \Psi_D^*dx_k=\Psi_D^*d^2x_k=0$, therefore by \eqref{ovelinea} we have $d\overline{A^g}=0$. Similarly we find $d(1-\pi_D)(\alpha)=0$ for general $1$-forms $\alpha$ on $\mathbb S^n$. Using this, analogously to \eqref{A-13} we obtain
\begin{eqnarray}\label{A-18}
\|d\widetilde{A_D^g}-\Psi_D^*\overline{F}\|^2_{L^2(\mathbb B^{n+1})}&=&\|d\widetilde{\pi_D(i^*_{\mathbb S^n}A_D^g)}-\Psi_D^*\overline{F}\|_{L^2(\mathbb B^{n+1})}\nonumber\\[3mm]
&\lesssim&\|d\pi_D(i^*_{\mathbb S^n}A_D^g)-i_{\mathbb S^n}^\ast \Psi_D^*\overline F\|_{L^2(\mathbb S^n)} \nonumber\\[3mm]
&=&\|i^*_{\mathbb S^n}(d A_D^g-\Psi_D^*\overline F)\|_{L^2(\mathbb S^n)}\ .
\end{eqnarray}
By \eqref{A-13} and the triangle inequality we obtain
\begin{multline}\label{A-19}
\|\widetilde{A_D^g}\wedge\widetilde{A_D^g}\|^{1/2}_{L^{\frac{n+1}{2}}(\mathbb B^{n+1})}\lesssim\|\widetilde{A_D^g}\|_{L^{n+1}(\mathbb B^{n+1})}\\[3mm]
\lesssim\|\widetilde{\pi_D(i^*_{\mathbb S^n}A_D^g)}- \Psi_D^*B\|_{L^{n+1}(B^{n+1})} + \|\Psi^*_DB\|_{L^{n+1}(\mathbb B^{n+1})} +\|\overline{A^g}\|_{L^{n+1}(\mathbb B^{n+1})}\\[3mm]
\lesssim\|\pi_D(i^*_{\mathbb S^n}A_D^g) - i^*_{\mathbb S^n}\Psi_D^*B\|_{L^n(\mathbb S^n)}+ \|\Psi^*_DB\|_{L^{n+1}(\mathbb B^{n+1})} +\|\overline{A^g}\|_{L^{n+1}(\mathbb B^{n+1})}\\[3mm]
\lesssim \|i^*_{\mathbb S^n}A_D^g\|_{L^n(\mathbb S^n)} + \|(1-\pi_D)(i^*_{\mathbb S^n}A_D^g)\|_{L^n(\mathbb S^n)}  +\|\overline{A^g}\|_{L^{n+1}(\mathbb B^{n+1})}\\[3mm]
+ \|i^*_{\mathbb S^n}\Psi_D^*B\|_{L^n(\mathbb S^n)}+ \|\Psi^*_DB\|_{L^{n+1}(\mathbb B^{n+1})}\ .
\end{multline}
Similarly we find the two lines in the following estimate, while the others are deduced by triangle inequality and by the gauge invariance of $F$'s norm:
\begin{eqnarray}\label{A-20}
\|d\widetilde{A_D^g}\|_{L^{\frac{n+1}{2}}(\mathbb B^{n+1})}&=&\|d\widetilde{\pi(i_{\mathbb S^n}^*A_D^g)}\|_{L^{\frac{n+1}{2}}(\mathbb B^{n+1})}\nonumber\\[3mm]
&\lesssim&\|i_{\mathbb S^n}^*d A_D^g\|_{L^{\frac{n}{2}}(\mathbb S^n)}+\|\Psi^*_DB\|_{L^{\frac{n+1}{2}}(\mathbb B^{n+1})}\nonumber\\[3mm]
&\lesssim&\|i_{\partial D}^*F_{A_D}\|_{L^{\frac{n}{2}}(\mathbb S^n)} + \|i^*_{\mathbb S^n}A_D^g\|_{L^n(\mathbb S^n)}^2+ \|\Psi^*_DB\|_{L^{\frac{n+1}{2}}(\mathbb B^{n+1})}.\quad\quad
\end{eqnarray}
To estimate the $\Psi_D^*B$-terms in \eqref{A-19} and \eqref{A-20} we use \eqref{A-14}, obtaining
\[
\vertii{i^*_{\mathbb S^n}\Psi_D^*}_{L^n(\mathbb S^n)} + \vertii{\Psi_D^*B}_{L^{n+1}(\mathbb B^{n+1})}\lesssim \verti{\overline{F}}\quad \mbox{and}\quad \vertii{\Psi^*_DB}_{L^{\frac{n+1}{2}}(\mathbb B^{n+1})}\lesssim \verti{\overline{F}}.
\]
By \eqref{bound1-pi} with $p=n$ combined with \eqref{A-1} and \eqref{psipresnorm}, we find
\[
\|(1-\pi_D)(i^*_{\mathbb S^n}A_D^g)\|_{L^n(\mathbb S^n)}+\|\overline{A^g}\|_{L^{n+1}(\mathbb B^{n+1})}\lesssim\vertii{i^*_{\partial D}F}_{L^{\frac{n}{2}}(\partial D)}+\vertii{i^*_{\partial D}A^g}_{L^n(\partial D)},
\]
whereas the remaining terms appear directly in \eqref{A-1}. Thus from \eqref{A-19} and \eqref{A-20}, respectively, we obtain the following two bounds:
\begin{eqnarray}
\vertii{\widetilde{A^g}\wedge\widetilde{A^g}}_{L^{\frac{n+1}{2}}(\mathbb B^{n+1})} &\lesssim&\vertii{i^*_{\partial D} A^g}_{L^n(\partial D)}^2 + \vertii{i^*_{\partial D}F}_{L^{\frac{n}{2}}(\partial D)}^2 + \verti{\overline F}^2,\label{A-19-1}\\
\vertii{d\widetilde{A^g_D}}_{L^{\frac{n+1}{2}}(\mathbb B^{n+1})}&\lesssim&\vertii{i^*_{\partial D} A^g}_{L^n(\partial D)} + \vertii{i^*_{\partial D}F}_{L^{\frac{n}{2}}(\partial D)} + \verti{\overline F},\label{A-20-1}
\end{eqnarray}
Summing \eqref{A-9}, \eqref{A-18} and \eqref{A-19-1} we have
\begin{multline}\label{A-21}
\|d\widetilde{A^g_D}+\widetilde{A^g_D}\wedge\widetilde{A^g_D}-\Psi_D^*\overline{F}\|^2_{L^2(\mathbb B^{n+1})}\lesssim |\overline F|\|i_{\partial D}^\ast (A - \overline{A})\|_{L^2(\partial D)} \\[3mm]
+\  \|i_{\partial D}^\ast (F-\overline{F})\|_{L^2(\partial D)}+ \|i^*_{\partial D}A^g\|_{L^n(\partial D)}^2 +\ \|i^*_{\partial D}F\|_{L^{n/2}(\partial D)}^2 + |\overline{F}|^2\ ,
\end{multline}
which will help to prove \eqref{AA} in a later step. Next, we perform the analogous boud where instead of \eqref{A-9} we use \eqref{A-20-1}, and we obtain
\begin{eqnarray}\label{A-21-1}
\|d\widetilde{A^g_D}+\widetilde{A^g_D}\wedge\widetilde{A^g_D}\|_{L^{\frac{n+1}{2}}(\mathbb B^{n+1})}&\lesssim&\|i^*_{\partial D}A^g\|_{L^n(\partial D)}+\|i^*_{\partial D}F\|_{L^{n/2}(\partial D)} +\ |\overline{F}|\nonumber \\
&&+\|i^*_{\partial D}A^g\|_{L^n(\partial D)}^2+\|i^*_{\partial D}F\|_{L^{n/2}(\partial D)}^2 +\ |\overline{F}|^2\nonumber\\
&\lesssim&\|i^*_{\partial D}A^g\|_{L^n(\partial D)}+\|i^*_{\partial D}F\|_{L^{n/2}(\partial D)}+\verti{\overline{F}},\quad\quad
\end{eqnarray}
where due to hypothesis \eqref{Fgood}, we were able to absorb the second line into the first.

\medskip

\textbf{Step 5.} \textit{Correcting the restriction on the boundary.} Extend now $g$ harmonically in $\mathbb B^{n+1}$ and denote by $\hat{g}$ this extension. Using \eqref{A-2} together with classical elliptic estimates, we find
\begin{multline}\label{A-22}
\|\hat{g}^{-1}(\Psi_D^*\overline{F})\hat{g}-\Psi_D^*\overline{F}\|_{L^2(\mathbb B^{n+1})}\\[3mm]
\lesssim\ |\overline{F}|\ \|\hat{g}-id\|_{L^2(\mathbb B^{n+1})}\, \lesssim\, |\overline{F}|\ \|i_{\partial D}^\ast (A -\overline{A})\|_{L^2(\partial D)}\ .
\end{multline}
Combining \eqref{A-21} and \eqref{A-22} gives
\begin{eqnarray}\label{A21}
\lefteqn{\|d\widetilde{A^g_D}+\widetilde{A_D^g}\wedge\widetilde{A_D^g}-\hat{g}^{-1}(\Psi_D^*\overline{F})\hat{g}\|_{L^2(\mathbb B^{n+1})} \lesssim  \|i_{\partial D}^\ast (F-\overline{F})\|_{L^2(\partial D)}}\nonumber\\[3mm]
&&+\ |\overline F|\|i_{\partial D}^\ast (A - \overline{A})\|_{L^2(\partial D)} + \|i^*_{\partial D}A^g\|_{L^n(\partial D)}^2 +\ \|i^*_{\partial D}F\|_{L^{n/2}(\partial D)}^2 + |\overline{F}|^2.\quad\quad
\end{eqnarray}
Denote
\begin{equation}\label{hatadef}
\hat{A}_D:=(\widetilde{A_D^g})_{\hat g^{-1}}.
\end{equation}
With this notation one has $F_{\hat{A}_D}=\hat{g}\, F_{A^g_D}\, \hat{g}^{-1}$, and after applying \eqref{psipresnorm} to pass from $\mathbb B^{n+1}$ back to $D$, we see that \eqref{A21} implies the estimate \eqref{AA}. By the same token, estimate \eqref{Fbetter} follows from \eqref{A-21-1}. Moreover we have 
\begin{equation}\label{hatg}
\hat A^{\hat g}_D=\tilde A^g_D
\end{equation}
in the notations \eqref{hatadef} and $\hat A^{\hat g}$ harmonic, and thus is smooth in the interior of $\mathbb B^{n+1}$. Note that
\begin{equation}\label{restrad}
i^*_{\mathbb S^n}\hat A_D=i^*_{\mathbb S^n}(\tilde A_D^g)_{\hat g^{-1}}=(i^*_{\mathbb S^n}\tilde A_D^g)_{\hat g^{-1}}=i^*_{\mathbb S^n}A_D\ .  
\end{equation}
Define 
\begin{equation}\label{hata}
\hat A:=(\Psi_D^{-1})^*\hat A_D.
\end{equation}
We observe that since $\hat A_D$ has $L^{n+1}$ bounds \eqref{A-19}, by the bound on $\overline{A^g}$, and by the Lipschitz bounds \eqref{psipresnorm}, it follows that $\hat A$ has $L^{n+1}$ bounds as well, as requested in \eqref{FbetterA}.

\medskip

By \eqref{psipresnorm} we also obtain that the distributional expression $F_{\hat A}=d\hat A+\hat A\wedge \hat A$ is well defined and $F_{\hat A}\in L^2$.

\medskip

\textbf{Step 6.} \textit{Verification of \eqref{AAA}.} We now use the definition \eqref{hatadef} of $\hat A_D$, as well as the estimates $\|d\hat g\|_{L^2(\mathbb B^{n+1})}\lesssim\|dg\|_{L^2(\mathbb S^n)}$ together with \eqref{A-2}, and then the bounds \eqref{A-1} on $i^*_{\mathbb S^n}A^g$ from Theorem \ref{coulstrange}. Note also that if $g\in G$ and $M\in\mathfrak{g}$ then $\verti{g^{-1}Mg - M}\lesssim\verti{g-id}\verti{M}$. We thus obtain the following chain of inequalities:
\begin{eqnarray}\label{verifaaa1}
\lefteqn{\|\hat A_D - \Psi_D^*\overline{A}\|_{L^2(\mathbb B^{n+1})}\lesssim\|d\hat g\|_{L^2(\mathbb B^{n+1})} + \vertii{\verti{\hat g -id}\verti{\widetilde{A_D^g} -\Psi_D^*\overline{A}}}_{L^2(\mathbb B^{n+1})}}\nonumber\\
&&+\vertii{\verti{\hat g - id}\verti{\Psi_D^*\overline{A}}}_{L^2(\mathbb B^{n+1})}+\vertii{\widetilde{A_D^g}-\Psi_D^*\overline{A}}_{L^2(\mathbb B^{n+1})}\nonumber\\
 &\lesssim&\|i_{\partial D}^\ast (A -\overline{A})\|_{L^2(\partial D)}+\|\overline{A^g} - \Psi_D^*\overline{A}\|_{L^2(\mathbb B^{n+1})}\nonumber\\
 &&+\|i_{\partial D}^\ast (A -\overline{A})\|_{L^2(\partial D)}\|\pi_D(i^*_{\mathbb S^n}A_D^g) - i^*_{\mathbb S^n}\Psi_D^*B\|_{L^2(\mathbb S^n)}\nonumber\\
 &&+ \|i_{\partial D}^\ast (A -\overline{A})\|_{L^2(\partial D)}\left(\|\Psi^*_DB\|_{L^{n+1}(\mathbb B^{n+1})} +\vertii{\psi_D^*B}_{L^{n+1}(\mathbb B^{n+1})}\right)\ .
\end{eqnarray}
In the last two lines from \eqref{verifaaa1}, we recognize the same expression as in the third line of \eqref{A-19}, thus we can use the same reasoning that leads from \eqref{A-19} to \eqref{A-19-1} to write
\begin{eqnarray}\label{verifaaa1-2}
\lefteqn{\|\hat A_D - \Psi_D^*\overline{A}\|_{L^2(\mathbb B^{n+1})}\lesssim\|i_{\partial D}^\ast (A -\overline{A})\|_{L^2(\partial D)}+\|\overline{A^g} - \Psi_D^*\overline{A}\|_{L^2(\mathbb B^{n+1})}}\nonumber\\
&&+\|i_{\partial D}^\ast (A -\overline{A})\|_{L^2(\partial D)}\left(\vertii{i^*_{\partial D} A^g}_{L^n(\partial D)} + \vertii{i^*_{\partial D}F}_{L^{\frac{n}{2}}(\partial D)} + \verti{\overline F}\right).
\end{eqnarray}
The remaining estimate we need is the one below, which follows from the definition of $\overline{A^g}$. We will use also the fact that $(1-\pi_D)(g^{-1}dg)=0$ and \eqref{bound1-pi} for $p=1$ together with the H\"older inequality and \eqref{A-2}, and \eqref{psipresnorm}.
\begin{eqnarray}\label{barad}
\|\overline{A^g} - \Psi_D^*\overline{A}\|_{L^2(\mathbb B^{n+1})}&\lesssim_\Psi& \|(1-\pi_D)(i^*_{\mathbb S^n}(A_D^g - \Psi^*\overline{A}))\|_{L^2(\mathbb S^n)}\nonumber\\[3mm]
&=&\|(1-\pi_D)(i^*_{\mathbb S^n}(g^{-1}\Psi_D^*Ag - \Psi_D^*\overline{A}))\|_{L^\infty(\mathbb S^n)}\nonumber\\[3mm]
&\lesssim&\|i^*_{\partial D}(A-\overline{A})\|_{L^2(\partial D)}\|i^*_{\partial D}A\|_{L^2(\partial D)} \nonumber\\[3mm]
&&+ \|(1-\pi_D)(i^*_{\mathbb S^n}(\Psi^*_DA - \Psi^*_D\overline{A})\|_{L^\infty(\mathbb S^n)}\ .
\end{eqnarray}
Regarding the first term in \eqref{barad}, the factor $\|i^*_{\partial D}A\|_{L^2(\partial D)}$ is bounded by $\epsilon_0$ by hypothesis, thus we may absorb the first term into \eqref{verifaaa1}. To estimate the second term we use the fact that $\Psi_D^*$ is in fact linear on $1$-forms (contrary to the case of $k$-forms for $k\ge 2$), i.e. $\Psi^*_DA - \Psi^*_D\overline{A}=\Psi^*_D(A -\overline{A})$, and thus we may use again \eqref{bound1-pi} and obtain the strong bound
\[
\|(1-\pi_D)(i^*_{\mathbb S^n}(\Psi^*_DA - \Psi^*_D\overline{A}))\|_{L^\infty(\mathbb S^n)}\lesssim_\Psi\|i^*_{\partial D}(A-\overline{A})\|_{L^1(\partial D)}\ .
\]
Combining this with \eqref{barad} and inserting then into \eqref{verifaaa1-2}, we obtain 
\begin{eqnarray}\label{verifaaa1-3}
\lefteqn{\|\hat A_D - \Psi_D^*\overline{A}\|_{L^2(\mathbb B^{n+1})}}\nonumber\\
&\lesssim&\|i_{\partial D}^\ast (A -\overline{A})\|_{L^2(\partial D)}\left(1+\vertii{i^*_{\partial D} A^g}_{L^n(\partial D)} + \vertii{i^*_{\partial D}F}_{L^{\frac{n}{2}}(\partial D)} + \verti{\overline F}\right),
\end{eqnarray}
and by using hypothesis \eqref{Fgood} we find the bound \eqref{AAA}, as desired.
\end{proof}
%
\section{The space $\widetilde{\mathcal A}_G([-1,1]^n)$ and the setup for tracking energy concentration}\label{sec:goodbad}
%
\subsection{Local model for the space of weak connections}
%
We prepare now to define (in Definitions \ref{def:loc} and \ref{def:locmod} below) a localized-in-space model $\widetilde{\mathcal A}_G([-1,1]^n)$ for our space $\mathcal A_G(M^n)$ for the case of $M=[-1,1]^n$. The intuition is that $[-1,1]^n$ models a chart on a general manifold $M^n$, and we orient it to follow the level-sets of the functions $f$ appearing in Definition \ref{def:weakconn}. Therefore in Definition \ref{def:locmod} below we only use coordinate functions as slicing functions $f$. Our main results will be first proved in this setting in order to make the proofs clearer, and then extended to a general setting in Section \ref{sec:global}.

\medskip

Let $1\le i\le n$ be an integer and denote 
\[
H(k,t):=\{(x_1,\ldots, x_n)\in[-1,1]^n:\ x_k=t\}\ .
\]
We then consider the natural coordinates 
\[
i_{k,t}:[-1,1]^{n-1}\to H(k,t),\quad i_{H(k,t)}(x_1,\ldots,x_{n-1}):=(x_1,\ldots,x_k,t,x_{k+1},\ldots,x_{n-1})\ .
\]
More generally, for the case of $k$-dimensional coordinate subspaces we proceed as follows. Let $I=\{i_1,\ldots,i_k\}$ where we used the ordering of indices $1\le i_1<i_2<\cdots<i_k\le n$. Then for a $k$-ple of real numbers $T=(t_{i_1},\ldots,t_{i_k}), t_i\in[-1,1]$ indexed by $I$, we denote
\[
H(I,T):=\{(x_1,\ldots, x_n)\in[-1,1]^n:\ \forall i\in I,\ x_i=t_i\}.
\]
The parameterization $i_{H(I,T)}:[-1,1]^{n-k}\to H(I,T)$ of $H(I,T)$ will then be given by
\[
i_{H(I,T)}(x_1,\ldots, x_k):=(y_1,\ldots,y_n) \quad \text{ with }\quad y_i:=\left\{\begin{array}{ll} x_{j_\alpha}&\text{ if }i=j_\alpha\in J\\[3mm]t_{i_\beta}&\text{ if }i=i_\beta\in I\end{array}\right. ,
\]
in which we used the ordering $1\le j_1<j_2<\cdots<j_{n-k}\le n$ of the indices in $J:=\{1,\ldots,n\}\setminus I$.

\medskip

We now pass to define the space $\widetilde{\mathcal A}_G$. 
\begin{definition}\label{def:loc} Let $n=4$. Then we define $\widetilde{\mathcal A}_G([-1,1]^4)$ as the set of $A\in L^2([-1,1]^4, \wedge^1\mathbb R^4, \mathfrak g)$ such that the following properties hold:
\begin{subequations}\label{ag4d}
\begin{equation}\label{ag4d-1}
\int_{[-1,1]^4}|F_A|^2<+\infty
\end{equation}
\begin{equation}\label{ag4d-2}
\forall B\Subset [-1,1]^4\text{ open }, \exists g_B\in W^{1,2}(B,G)
\mbox{ s.t. }\ A^{g_B}\in L^4(B, \wedge^1\mathbb R^4, \mathfrak g).
\end{equation}
\end{subequations}
\end{definition}
\begin{definition}[Local model of weak connections, $n\ge 5$]\label{def:locmod}
We  define the space $\widetilde{\mathcal A}_{G}([-1,1]^n)$ of $L^2$ weak connections on singular bundles over $[-1,1]^n$ to be composed of all $A\in L^2([-1,1]^n, \wedge^1\R^n\otimes\mathfrak g)$ such that the following hold:
\begin{subequations}\label{eqdefag}
\begin{equation}\label{eqdefag-1}
F_A\stackrel{\mathcal D'}{=}dA+A\wedge A\in L^2,
\end{equation}
\begin{equation}\label{eqdefag-2}
  \forall k=1,\ldots,n,\,\text{ a.e. }t\in[-1,1],\ i^*_{H(k,t)}A\in\bar{\mathcal A_G}([-1,1]^{n-1})
\end{equation}
\end{subequations}
\end{definition}
We now note some important facts concerning the above definitions.
\begin{rmk}[slicing only by $4$-planes]\label{rmk:locmodcompare}
By expanding the inductive condition \eqref{eqdefag-2} of Definition \ref{def:locmod}, we may replace it by 
\begin{equation}\label{eqdefag-2bis}
 \forall I\subset\{1,\ldots,n\}, \# I=n-4,\,\text{ a.e. }T\in[-1,1]^I,\ i^*_{H(I,T)}A\in \widetilde{\mathcal A}_G([-1,1]^4)\ ,
\end{equation}
and we note that this condition would become equivalent to the one from Definition \ref{def:weakconn} for $M^n=[-1,1]^n$ if we were to replace the class $f\in C^\infty([-1,1]^n,\mathbb R^{n-4})$ by the smaller one given just by subsets of the coordinates:
\begin{equation}\label{f-coord}
\mathcal C_{n,n-4}:=\left\{f:[-1,1]^n\to \mathbb R^{n-4}:\ \exists I\subset\{1,\ldots,n\}, \# I=n-4,\ f(x_1,\ldots,x_n)=(x_i)_{i\in I}\right\}.
\end{equation}
\end{rmk}
\begin{rmk}[About $L^4$ and $W^{1,2}$]
As a consequence of the gauge extraction theorem \ref{coulstrange}, condition \eqref{ag4d-2} is equivalent to the following more classical condition (present also in \cite{PRym}):
\begin{equation}\label{ag4d-2bis}
\forall B\Subset[-1,1]^4\text{ open }\exists g_B\in W^{1,2}(B,G)\text{ s. t. } A^{g_B}\in W^{1,2}(B, \wedge^1\mathbb R^4, \mathfrak g).
\end{equation}
The equivalence of condition \eqref{ag4d-2} and \eqref{ag4d-2bis} under the condition \eqref{ag4d-1} in $4$ dimensions can be proved as follows. First note that the proof of our gauge extraction theorem \ref{coulstrange} for $\pi=0$ and $n=4$ remains valid in case we replace $\mathbb S^4$ by a small ball $B\subset [-1,1]^4$ and provides a local gauge $g_B$ in which $\|A^{g_B}\|_{L^4(B)}\le \|F\|_{L^2(B)}=\|dA^{g_B} + A^{g_B}\wedge A^{g_B}\|_{L^2(B)}, d^*A^g=0$ thus $\|A^{g_B}\|_{W^{1,2}}\le C (\|F\|_{L^2} +\|A^{g_B}\|_{L^4}^2)$ if $\|F\|_{L^2(B)}\le \epsilon_0$. Then the gauge-patching reasoning similar to the one of \cite{PRym}'s compactness result (H) allows to prove the fact that such good gauges can be patched over finite unions of small-energy balls covering a given compact, as desired. The existence of such covers follows by the fact that $F\in L^2$.
\end{rmk}
In direct analogy to \eqref{def_delta}, for $A,A'\in\widetilde{\mathcal A}_G([-1,1]^n)$ we define the pseudo-distance $\tilde\delta:=\tilde\delta_{\mathrm{conn}}$ by
\begin{equation}\label{tilde_delta}
\tilde\delta^2(A,A')=\tilde\delta_{\mathrm{conn}}^2(A,A'):=\sup_{f\in \mathcal C_{n,n-4}}\inf_{g:[-1,1]^n\to G}\int_{[-1,1]^n}\left|\left(dg + Ag-gA'\right)\wedge f^*\omega\right|^2\frac{\mathrm{d}vol}{\left|f^*\omega\right|},
\end{equation}
which only differs from \eqref{def_delta} by the fact that we have replaced the constraint $f\in C^\infty([-1,1]^n,\mathbb R^{n-4})$ by $f\in \mathcal C_{n,n-4}$. We analogously can define a distance $\tilde\delta_{\mathrm{curv}}(F,F')$ as in \eqref{def_delta_curv} between curvature forms $F,F'\in L^2\curvforms{[-1,1]^n}$.
%
\subsection{Choosing cubeulations}\label{goodcubeulation}
%
To choose well-behaved cubeulations we base ourselves mainly on Fubini's theorem. We proceed as follows:
\begin{itemize}
\item Fix a small scale $r>0$ which will be the size of the cubes used in our cubeulation.
\item For $i\in\{1,\ldots,n\}$ and $t\in[0,r[$, the family of coordinate hyperplanes inside $[-1,1]^n$ is denoted as follows:
\[
\mathcal F_{r,i,t}:=\left\{H(i,t'):\ t'\in(r\mathbb Z + t)\cap[-1,1]\right\}\ .
\]
\item For $I\subset \{1,\ldots,n\}$ with $\#I=n-k$ and $t_I\in[0,r[^I$ we parameterize $k$-dimensional cubes as follows:
\[
\mathcal F_{r,I,t_I}:=\left\{\cap_{i\in I}H(i,t'_i):\ \forall i\in I\ H(i,t'_i)\in\mathcal F_{r,i,t_i}\right\}\ .
\]
\item For $t=(t_1,\ldots,t_n)\in[0,r[^n$ and $\alpha=(\alpha_1,\ldots, \alpha_n)\in \mathbb Z^d, i\in\{1,\ldots,n\}$, we define the cube 
\[
C_{r,t,\alpha}:= t+ r\alpha + [0,r]^n\ .
\]
\item Corresponding to the subdivision of $[-1,1]^n$ in cubes from the family $\{C_{r,t,\alpha}\}_\alpha$ we denote by $\mathcal C_{r,t}$ the polyhedral complex generated by the following under intersection:
\[
P_{r,t}:=\left\{C_{r,t,\alpha}:\ \alpha\in \mathbb Z^d,\ C_{r,t,\alpha}\cap[-1,1]^n\neq\emptyset\}\right\}\ .
\]
Then we denote by 
\[
P^{(k)}_{r,t}
\]
the $k$-skeleton of $P_{r,t}$ and by 
\[
C^{(k)}_{r,t}
\]
the set of $k$-dimensional faces of cubes $C_{r,t,\alpha}$ contributing to $P^{(k)}_{r,t}$. More generally, if $\mathcal S$ is a subcomplex of $\mathcal C_{r,t}$ then we denote by $\mathcal S^{(k)}$ the set of $k$-dimensional faces of $\mathcal S$. We note that $\mathcal F_{r,I,t_I}\subset P^{(k)}_{r,t}$ if $k=\#I$ and if the $I$-coordinates of $t$ coincide with those of $t_I$.
\item For $C_\alpha\in P_{r,t}$ and $\omega\in L^1_{loc}([-1,1]^n,V)$ we define by superposition 
\[
\bar\omega_{C_\alpha}:=\frac{1}{|C_\alpha|}\int_{C_\alpha}\omega(x) dx\in V\ .
\]
This will be only used in the case where 
\[
V=\wedge^1\mathbb R^n\otimes\mathfrak{g}\quad\text{or}\quad V=\wedge^2\mathbb R^n\otimes\mathfrak{g}\ .
\]
Note that in order to simplify the proof later, in Section \eqref{sec:strapp} we will re-define the quantities $\overline{A}_{C_\alpha}$ and $\overline{F}_{C_\alpha}$ to denote a (slightly) different averaging.
\item If $\chi_{C_\alpha}:[-1,1]^n\to \{0,1\}$ is the function which equals $1$ on $C_\alpha$ and $0$ outside it, then corresponding to the complex $P_{r,t}$ we also define the piecewise constant $V$-valued function
\[
\bar \omega:=\sum_{C_\alpha\in P_{r,t}} \chi_{C_\alpha}(x) \bar \omega_{C_\alpha}\ .
\]
\end{itemize}
In the above notations, if no confusion arises we will often omit either one or both of the indices $t$ and $r$.

\medskip

\noindent Next, fixing the underlying connection and curvature forms $\omega=A$ or $\omega = F_A$ as in \eqref{eqdefag}, we want to find good choices of $t$ such that the skeleta defined above do not carry too much $L^2$-energy. This is done in the next proposition, which we state for general forms for clarity (cf. also \cite[Prop. 2.6]{PRym} for a particular case).
\begin{proposition}\label{prop:centergridball}
Let $0\le p\le n$ and $\omega\in L^2([-1,1]^n,\wedge^p\mathbb R^n\otimes \mathfrak g)$ and fix $\delta\in]0,1[$. There exists a decreasing function $o_{\omega,\delta}:[0,1]\to\mathbb R^+$ such that $\lim_{r\downarrow 0}o_{\omega,\delta}(r)=0$ with the following properties.

\medskip

For all $r>0$ there exist a subset $T_{r,\delta}(\omega)\subset [0,r]^n$ with $|T_{r,\delta}(\omega)|\ge\delta r^n$, a constant $C$ depending only on $\delta,n,\mathfrak g$ such that for all $p\le k\le n-1$ and for all $t\in T_{r,\delta}(\omega)$ all the restrictions $i^*_Q\omega$ appearing below are measurable and such that there holds:
\begin{subequations}\label{kskeleton}
\begin{equation}\label{kskeletonbound}
  r^{n-k}\sum_{C_\alpha\in P_{r,t}}\sum_{Q\in C_\alpha^{(k)}}\int_Q|i^*_Q\omega|^2\leq C_1 \int_{[-1,1]^n}|\omega|^2\ ,
\end{equation}
\begin{equation}
  \label{kskeletonaverage}
   r^{n-k}\sum_{C_\alpha\in P_{r,t}}\sum_{Q\in C_\alpha^{(k)}}\int_Q|i^*_Q(\omega-\bar \omega_{C_\alpha})|^2\leq o_{\omega,\delta}(r)\ .
\end{equation}
\end{subequations}
\end{proposition}
\begin{proof}
Note that because there are $n-k$ coordinates which are constant along any $k$-face, it follows that for any $Q\in P_{r,t}^{(k)}$ there holds (and the inequality may be non-sharp only for those $Q$ which are $r$-close to $\partial[-1,1]^n$)
\[
 \#\{C_\alpha\in P_{r,t}:\ Q\in C_\alpha^{(k)}\} \le 2^{n-k}\ .
\]
If $t_I\in [0,r]^I$ is the vector of $I$-indexed ordered coordinates of $t\in[0,r]^n$ then we have
\[
\sum_{C_\alpha\in P_{r,t}}\sum_{Q\in C_\alpha^{(k)}}\int_Q|i^*_Q\omega|^2=2^{n-k} \sum_{I:\#I=n-k}\int_{\mathcal F_{r,I,t_I}}|i_{F_{r,I,t_I}}^*\omega|^2:=2^{n-k}\sum_{\#I=n-k}\mathcal I_{I,t_I}(\omega) \ .
\]
Note that with the above notations if $\omega^I$ is the form obtained from $\omega$ by retaining only the terms $dx_{i_1}\wedge\cdots\wedge dx_{i_p}$ with $i_1,\ldots,i_p\in I$ then there holds
\[
\int_{t_I\in[0,r]^I} I_{I,t_I}(\omega) dt_I =\int_{[-1,1]^n}|\omega^I|^2:=\mathcal I_I(\omega)\le\int_{[-1,1]^n}|\omega|^2\ .
\]
By Chebychev's inequality we obtain that since $|[0,r]^I|=r^{n-k}$ there holds
\begin{eqnarray*}
|T_{I,\delta}(\omega)&:=&\left|\left\{t\in[0,r]^n:\ 2^{k-n}\frac{n!}{k!(n-k)!}\frac{\mathcal I_{I,t_I}(\omega)}{r^{n-k}}> C_1\mathcal I(\omega) \right\}\right|\\
&\le& 2^{n-k}\frac{n!}{C_1 k!(n-k)!}r^n:= C_2 r^n\ .
\end{eqnarray*}
Then by subadditivity we obtain
\[
\left|\bigcup_{\#I=n-k}T_{I,\delta}(\omega)\right|\le \sum_{\#I=n-k}|T_{I,\delta}(\omega)|\le \frac{n!}{k!(n-k)!}C_2r^n:= C_3 r^n.
\]
If we denote 
\[
T(\omega):=[0,r]^n\setminus \bigcup_{\#I=n-k}T_{I,\delta}(\omega)
\]
then we obtain \eqref{kskeletonbound} for all $t\in T(\omega)$ and 
\[
|T(\omega)|\ge (1-C_3)r^n\ ,
\]
which can be made arbitrarily close to $r^n$ by choosing $C_1$ large enough.

\medskip

Regarding \eqref{kskeletonaverage} we first note that by mollification for any $\epsilon>0$ we may obtain $\omega_\epsilon\in C^1([-1,1]^n,\wedge^p\mathbb R^n\otimes \mathfrak g)$ such that 
\begin{equation}\label{mollifest}
\|\omega_\epsilon-\omega\|_{L^2}^2\le \epsilon
\end{equation}
and then we find a set $T(\omega_\epsilon-\omega)$ as above. For $r>0$ and $t\in T(\omega)\cap T(\omega_\epsilon -\omega)$ we find that for all $C_\alpha\in P_{r,t}$ and all $Q\in C_\alpha^{(k)}$ there holds 
\begin{equation}\label{c1est}
\int_Q|i^*_Q(\omega_\epsilon - \overline{(\omega_\epsilon)}_{C_\alpha})|^2\le r^{k+2}\|\nabla\omega_\epsilon\|_{L^\infty}^2\ .
\end{equation}
Now we note that the relation
\[
R_{r,t}:=\{(Q, C_\alpha)\in P_{r,t}^{(k)}\times P_{r,t}:\ Q\in C_\alpha^{(k)}\}
\]
can be partitioned into $2^{n-k}$ relations $R_{r,t}^\beta, \beta\in\{1,\ldots,2^{n-k}\}$ such that for each such $R_{r,t}^\beta$ and for each $Q\in P_{r,t}^{(k)}$ there exists \emph{at most one} choice of $C_\alpha\in P_{r,t}$ such that $(Q, C_\alpha)\in R_{r,t}^\beta$. Let $\mathcal Q_\beta$ be the set of such $Q$ such that one such $C_\alpha$ exists. We then find that
\begin{equation}\label{subdivbeta}
\sum_{C_\alpha\in P_{r,t}}\sum_{Q\in C_\alpha^{(k)}}\int_Q|i^*_Q(\omega-\bar \omega_{C_\alpha})|^2 = \sum_\beta \sum_{Q\in \mathcal Q_\beta}\sum_{\{C_\alpha: (Q,C_\alpha)\in R_{r,t}^\beta\}}\int_Q|i^*_Q(\omega-\bar \omega_{C_\alpha})|^2\ .
\end{equation}
We find from \eqref{mollifest} via Jensen's inequality and since $|i^*_Q\gamma|^2\le |\gamma|^2$ if $\gamma$ is a constant form, that for each $\beta$ 
\begin{equation}\label{estbeta}
\sum_{Q\in \mathcal Q_\beta}\sum_{\{C_\alpha: (Q,C_\alpha)\in R_{r,t}^\beta\}}\int_Q|i^*_Q(\overline{(\omega_\epsilon)}_{C_\alpha}-\bar \omega_{C_\alpha})|^2 \le \epsilon\ .
\end{equation}
Now by the inequality $(a+b+c)^2\le 9(a^2+b^2+c^2)$ applied to the integrals over each $Q$ appearing above and using \eqref{mollifest} together with \eqref{kskeletonbound} for $\omega-\omega_\epsilon$, \eqref{c1est}, \eqref{estbeta} and \eqref{subdivbeta}, we obtain that for $t\in T(\omega)\cap T(\omega-\omega_\epsilon)$ there holds
\begin{eqnarray}
\sum_{C_\alpha\in P_{r,t}}\sum_{Q\in C_\alpha^{(k)}}\int_Q|i^*_Q(\omega-\bar \omega_{C_\alpha})|^2 &\le& 9\left[(2^{n-k} + 1)C_1r^{k-n}\epsilon + r^{k+2}\|\nabla\omega_\epsilon\|_{L^\infty}^2 \# R_{r,t}\right]\nonumber\\[3mm]
&\le& r^{k-n}\cdot 9\left[(2^{n-k} + 1)C_1\epsilon + 2^{n-k}r^2\|\nabla\omega_\epsilon\|_{L^\infty}^2 \right] \ ,\label{epsiloner}
\end{eqnarray}
where in the last inequality we estimate $\#R_{r,t}\le 2^{n-k}r^{-n}$ because $\#P_{r,t}\le r^{-n}$ and $\#C_\alpha^{(k)}\le 2^{n-k}$ for each $C_\alpha\in P_{r,t}$.

\medskip

We now fix $C_1$ such that $2C_3<1-\delta$. This ensures that 
\[
 |T(\omega)\cap T(\omega -\omega_\epsilon)|>\delta r^n\ .
\]
Consider now \eqref{epsiloner}: By fixing $\epsilon>0$ small and then choosing $r\lesssim \epsilon\|\nabla \omega_\epsilon\|_{L^\infty}^{-1}$, we find that the term multiplying $r^{k-n}$ on the right can be made arbitrarily small for small $r$, thus there exists $o_{\omega, \delta}(r)$ such that $o_{\omega,\delta}(r)\to 0$ as $r\downarrow 0$ and such that \eqref{kskeletonaverage} holds. We thus choose $T_{r,\delta}(\omega)=T(\omega)\cap T(\omega -\omega_\epsilon)$, and the properties \eqref{kskeleton} hold, as desired.
\end{proof}
%
\subsection{Good cubes and bad cubes}
%
Later on we will to apply Proposition \ref{thm:extension} on $r$-dilations of $k$-faces of our good cubeulation of scale $r$. Thus the hypotheses of the Proposition will have to hold for all $k\le n$, and the estimates \eqref{AA}, \eqref{AAA} should be fitting the same bounds as in the gauge extraction Theorem \ref{coulstrange} for $k+1$. We are lead to consider "good" those $k$-faces of our good cubeulation on which these iterative criteria are feasible. The fact that in Proposition \ref{thm:extension} the bounds for the replacements $\hat A, F_{\hat A}$ are controlled in terms of the ones for $i^*_{\partial D}A, i^*_{\partial D}F$ allows us to just require bounds once, in the starting dimension $4$.
\begin{definition}[good $k$-face]\label{goodkface}
Let $A,F$ be respectively a connection and a curvature form like in \eqref{eqdefag}. Fix a scale-$r$ good cubeulation $P_{r,t}$.

\medskip

We define all faces of dimension $3$ or lower to be good.

\medskip

Let $C\in P_{r,t}^{(k),}, k\ge 4$. We say that $C$ is a \emph{$\delta$-good $k$-face} if for each $n$-dimensional cube $C_\alpha:=C_{r,t,\alpha}\in P_{r,t}$ such that $C\in C_\alpha^{(k)}$, the following estimates hold for all $C'\in C^{(4)}$:
\begin{equation}\label{deltagood4}
\begin{array}{c}
\verti{\frac{1}{C'}\int_{C'}i^*_{C'}F}\le\delta r^{-2},\quad\verti{\frac{1}{C'}\int_{C'}i^*_{C'}A}\le\delta r^{-1}, \quad\int_{C'}|i^*_{C'}F|^2\le \delta,\\[3mm]
 \int_{C'}|i^*_{C'}(F-\overline{F}_{C_\alpha})|^2\le \delta,\quad \int_{C'}|i^*_{C'}(A-\overline{A}_{C_\alpha})|^2\le \delta r^{2}\ .
\end{array}
\end{equation}
If $C$ is not a $\delta$-good $k$-face then we call it a \emph{$\delta$-bad $k$-face}.
\end{definition}
Note that the above conditions are scale-invariant.

\medskip

The following direct consequence of the above definition and of Proposition \ref{prop:centergridball} will be used in order to allow a dominated convergence argument within the "rough approximation" step of our proof.
\begin{lemma}\label{goodballsaremany}
If $P_{r,t}$ is a good cubeulation for $A, F$ then the total number $N_\delta$ of $\delta$-bad $n$-faces satisfies for $r>0$ small enough, depending on $A,F$,
\begin{equation}\label{boundbadcubes}
N_\delta \lesssim \frac{\|F\|_{L^2([-1,1]^n)}^2}{\delta r^{n-4}} + \frac{\|A\|_{L^2([-1,1]^n)}^2}{\delta^2r^{n-2}}+ \frac{1}{\delta r^{n-4}}\ ,
\end{equation}
in particular the total volume of all bad $n$-cubes $r^n N_\delta$ vanishes as $r\to 0$.
\end{lemma}
%
%
\section{The strong approximation theorem}\label{sec:strapp}
%
In this section we prove that forms $F_A$ corresponding to $A\in\widetilde{\mathcal A}_{G}([-1,1]^n)$ can be strongly approximated up to gauge by \emph{smooth} curvatures on bundles with \emph{controlled} defects, i.e. by elements of the following space:
\begin{equation}\label{Rinf}
 \mathcal R^{\infty}([-1,1]^n):=\left\{
 \begin{array}{c}
  A\text{ connection form s.t. there exists}\\[3mm]
  \text{ a polyhedral set }\Sigma^{n-5}\subset [-1,1]^n,\\[3mm]
  \text{ s.t. }A=A_\nabla \text{ for a smooth connection}\nabla\\[3mm]
  \text{on some smooth }G\text{-bundle }E\to M^n\setminus\Sigma^{n-5}
 \end{array}
\right\}\ .
\end{equation}
In our construction the set $\Sigma^{n-5}$ is the union of a finite number of intervals of dimension $n-5$ parallel to the coordinate directions.
\begin{rmk}
Equivalent to having a smooth connection $\nabla$ as above is to have a smooth presheaf. This means that we have a good cover $\{U_\alpha\}$ of $M^n\setminus \Sigma^{n-5}$ and smooth connection forms $A_\alpha\in C^\infty\connforms{U_\alpha}$ related by smooth changes of gauges $g_{\alpha\beta}\in C^\infty\gauges{U_\alpha\cap U_\beta}$ such that $A_\alpha=\left(A_\beta\right)^{g_{\alpha\beta}}$.  See \cite{isobehigher} for more discussion on the presheaf point of view on classical connections, and Sections 1.2.4 and Appendix A of \cite{PRym} for the description of the above realization map in $4$-dimensions. By a reasoning completely analogous to \cite[App. A]{PRym} it is possible to obtain the existence of classical bundles based only on our locally-$L^n$-connection forms related by $W^{1,n}$-gauges, as obtained in \eqref{regulara} over our good cubes.
\end{rmk}
The result which we prove is the following:
\begin{theorem}\label{thm:strongapprox}
If $A\in\widetilde{\mathcal A}_G([-1,1]^n)$ then there exists a sequence of connection forms $A_j\in \mathcal R^\infty([-1,1]^n)$ with connection forms $F_j:=F_{A_j}=dA_j+A_j\wedge A_j$ such that there exist a sequence of gauge changes $g_j\in W^{1,2}([-1,1]^n,G)$ for which, as $j\to\infty$, there holds
\[
\|g_j^{-1}dg_j+g_j^{-1}A_j g_j - A\|_{L^2([-1,1]^n)}\to 0, \quad \|g_j^{-1}F_jg_j - F_A\|_{L^2([-1,1]^n)}\to 0\ .
\]
\end{theorem}
We will construct our approximants by successive extensions starting with the restriction of a starting $A\in \widetilde{\mathcal A}_G([-1,1]^n)$ to the support of the $4$-skeleton of a well-chosen cubeulation of $[-1,1]^n$.
Once the above result is proved, in order to pass to the situation of a general compact Riemannian manifold $(M^n,h)$, it will suffice to approximate the connections locally in coordinate charts, and to use the fact that the coordinate transformations of change of chart, or the presence of a $C^1$-regular metric $h$ do not alter the $W^{1,k}$-bounds which appear throughout the proof, and the final mollification away from $\Sigma^{n-5}$ can be performed in the same way. 

\medskip

We note here that the theory of Sobolev presheaves as in \cite{Isobecrit}, \cite{isobehigher} can be used in order to link the setting of weak connections treated here to that of classical connections, like explain in the appendix of our paper \cite{PRym} about the $5$-dimensional case. In particular the same reasoning shows that having smooth connection $1$-forms on local charts directly allows to create a principal bundle such that these $1$-forms are the differential-geometric connection forms of a connection on the associated bundle for the adjoint representation. We do not delve onto this topic in this paper, and we refer the interested reader to the above-cited works instead.

\medskip

We next set up the proof of Theorem \ref{thm:strongapprox}.
%
\subsection{Notations and framework}
%
For the whole proof, we will use a small parameter $\delta_0>$, whose choice will be precised during the proof, and will depend only on $n$ and $G$.
%
\subsubsection{Choice of a cubelation, good and bad cubes}\label{sec:choicecube}
%
We choose a cubeulation $P_{r,t}$ at scale $r>0$, such that
\begin{itemize}
\item $P_{r,t}$ satisfies \eqref{kskeleton} contemporarily for $\omega=A$ and for $\omega=F_A$,
\item all hyperplane families $\mathcal F_{r,I,t_I}$ with $\#I=n-4$ are composed exclusively of planes such that the good gauges as assumed in \eqref{eqdefag} exist. In particular we have a good $L^4$-gauge on all $4$-faces of the relative boundary of $5$-faces in $P^{(5)}_{r,t}$.
\item for any $4$-plane $H(I,T)$ which intersects some face of $P_{r,t}^{(4)}$, condition \eqref{eqdefag-2bis} holds, i.e. there exists a gauge $g(I,T)$ on $H(I,T)$ chosen, such that $\left(i^*_{H(I,T)}A\right)^{g(I,T)}$ is $L^4$.
\end{itemize}
In order to obtain the existence of such $P_{r,t}$, we first apply Proposition \ref{prop:centergridball} separately with the choices $\omega=A$ and $\omega=F_A$, and obtain good sets of parameters which we may denote $T_{r,\delta}^{A}$ and $T_{r,\delta}^{F_A}$, respectively. If $\delta>1/2$ then we find that $T_{r,\delta}^{A}\cap T_{r,\delta}^{F_A}$ has positive volume, and then any parameter $t$ in this intersection satisfies the properties \eqref{kskeleton} for both $A$ and $F_A$. Then, using the definition \ref{def:locmod} and Remark \ref{rmk:locmodcompare}, we find that some such good $t$ is such that each $4$-plane which enters the definition of $P_{r,t}^{(4)}$ also satisfies \eqref{eqdefag-2bis}.

\medskip

We then fix the cubeulation $P_{r,t}$ as above. We will call an element of $C^{(n)}_{r,t}$ a \textbf{good cube} provided it satisfies (for $k=n$, the above choice of $\delta_0>0$ and for our present weak connection form $A$) the conditions of Definition \ref{goodkface}, i.e. if \eqref{deltagood4} holds. Any $n$-cube from $C^{(n)}_{r,t}$ which is not good will be called a \textbf{bad cube}.
%
\subsubsection{Bilipschitz parameterizations in intermediate dimensions}\label{sec:parametriz}
%
We fix bilipschitz parametrizations
\begin{equation}\label{bilippar}
\Psi_k:\mathbb B^k\to [-1,1]^k\ .
\end{equation}
Then for each coordinate $k$-face $C_r\in P_{r,t}^{(k)}$ we will denote
\begin{equation}\label{rescalingtotheorigin}
C_r = \tau_{C_r}\circ\delta_r (C_1)\ ,\quad k=5,\ldots,n.
\end{equation}
where $\tau_{C_r}$ is a translation sending the origin to the center of $C_r$ and $\delta_r$ is a dilation by a factor of $r$. We then use a parameterization 
\[
\Psi_{C_1}=\Psi_k\circ R_{C_1}\text{, where }R_{C_1}\in SO(n)\ .
\]
The estimates in our proof will not depend on the precise choices of parameterisation effectuated at this stage, and only the Lipschitz constants of the intervening maps and of their inverses will be relevant.

\medskip

We may also assume that if $C_\alpha$ is a $k$-face of a cube $C_\beta$, i.e. $C_\alpha\in C_\beta^{(k)}$, then $\left.\Psi_{C_\beta}\right|_{C_\alpha}=\Psi_{C_\alpha}$. Denote by $\lambda_k$ the bi-lipschitz norm of $\Psi_k$.
%
\subsection{Proof of the approximation Theorem \ref{thm:strongapprox}}\label{sec:proofapprox}
%
The proof of Theorem \ref{thm:strongapprox} will proceed through the following steps:
\begin{enumerate}
\item We start with local gauges in which our connection is $L^4$-integrable on the $4$-skeleton $P_{t,r}$.
\item With a suitable choice of $\delta$, on the $\delta$-good $k$-dimensional faces, iteratively with respect to $k\ge4$, we extend the connection forms from the boundaries of $(k+1)$-dimensional cells to the interiors, via Theorem \ref{thm:extension}.
\item On the $\delta$-bad $k$-dimensional faces, we extend radially, again iteratively for $k\ge4$.
\item At the end of the extension we are able mollify our connections outside a $5$-dimensional polyhedral set (which is the support of the dual skeleton to the complex of bad cubes), providing the approximantion bounds as required in the statement of the theorem.
\end{enumerate}
%
\subsubsection{Step 1: $L^4$-connections locally on the $4$-skeleton}\label{sec:localconn}
%
We start with the following result, which we need for setting up the gauges defined on the faces on our skeleta:
\begin{proposition}[Controlled gauges with Dirichlet boundary datum]\label{prop:gauge_dirichlet}
Assume that $A\in L^2\connforms{\mathbb B^n}$ and that there exist $g\in W^{1,2}\gauges{\mathbb B^n}$ such that $A^g\in L^n\connforms{\mathbb B^n}$, and suppose that the curvature form $F_A$ satisfies $\vertii{F_A}_{L^{n/2}(\mathbb B^n)}<\epsilon_0$. Then there exists a gauge change $\tilde g\in W^{1,2}\gauges{\mathbb B^n}$ such that $A^{\tilde g}\in L^4\connforms{\mathbb B^n}$, $\tilde g|_{\partial \mathbb B^n}\equiv id$ and $\vertii{A^{\tilde g}}_{L^n(\mathbb B^n)}\lesssim \vertii{F_A}_{L^{n/2}(\mathbb B^n)}$.
\end{proposition}
The above result is the same as the main result of Uhlenbeck \cite{Uhl2}, with the two differences that we work with $A$ of regularity $L^n$ rather than $W^{1,n/2}$ and that we impose on our gauges $g$ the Dirichlet boundary condition rather than the Neumann one. This can be directly implemented in the proof (as presented in \cite[Thm. IV.4]{Riv1}) by treating the linearized operator between the so-defined spaces directly, without further essential modifications.

\medskip

The application of Proposition \eqref{prop:gauge_dirichlet} gives the following result. Note that the term ``good cover'' means that the maximal number of sets from the cover that overlap at any given point is finite.
\begin{corollary}[Finding $L^4$-connections on $4$-faces]\label{cor:l4_on_4face}
Assume that $C^4$ is a $4$-face of our skeleton. Then exists a finite good cover $\{U_\alpha\}_\alpha$, of $C^4$ by sets $U_\alpha$ that are bi-lipschitz equivalent to $\mathbb B^4$ and gauge change maps $g_\alpha\in W^{1,2}\gauges{U_\alpha}$ such that for all $\alpha$ we have $i_{U_\alpha}^*A^{g_\alpha}\in L^4\connforms{U_\alpha}$, ${g_\alpha}|_{U_\alpha \cap\partial C^4}\equiv id$ and $\vertii{i^*_{U_\alpha}A^{g_\alpha}}_{L^4(U_\alpha)}\lesssim \vertii{i^*_{U_\alpha}F_A}_{L^2(U_\alpha)}$.

\medskip

Moreover if $C^4$ is a good cube, then the above holds already for the trivial cover formed by only $C^4$ itself.
\end{corollary}
\begin{proof}
From the definition of a good skeleton, we use here only the property that $A\in L^2\connforms{C^4}$ and that there exists $g\in W^{1,2}\gauges{C^4}$ such that $A^g\in L^4\connforms{C^4}$. From the definition of a good cube we only need that $\vertii{i_{C^4}^*F_A}_{L^2(C^4)}\le \delta$ for $\delta\le C \epsilon_0$, where $C=(\op{Lip}\Psi_\alpha)^2(\op{Lip}(\Psi_\alpha^{-1}))^4$ is a geometric constant depending only on the bi-lipschitz constant of the map that identifies $U_\alpha$ to a ball -- and can be bounded independently on our choice of cover -- and $\epsilon_0$ is as in Proposition \ref{prop:gauge_dirichlet}.

\medskip

We may find a finite clopen cover $\{U_\alpha\}_\alpha$ of $C^4$, such that each $U_\alpha$ is itself bilipschitz-equivalent to a ball $\mathbb B^4_r$ via a map $\Psi_\alpha:\mathbb B^4_r\to U_\alpha$ and that $\int_{\Psi_\alpha^{-1}(U_\alpha)}|\Psi_{U_{j,\alpha}}^*F|^2 \le \epsilon_0$. Then we can identify $\mathbb B^4_r$ to the unit ball $\mathbb B^4$ by dilation, and this change of coordinates in $4$-dimensions leaves the $L^2$-norm of $F$ unchanged.

\medskip

Note that if $C^4$ is a good cube, then we can take the trivial covering $\{C^4\}$, because the smallness condition on $F$ is already satisfied for $\delta$ in \eqref{deltagood4} chosen as in the beginning of the proof.

\medskip

Then we use Proposition \ref{prop:gauge_dirichlet} -- and transfer the result to $U_{j,\alpha}$ via $\Psi_{U_{j,\alpha}}$ -- in order to find local gauges $g_{j,\alpha}\in W^{1,2}\gauges{U_{j,\alpha}}$ such that $g_{j,\alpha}|_{\partial U_{j,\alpha}}\equiv id$ and $A^{g_{j,\alpha}}\in L^4\connforms{U_{j,\alpha}}$. In particular, we find that ${g_\alpha}|_{U_\alpha\cap \partial C^4}\equiv id$, as desired.
\end{proof}
Next, we proceed by induction on the skeleta, using the following Lemma:
\begin{lemma}[gluing gauges]\label{lem:gluing_gauges}
Assume that $C^{k+1}$ is a $(k+1)$-dimensional cube, with $k\ge 4$. For each one of its $k$-dimensional faces $C^k_\alpha\in\left(\partial C^{k+1}\right)^{(k)}$, let $g_\alpha\in W^{1,2}\gauges{C^k_\alpha}$ and $A_\alpha\in L^2\connforms{C^k_\alpha}$ be such that $\left(A_\alpha\right)^{g_\alpha}\in L^k\connforms{C^k_\alpha}$ and such that whenever $C^k_\alpha$, $C^k_\beta \in \left(\partial C^{k+1}\right)^{(k)}$, then $g_\alpha=g_\beta$ on $C^k_\alpha\cap C^k_\beta$. If we define
\begin{equation}\label{glue_gauges}
g_{\partial C^{k+1}} := \sum_{C_\alpha^k\in\left(\partial C^{k+1}\right)^{(k)}} 1_{C^k_\alpha} g_\alpha,\quad  A_{\partial C^{k+1}} := \sum_{C_\alpha^k\in\left(\partial C^{k+1}\right)^{(k)}} 1_{C^k_\alpha} A_\alpha,
\end{equation}
then $g_{\partial C^{k+1}}\circ \Psi_{k+1}\in W^{1,2}(\mathbb S^k, G)$, $A_{\partial C^{k+1}}\in L^2\connforms{\partial C^{k+1}}$, and there holds
\begin{equation}\label{conserve_Lk}
\left(\Psi_{k+1}^*A_{\partial C^{k+1}}\right)^{g_{\partial C^{k+1}}\circ \Psi_{k+1}}\in L^k\curvforms{\mathbb S^k}.
\end{equation}
\end{lemma}
\begin{proof}
The fact that $d\left(g_{\partial C^{k+1}}\circ \Psi_{k+1}\right)$ is $L^2$ as desired follows by integration by parts, using the fact that $g_\alpha=g_\beta$ on the intersection of their domains. The property \eqref{conserve_Lk} follows by applying the chain rule to \eqref{glue_gauges}, and using the fact that the normals to the common boundary of neighboring regions $\Psi_{k+1}(C_\alpha^k)$, $\Psi_{k+1}(C_\beta^k)$ cancel each other.
\end{proof}
%
%
%
\subsubsection{Step 2: Extension on the good skeleton}\label{sec:ext_good}
%
As the extension done in this step of the proof will be by iteration on the dimension, starting from the case of $4$-faces, on which the connections that we create are equal to the original one, and then we replace the initial connection iteratively on the interiors of $(k+1)$-faces for $4\le k\le n-1$. The extension from $4$-faces to $5$-faces is slightly different than the general step, because it uses the conditions \eqref{deltagood4} directly, instead of using bounds obtained from previous steps. Therefore we explicitly present the following passages
\begin{itemize}
\item the first step of the induction, i.e. the extension from $4$-faces to $5$-faces,
\item the passage from $k$-faces to $(k+1)$-faces for general $5\le k\le n-1$, 
\item the final bound obtained for the $n$-faces.
\end{itemize}
%
\textit{Step 2.0: Preparation.} As our skeleton is already chosen we will denote it $P$ rather than $P_{r,t}$ for the rest of the proof. We first note that since the conditions \eqref{deltagood4} are dilation-invariant, we may, without loss of generality assume that $r=1$ up to dilation.

\medskip

\textit{Step 2.1: Base for the inductive extension.} Note that by applying Corollary \ref{cor:l4_on_4face} on all the faces of our skeleton $P^{(4)}$, we obtain for each good $4$-face $C^4_j$ a gauge $g_{C^4_j}\in W^{1,2}\gauges{C^4_j}$ such that $A^{g_{C_j^4}}\in L^4\connforms{C_j^4}$, $\left.g_{C^4_j}\right|_{\partial C_j^4}\equiv id$ and $\vertii{A^{g_{C_j^4}}}_{L^4(C_j^4)}\lesssim \vertii{F}_{L^2(C_j^4)}$. Then by using Lemma \ref{lem:gluing_gauges} we find for each good $5$-face $C^5_j$ a gauge $g_{\partial C_j^5}\in W^{1,2}\gauges{\partial C_j^5}$ such that $\left(i_{\partial C_j^5}^*A\right)^{g_{\partial C_j^5}}\in L^4\connforms{\partial C_j^5}$ and
\begin{equation}\label{agood_v}
\vertii{\left(i_{\partial C_j^5}^*A\right)^{g_{\partial C_j^5}}}_{L^4(\partial C_j^5)}\lesssim \vertii{i_{\partial C_j^5}^*F}_{L^2(\partial C_j^5)}.
\end{equation}
%
\textit{Step 2.2: First step of the extension, from $4$-faces to $5$-faces.} We next apply Theorem \ref{thm:extension} with the following choices (indicated by ``$\mapsto$''):
\begin{itemize}
\item $n\mapsto4$,
\item $D\mapsto C_j^5$ and $\Psi_D\mapsto\Psi_{C_j^5}$ as defined in Section \ref{sec:parametriz},
\item $\overline{F} \mapsto \overline{F}_{C^5_j}$ and $\overline{A} \mapsto \overline{A}_{C^5_j}$, where we denote
\begin{align}\label{f_5-a_5}
\overline{F}_{C^5_j}:=\frac{1}{10}\sum_{C_\alpha^4\in\left(C^5_j\right)^{(4)}}\tfrac{1}{\verti{C_\alpha^4}}\int_{C_\alpha^4}i_{C^4_\alpha}^*F,\nonumber\\ \overline{A}_{C^5_j}:=\frac{1}{10}\sum_{C_\alpha^4\in\left(C^5_j\right)^{(4)}}\tfrac{1}{\verti{C_\alpha^4}}\int_{C_\alpha^4}i_{C^4_\alpha}^*A,
\end{align}
\item $A\mapsto i_{C_j^5}^*A$ and $F\mapsto i_{C^5_j}^*F$,
\item $g\mapsto g_{\partial C_j^5}$, described above.
\end{itemize}
With the above choices, we can verify that the bilipschitz constant $C_D$ from Theorem \ref{thm:extension} is replaced by the bilipschitz constant of $\Psi_4$ fixed depending only on $k=5$ in Section \ref{sec:parametriz}, and the other hypotheses are verified as follows
\begin{itemize}
\item The condition \eqref{Agood} on $g\mapsto g_{\partial C_j^5}$ is verified by \eqref{agood_v}, and the bound on $A^g\mapsto \left(i^*_{C_j^5}A\right)^{g_{\partial C_j^5}}$ as required in \eqref{Fgood} follows from \eqref{agood_v} and from the bound on $F$ in \eqref{Fgood}.
\item The bounds \eqref{Fgood} on $\overline{F}\mapsto \overline{F}_{C_j^5}$ and $\overline{A}\mapsto\overline{A}_{C_j^5}$ follow by the conditions \eqref{deltagood4} valid in the case $k\mapsto 5$ and with $C\mapsto C_j^5$, due to the definitions \ref{f_5-a_5} and by triangle inequality, provided we have $\delta<\delta^{(4)}$ for a constant $\delta^{(4)}$ which will depend only on the bilipschitz constant of $\Psi_5$, and on the value of $\epsilon_0:=\epsilon_0^{(4)}$ appearing in Theorem \ref{thm:extension} in which we chose $n\mapsto4$.
\item The bounds \eqref{Fgood} on $F\mapsto F_{\partial C_j^5}$ and on $A\mapsto A_{\partial C_j^5}$ follow from the analogous bounds that appear in \eqref{deltagood4} with the choices $k\mapsto 5$ and $C\mapsto C_j^5$, by triangle inequality, up to diminishing $\delta^{(4)}$ by a combinatorial factor of $10$, equal to the number of $4$-faces of $C_j^5$.
\end{itemize}
As an outcome of Theorem \ref{thm:extension}, we find forms $\hat A\in L^2\connforms{C_j^5}$, $\hat F=d\hat A+\hat A\wedge \hat A\in L^2\curvforms{C_j^5}$ and a gauge $\hat g\in W^{1,2}\gauges{C_j^5}$ which we rename as
\[
A_{C_j^5}:=\hat A,\quad A_{C_j^5}:=\hat F,\quad g_{C_j^5} = \hat g,
\]
and which satisfy the boundary conditions
\begin{equation}\label{boundary_5}
i_{\partial C_j^5}^*A_{C_j^5}=i_{\partial C_j^5}^*A,\quad i_{\partial C_j^5}^*F_{C_j^5}=i_{\partial C_j^5}^*F,\quad \left.g_{C_j^5}\right|_{\partial C_j^5}=g_{\partial C_j^5}.
\end{equation}
With the above notations the bounds \eqref{bounds_FA} then translate into the following, in which we used \eqref{agood_v} to absorb the connection contributions into curvature terms:
\begin{subequations}\label{bounds_FA_4}
\begin{eqnarray}
 \lefteqn{\|F_{C_j^5} - \overline{F}_{C_j^5}\|_{L^2(C_j^5)}\lesssim  \frac1{10}\sum_{C_\alpha^4\in\left(C_j^5\right)^{(4)}}\| i^*_{C_\alpha^4}F-\overline{F}_{C_\alpha^4}\|_{L^2(\partial C_\alpha^4)}}\label{AA_5}\\[3mm]
&&+\ \frac{1}{10}\verti{\overline F_{C_j^5}}\sum_{C_\alpha^4\in\left(C_j^5\right)^{(4)}}\|i^*_{C_\alpha^4}A - \overline{A}_{C_\alpha^4}\|_{L^2(\partial C_\alpha^4)}+\ \|F_{\partial C_j^5}\|_{L^{2}(\partial C_j^5)}^2 + \verti{\overline{F}_{C_j^5}}^2,\nonumber
\end{eqnarray}
\begin{eqnarray}
 \|A_{C_j^5} - \overline{A}_{C_j^5}\|_{L^2(C_j^5)}&\lesssim&\frac1{10}\sum_{C_\alpha^4\in\left(C_j^5\right)^{(4)}}\|i^*_{C_\alpha^4}A - \overline{A}_{C_\alpha^4}\|_{L^2(\partial C_\alpha^4)},\label{AAA_5}\\[3mm]
\|F_{C_j^5} \|_{L^{\frac{5}{2}}(C_j^5)}&\lesssim&\ \|F_{\partial C_j^5}\|_{L^{2}(\partial C_j^5)} +\ |\overline{F}_{C_j^5}|,\quad\label{Fbetter_5}\\[3mm]
\left(A_{C_j^5}\right)^{g_{C_j^5}}&\in& L^5(C_j^5, \wedge^1 \R^5\otimes \mathfrak g),\label{FbetterA_5}
\end{eqnarray}
\end{subequations}
where we also used the triangle inequality and the formulas \eqref{f_5-a_5}.

\medskip

By performing the above extension over the interiors of all $\delta^{(4)}$-good $5$-faces $C_j^5$, we conclude the first step of our iterative extension.

\medskip

\textit{Step 2.2': Preparation for the extension to $6$-faces.} We now fix a $6$-face $C_j^6$ then, due to the condition \eqref{boundary_5} on the $g_{C_j^5}$'s, we find that $g_{C_j^5}=g_{C_{j^\prime}^5}$ on $C_j^5\cap C_{j^\prime}^5$ for all $j,j^\prime$, thus we can apply Lemma \ref{lem:gluing_gauges} with $k\mapsto 5$, $C_\alpha^k\mapsto C_\alpha^5$, $C^{k+1}\mapsto C_j^6$, $g_\alpha\mapsto g_{C_\alpha^5}$ and $A\mapsto \sum_{\alpha} 1_{C_\alpha^5}A_{C_\alpha^5}$. With these choices the hypotheses of the lemma are valid due to property \eqref{Fbetter_5}. Then the lemma gives as an output a gauge $g_{\partial C_j^6}\in W^{1,2}\gauges{\partial C_j^6}$ and a connection form $A_{\partial C_j^6}\in L^2\connforms{\partial C_j^6}$ such that
\begin{equation}\label{agood_5}
\left(A_{\partial C_j^6}\right)^{g_{\partial C_j^6}}\in L^5\connforms{\partial C_j^6}, 
\end{equation}
which allows to start the next step in the iteration.

\medskip

\textit{Step 2.4: General step of the extension, from $k$-faces to $(k+1)$-faces for $5\le k \le n-1$.} Fix a $(k+1)$-face $C_j^{k+1}$. After the extension on $k$-faces we have for each $k$-face $C_\alpha^4\in\left(C^{k+1}_j\right)^{(k)}$ a connection form $A_{C_\alpha^k}\in L^2\connforms{C_\alpha^k}$ whose curvature form $F=dA_{C_\alpha^k}+A_{C_\alpha^k}\wedge A_{C_\alpha^k}$ satisfies $F\in L^2\curvforms{C_\alpha^k}$ and constant forms $\overline{A}_{C_\alpha^k}\in \wedge^1\R^k\otimes \mathfrak{g}$ and $\overline{F}_{C_\alpha^k}\in\wedge^2\R^k\otimes \mathfrak{g}$ such that the following bounds, generalizing \eqref{deltagood4} to faces of dimension $k> 4$, hold. The bounds are dilation-invariant, but we present them in the version valid at general scale $r$, for clarity:
\begin{align}\label{deltagoodk}
\verti{\overline{F}_{C_\alpha^k}}\le\delta^{(k)} r^{-2},\quad &\verti{\overline{A}_{C_\alpha^k}}\le\delta^{(k)} r^{-1}, \quad \int_{C_\alpha^k}|F_{C_\alpha^k}|^{\frac{k}2}\le \delta^{(k)}.
\end{align}
At this point again we may reduce to scale $r=1$ by dilation invariance. We present the justification of the above bounds \eqref{deltagoodk} for the case $k=5$ with $r=1$.
\begin{itemize}
\item The bounds on $\verti{\overline{F}_{C_\alpha^5}}$ and on $\verti{\overline{A}_{C_\alpha^5}}$ follow under the condition $\delta<\delta^{(5)}$ from the definition \eqref{f_5-a_5} and \eqref{deltagood4}, by triangle inequality, due to the fact that $C_j^{k+1}=C_j^6$ is in this a $\delta$-good $6$-face.
\item The bound on $\vertii{F_{C_\alpha^5}}_{L^{\frac52}(C_\alpha^5)}^2$ in \eqref{deltagoodk} follows from \eqref{FbetterA_5} provided $\delta^{(5)}\ge C_4\delta^{(4)}$, where $C_4$ is a combinatorial constant, which in this case can be taken to be equal to $441C_{\eqref{Fbetter_5}}^2$, where $C_{\eqref{FbetterA_5}}$ is the implicit constant appearing in \eqref{FbetterA_5}. Then \eqref{FbetterA_5} together with the bound on $\verti{\overline F_{C_\alpha^5}}$ already discussed above, and the bounds on $A_{\partial C_\alpha^5}$ and $F_{\partial C_\alpha^5}$ which come from \eqref{deltagood4} by triangle inequality. As the number of $4$-faces of $C_\alpha^5$ is $10$, we find indeed that 
\begin{eqnarray*}
\vertii{F_{C_\alpha^5}}_{L^{\frac52}(C_\alpha^5)}
&\le& C_{\eqref{FbetterA_5}} \left[\sum_{C_\beta^4\in\left(\partial C_\alpha^5\right)^{(4)}}\vertii{F_{C_\beta^4}}_{L^2(C_\beta^4)}+\sqrt{\delta^{(4)}}\right]\le 21C_{\eqref{FbetterA_5}}\sqrt{\delta^{(4)}}.
\end{eqnarray*}
\item The bound on $\vertii{A_{C^5_\alpha}}_{L^2(C_\alpha^5)}^2$ in \eqref{deltagoodk} follows using \eqref{AAA_5}, and the already-discussed bound for $\overline{A}_{C^5_\alpha}$. This term is bounded by $\delta^{(5)}\ge (C_{\eqref{AAA_5}}+1)^2\delta^{(4)}$ in which the constant $\delta^{(4)}$ comes from \eqref{deltagood4} and $C_{\eqref{AAA_5}}$ is the implicit constant appearing in \eqref{AAA_5}. The bound is as follows:
\begin{eqnarray*}
\vertii{A_{C^5_\alpha}}_{L^2(C_\alpha^5)}&\le&\vertii{A_{C_\alpha^5}- \overline{A}_{C_\alpha^5}}_{L^2(C_\alpha^5)} + \vertii{\overline{A}_{C_\alpha^5}}_{L^2(C_\alpha^5)}\\[3mm]
&\le& C_{\eqref{AAA_5}}\max_{C^4_\beta\in\left(\partial C_\alpha^5\right)^{(4)}}\vertii{i^*_{C_\beta^4}A-\overline{A}_{C_\beta^4}}_{L^2(C_\beta)} + \sqrt{\delta^{(4)}}\\[3mm]
&\le&(C_{\eqref{AAA_5}}+1)\sqrt{\delta^{(4)}}.
\end{eqnarray*}
\end{itemize}
For general $5\le k\le n-1$ we then assume that \eqref{deltagoodk} holds for all $C_\alpha^k\in\left(C_j^{k+1}\right)^{(k)}$ and by using the same reasoning as in Step 2.2', we define 
\begin{equation}\label{a_partial_k}
A_{\partial C_j^{k+1}}:=\sum_{C_\alpha^k\in\left(C_j^{k+1}\right)^{(k)}}1_{C_\alpha^k}A_{C_\alpha^k}, 
\end{equation}
whose associated curvature form is then 
\[
F_{\partial C_j^{k+1}}:=dA_{\partial C_j^{k+1}}+A_{\partial C_j^{k+1}}\wedge A_{\partial C_j^{k+1}}=\sum_{C_\alpha^k\in\left(C_j^{k+1}\right)^{(k)}}1_{C_\alpha^k}F_{C_\alpha^k}.
\]

By using Lemma \ref{lem:gluing_gauges} together with the conclusion from the preceding step of our iteration (which we proved to hold for $k=5$ and follows from \eqref{FbetterA_k} below for the case $k\mapsto k+1>5$), we find that there exists $g_{\partial C_j^{k+1}}\in L^2\gauges{\partial C_j^{k+1}}$ such that $\left(A_{\partial C_j^{k+1}}\right)^{g_{\partial C_j^{k+1}}}\in L^k\connforms{\partial C_j^{k+1}}$ and
\begin{equation}\label{agood_k}
\vertii{\left(A_{\partial C_j^{k+1}}\right)^{g_{\partial C_j^{k+1}}}}_{L^k(\partial C_j^{k+1})}\lesssim \vertii{F_{\partial C_j^{k+1}}}_{L^{\frac{k}2}(\partial C_j^{k+1})}.
\end{equation}
At this point we may apply Theorem \ref{thm:extension} with the following choices:
\begin{itemize}
\item $n\mapsto k$,
\item $D\mapsto C_j^{k+1}$ and $\Psi_D\mapsto\Psi_{C_j^{k+1}}$ as defined in Section \ref{sec:parametriz},
\item $\overline{F} \mapsto \overline{F}_{C^{k+1}_j}$ and $\overline{A} \mapsto \overline{A}_{C^{k+1}_j}$, where we denote
\begin{align}\label{f_k-a_k}
\overline{F}_{C^{k+1}_j}:=\frac{1}{2(k+1)}\sum_{C_\alpha^k\in\left(C^{k+1}_j\right)^{(k)}}\overline{F}_{C_\alpha^k},\nonumber\\ \overline{A}_{C^{k+1}_j}:=\frac{1}{2(k+1)}\sum_{C_\alpha^k\in\left(C^{k+1}_j\right)^{(k)}}\overline{A}_{C_\alpha^k}.
\end{align}
\item $A_{\partial D}\mapsto A_{\partial C_j^{k+1}}$ and $F_{\partial D}\mapsto F_{\partial C_j^{k+1}}$,
\item $g\mapsto g_{\partial C_j^{k+1}}$, described above.
\end{itemize}

We now verify again that the hypotheses of Theorem \ref{thm:extension} hold. We have $C_D\mapsto \max\{\op{Lip}(\Psi_k), \op{Lip}(\Psi_k^{-1})\}$, fixed depending only on $k$ in Section \ref{sec:parametriz}, and
\begin{itemize}
\item The condition \eqref{Agood} on $g\mapsto g_{\partial C_j^k}$ was justified in \eqref{agood_k}.
\item The bounds \eqref{Fgood} on $\overline{F}\mapsto \overline{F}_{C_j^k}$ and $\overline{A}\mapsto\overline{A}_{C_j^k}$ follow, by triangle inequality, from the definitions \ref{f_k-a_k} and by the bounds in \eqref{deltagoodk}, provided we have $\delta^{(k)}<\epsilon_0^{(k)}$ where $\epsilon_0^{(k)}$ is the value of $\epsilon_0$ appearing in Theorem \ref{thm:extension} if we chose $n\mapsto k$.
\item The bounds \eqref{Fgood} on $F\mapsto F_{\partial C_j^k}$ and on $A\mapsto A_{\partial C_j^k}$ follow from the analogous bounds that appear in \eqref{deltagoodk}, by triangle inequality, up to diminishing $\delta^{(k)}$ by a combinatorial factor of $2(k+1)$, equal to the number of $k$-faces of $C_j^{k+1}$.
\end{itemize}
By applying Theorem \ref{thm:extension}, we find forms $\hat A\in L^2\connforms{C_j^k}$, $\hat F=d\hat A+\hat A\wedge \hat A\in L^2\curvforms{C_j^k}$ and a gauge $\hat g\in W^{1,2}\gauges{C_j^k}$ which we rename as
\[
A_{C_j^{k+1}}:=\hat A,\quad A_{C_j^{k+1}}:=\hat F,\quad g_{C_j^{k+1}} = \hat g,
\]
and which satisfy the boundary conditions
\begin{equation}\label{boundary_k}
i_{\partial C_j^{k+1}}^*A_{C_j^{k+1}}=A_{\partial C_j^{k+1}},\quad i_{\partial C_j^{k+1}}^*F_{C_j^{k+1}}=F_{\partial C_j^{k+1}},\quad \left.g_{C_j^{k+1}}\right|_{\partial C_j^{k+1}}=g_{\partial C_j^{k+1}}.
\end{equation}
Then the bounds \eqref{bounds_FA} translate into:
\begin{subequations}\label{bounds_FA_k}
\begin{eqnarray}
 \lefteqn{\|F_{C_j^{k+1}} - \overline{F}_{C_j^{k+1}}\|_{L^2(C_j^{k+1})}\lesssim  \frac1{2(k+1)}\sum_{C_\alpha^k\in\left(C_j^{k+1}\right)^{(k)}}\| i^*_{C_\alpha^k}F-\overline{F}_{C_\alpha^k}\|_{L^2(\partial C_\alpha^k)}}\nonumber\\[3mm]
&&+\ \frac{1}{2(k+1)}\verti{\overline F_{C_j^{k+1}}}\sum_{C_\alpha^k\in\left(C_j^{k+1}\right)^{(k)}}\|i^*_{C_\alpha^k}A - \overline{A}_{C_\alpha^k}\|_{L^2(\partial C_\alpha^k)}\nonumber\\[3mm]
&&+\ \sum_{C_\alpha^k\in\left(C_j^{k+1}\right)^{(k)}}\left(\|F_{C_\alpha^k}\|_{L^{\frac{k}2}(C_\alpha^k)}^2 + \verti{\overline{F}_{C_\alpha^k}}^2\right),\label{AA_k}\nonumber
\end{eqnarray}
\begin{equation}
 \|A_{C_j^{k+1}} - \overline{A}_{C_j^{k+1}}\|_{L^2(C_j^{k+1})}\lesssim\frac1{2(k+1)}\sum_{C_\alpha^k\in\left(C_j^{k+1}\right)^{(k)}}\|i^*_{C_\alpha^k}A - \overline{A}_{C_\alpha^k}\|_{L^2(\partial C_\alpha^k)},\label{AAA_k}
\end{equation}
\begin{equation}
\|F_{C_j^{k+1}} \|_{L^{\frac{k+1}{2}}(C_j^{k+1})}\lesssim\sum_{C_\alpha^k\in\left(C_j^{k+1}\right)^{(k)}}\left(\|F_{C_\alpha^k}\|_{L^{\frac{k}2}(C_\alpha^k)} + |\overline{F}_{C_\alpha^k}|\right),\quad\label{Fbetter_k}
\end{equation}
\begin{equation}
\left(A_{C_j^{k+1}}\right)^{g_{C_j^{k+1}}}\in L^{k+1}(C_j^{k+1}, \wedge^1\R^{k+1}\otimes \mathfrak g),\label{FbetterA_k}
\end{equation}
\end{subequations}
where we also used the triangle inequality and the formulas \eqref{f_5-a_5}.

\medskip

\textit{Step 2.5: Final bound for $n$-faces.} In conclusion, the bounds \eqref{bounds_FA_k} together with the conditions \eqref{deltagoodk} for $k\ge 5$ and \eqref{deltagood4} for $k=4$, can be summed up to give the following:
\begin{lemma}[estimate on a good $n$-face]\label{lem:est_good_kface}
There exists a constant $\epsilon_0$, depending only on the dimension, for any cubeulation $P$ of scale $r=1$, for each $0<\epsilon<\epsilon_0$ there exists $\delta>0$, such that if \eqref{deltagood4} holds for such choice of $\delta$ for a given $n$-face $C_j^n$, then we may construct $A_{C_j^n}\in L^2\connforms{C_j^n}$ and $F_{C_j^n}:=dA_{C_j^n}+A_{C_j^n}\wedge A_{C_j^n}\in L^{\frac{n}2}\curvforms{C_j^n}$ which satisfy the bounds
\begin{subequations}\label{bounds_good_n}
\begin{eqnarray}
\lefteqn{\vertii{F_{C_j^n}-\overline{F}_{C_j^n}}_{L^2(C_j^n)}\lesssim\frac1{c_n} \sum_{C\in\left(C_j^n\right)^{(4)}}\vertii{i_C^*F - \frac1{\verti{C}}\int_Ci^*_CF}_{L^2(C)}}\nonumber\\
&&+\frac{1}{c_n}\verti{\overline{F}_{C_j^n}}\sum_{C\in\left(C_j^n\right)^{(4)}}\vertii{i_C^*A - \frac1{\verti{C}}\int_Ci^*_CA}_{L^2(C)}+\sum_{C\in\left(C_j^n\right)^{(4)}}\vertii{i^*_CF}_{L^2(C)}^2,\label{F-average}\\[3mm]
\lefteqn{\vertii{A_{C_j^n}-\overline{A}_{C_j^n}}_{L^2(C_j^n)}\lesssim\frac1{c_n} \sum_{C\in\left(C_j^n\right)^{(4)}}\vertii{i_C^*A - \frac1{\verti{C}}\int_Ci^*_CA}_{L^2(C)}}\label{A-average}
\end{eqnarray}
\end{subequations}
where $c_n:=2^{n-4}\left(\begin{array}{c} n\\4\end{array}\right)$, and there exists a gauge $g_{C_j^n}\in W^{1,2}\gauges{C_j^n}$ such that 
\begin{equation}\label{regulara}
\left(A_{C_j^n}\right)^{g_{C_j^n}}\in L^n\connforms{C_j^n}.
\end{equation}
\end{lemma}
%
\subsubsection{Extension on the bad skeleton}\label{sec:ext_bad}
%
We proceed by replacing $A$ and $F$ via an iterative precedure over the bad cubes. This will be performed via maps whose models, depending on the dimension $5\le k\le n$, are denoted as follows:
\begin{equation}\label{pi_k}
\pi^{(k)}:[-1,1]^k\setminus\{0\}\to \partial [-1,1]^k.
\end{equation}
The map $\pi^{(k)}$ is assumed to belong to $C^0([-1,1]^k\setminus\{0\})\cap C^\infty(]-1,1[\setminus\{0\})$ and to be equal to the identity on $\partial [-1,1]^k$. For the sake of concreteness, a possible explicit choice is $\pi^{(k)}(x):=\int \rho_{\op{dist}(x,\partial [-1,1]^k)}(x-y) \pi^{(k)}_\square(y)dy$ where $\pi^{(k)}_\square(x):=\tfrac{x}{\verti{x}_\square}$, $\verti{x}_\square:=\max_{1\le j\le k}\verti{x_i}$ and $\rho_\epsilon(x)=\epsilon^{-k}\rho(x/\epsilon)$ where $\rho$ is a smooth positive radial function of integral $1$ supported in $\{x\in\R^k: \verti{x}\le 1\}$.

\medskip

After composing with a suitable translation and rotation like in \S \ref{sec:parametriz}, we obtain smoothened radial projections 
\begin{equation}
\label{pi_ck}\pi_{C_j^k}:C_j^k\setminus\{c_j^k\}\to \partial C_j^k,
\end{equation}
where $c_j^k$ is the center of the $k$-face $C_j^k$.

\medskip

If now we consider the clopen set cover from Corollary \ref{cor:l4_on_4face}, we can then construct from it a clopen cover of the union of all $4$-faces, still denoted by $\{U_\alpha^{(4)}\}_\alpha$. We then extend the cover to higher-dimensional skeleta by defining iteratively 
\[
\left\{U_\alpha^{(k+1)}\right\}_\alpha:=\left\{\widetilde{U}_\alpha^{(k+1)}:=\left(\pi_{C_j^{k+1}}\right)^{-1}\left(U_\alpha^{(k)}\right): U_\alpha^{(k)}\subset \partial C_j^{k+1},\ C_j^{k+1}\in P_{r,t}^{(k+1)}\right\}.
\]
With the above notation relating $\widetilde{U}_\alpha^{(k+1)}$ to $U_\alpha^{(k)}$ and $C_j^{k+1}$, we also iteratively define gauges $g_{\widetilde{U}_\alpha^{(k+1)}}\in W^{1,2}\gauges{\widetilde{U}_\alpha^{(k+1)}}$ and connections $A_{C_j^{k+1}}\in L^2\connforms{C_j^{k+1}}$ by
\begin{equation}\label{bad_gA}
g_{\widetilde{U}_\alpha^{(k+1)}}:=g_{U_\alpha^{(k)}}\circ\pi_{C_j^{k+1}} \quad\mbox{and}\quad A_{C_j^{k+1}}:=\left(\pi_{C_j^{k+1}}\right)^*A_{\partial C_j^{k+1}},
\end{equation}
Where we use the same definition \eqref{a_partial_k} as in the case of good cubes to define $A_{\partial C_j^{k+1}}$. We also consider the inclusion $i_{\widetilde{U}_\alpha^{(k+1)}}:\widetilde{U}_\alpha^{(k+1)}\to C_j^{k+1}$ and define $A_{\widetilde{U}_\alpha^{k+1}}:= i_{\widetilde{U}_\alpha^{(k+1)}} A_{C_j^{k+1}}$. Then inductively from the condition described in Proposition \ref{prop:gauge_dirichlet} and Corollary \ref{cor:l4_on_4face}, we find that there holds
\begin{equation}\label{condition_badgauge}
g_{U_\alpha^{(k)}}=g_{U_\beta^{(k)}} \quad \mbox{on}\quad U_\alpha^{(k)}\cap U_\beta^{(k)} \quad\mbox{and}\quad \left\{\begin{array}{c}\left(A_{U_\alpha^{(k)}}\right)^{g_{U_\alpha^{(k)}}}\in L^4\connforms{U_\alpha^{(k)}}, \\[3mm]
i_{\partial C_j^{k+1}\cap \widetilde{U}_\alpha^{(k)}}^*A_{\widetilde{U}_\alpha^{(k+1)}}=A_{U_\alpha^{(k)}}.\end{array}\right.
\end{equation}
Moreover we have the bounds 
\begin{equation}\label{L2_bound_bad_k}
\vertii{A_{C_j^{k+1}}}_{L^2(C_j^{k+1})}\lesssim\vertii{A_{\partial C_j^{k+1}}}_{L^2(\partial C_j^{k+1})},\quad \vertii{F_{C_j^{k+1}}}_{L^2(C_j^{k+1})}\lesssim\vertii{F_{\partial C_j^{k+1}}}_{L^2(\partial C_j^{k+1})}.
\end{equation}
Considering the above bounds together with the definition of $A_{\partial C_j^{k+1}}$ given in \eqref{a_partial_k} and the triangle inequality, we find at stage $n$ that there exists a constant $c_n$ depending only on the implicit constants in \eqref{L2_bound_bad_k}, such that for each $n$-dimensional face $C_j^n$ there holds
\begin{subequations}\label{L2_bound_bad}
\begin{eqnarray}
\vertii{A_{C_j^n}}_{L^2(C_j^n)}&\le& c_n\sum_{C\in\left(C_j^n\right)^{(4)}}\vertii{i^*_CA}_{L^2(C)},\\[3mm]
\vertii{F_{C_j^n}}_{L^2(C_j^n)}&\le& c_n\sum_{C\in\left(C_j^n\right)^{(4)}}\vertii{i^*_CF}_{L^2(C)}.
\end{eqnarray}
\end{subequations}
As a final step, after we have performed all our iterative extensons over all bad cubes, we describe more precisely what is the set over which the gauges and connection forms are not defined, which is also the set which remains not covered by the open sets $\{U_\alpha^{(n)}\}_\alpha$.
Before we proceed to the next lemma, we need some definitions.
\begin{definition}[dual skeleta]\label{def:dual_skel}
Let $r>0$, $t\in \R^n$ and let $\mathcal S^{(k)}$ be a cube complex of dimension $k$, generated by the cubes of the form
\begin{equation}\label{cube_typical}
C_j^k:=C(r,J,c):=r[-1,1]^J + c, \quad J\subset\{1,\ldots,n\}, \quad \verti{J}=k,\quad c \in \left(r\mathbb Z\right)^n + t.
\end{equation}
Consider a subcomplex $\mathcal T^{(k^\prime)}$ of $\mathcal S^{(k)}$, of dimension $k^\prime\le k$. Then the \emph{dual complex} to $\mathcal T^{(k^\prime)}$ inside $\mathcal S^{(k)}$ is the complex $\mathcal U^{(k-k^\prime)}$ formed by cubes that can be written as
\[\overline{C}_l^h=C(r,J_l,c_l),\quad \verti{J_l}\le k-k^\prime, \quad c_l\in \left(r\mathbb Z\right)^d + t,
\]
and such that there exists a cube $C(r,J,c_l)\in \mathcal S^{(k)}$ with $\verti{J}=k$ and such that the cube $C(r,J\setminus J_l,c_l)$ belongs to $\mathcal T^{(k^\prime)}$.
\end{definition}
For our fixed $\delta$ as in \S \ref{sec:ext_good} we denote by
\begin{equation}\label{complex_bad}
P_\mathrm{bad}^{(k)}:=\mbox{subcomplex of }P_{r,t}\mbox{ generated by all }\delta\mbox{-bad faces},
\end{equation}
and as we produce the extensions over (the support of) $P_\mathrm{bad}^{(k)}$ of our connections via the iterative extension by the $\pi_{C_j^{k^\prime}}$, $5\le k^\prime\le k$ as in \eqref{pi_ck}, we define the singular set introduced in this way by
\[
\Sigma^{(k)}:=P_\mathrm{bad}^{(k)}\setminus \bigcup_\alpha U_\alpha^{(k)}.
\]
For $A\subset \R^n, x\in \R^n$ we also denote by 
\[
\mathsf{Cone}(A,x):=\{(1-t)a+tx:\ a\in A, t\in [0,1]\}.
\]
\begin{lemma}[structure of the introduced singular set]\label{lem:singular}
The set $\Sigma^{(k)}\cap P^{(k)}_\mathrm{bad}$ is the support of the dual complex in $P^{(k)}_\mathrm{bad}$ to the one formed by the bad faces of dimension $5\le k^\prime\le k$.
\end{lemma}
\begin{proof}
The proof will proceed by induction:
\begin{itemize}
\item Initially, we have that all the $4$-skeleton of bad cubes is covered by the sets $\{U_\alpha^{(4)}\}_\alpha$. In this case there are no singular points introduced and $\Sigma^{(4)}=\emptyset$, as desired. 
\item After the first step of the iterative extension the part $\Sigma^{(5)}$ of the $5$-skeleton which is not covered by the $\{U_\alpha^{(5)}\}_\alpha$ consists of a $0$-dimensional set, formed by the centers of all bad $5$-dimensional faces. 
\item Then, assume by induction that before extending to $k+1$-bad cubes the part $\Sigma^{(k)}$ of the $k$-dimensional skeleton which is not covered is the dual complex, inside the $k$-skeleton, of the skeleton of $5$-dimensional bad faces. Then consider a $k$-dimensional face $C^k$ and denote its two $(k+1)$-dimensional neighboring faces by $C^{k+1}_+$ and $C^{k+1}_-$ and their centers by $c_+$ and $c_-$, respectively. We note that, by the symmetry of the extension maps $\pi_{C_\pm^{k+1}}$ that we use, we have 
\begin{multline}\label{ext_sigma}
\Sigma^{(k+1)}\cap \left(\mathsf{Cone}(C^k, c_+) \cup \mathsf{Cone}(C^k,c_-)\right) \\
= \mathsf{Cone}(\Sigma^{(k)}\cap C^k, c_+) \cup \mathsf{Cone}(\Sigma^{(k)}\cap C^k,c_-).
\end{multline}
The fact that $\Sigma^{(k+1)}\cap P^{(k+1)}_\mathrm{bad}$ is the dual of the bad complex then follows, by iterating the above construction for all $(k+1)$-dimensional faces.
\end{itemize}
\end{proof}
%
%
\subsubsection{Mollification and completion of the proof}\label{sec:mollif}
%
For the last stage of our construction of approximants, we have already filled all the good $n$-cubes of our cubeulation with extensions as in \S \ref{sec:ext_good} and the bad $n$-cubes with extensions as in \S \ref{sec:ext_bad}.

\medskip

Together with the connection forms, we have been extending also the local gauges, outside the codimension-$5$ set $\Sigma^{(n)}$ from Lemma \ref{lem:singular}. In such gauges our connection forms have locally $L^4$-coefficients: at the base step of the iterative extension this follows from Corollary \ref{cor:l4_on_4face} of \S \ref{sec:localconn}. Then in the iteration this property is preserved by \eqref{regulara} by the $k\mapsto n$ case of \eqref{condition_badgauge}.

\medskip

The following result, analogous to \cite[Lem. 2.4]{PRym}, shows that if in compatible local gauges we have $L^4$-integrable connection forms, then mollifying the coefficient forms of the connection forms provides smooth approximants in our desired norms.
\begin{lemma}[mollification and local $L^4$-gauges]\label{lem:mollif}
Assume that $\Omega\subset \R^n$ is a compact set with open interior. Let $A\in L^2\connforms{\Omega}$ and $F:=dA+A\wedge A\in L^2\curvforms{\Omega}$. If  $\{U_\alpha\}_\alpha$ is a cover of an open set $U\subset \R^n$ such that for each $U_\alpha$ there exists $g_\alpha\in W^{1,2}\gauges{U_\alpha}$ such that $\left (\left.A\right|_{U_\alpha}\right)^{g_\alpha}\in L^4\connforms{U_\alpha}$ and $g_\alpha =g_\beta$ over $U_\alpha\cap U_\beta$ for all $\alpha, \beta$, then there exists a sequence $A_n\in L^2\connforms{U}$ such that for all $\alpha$ there holds
\begin{subequations}\label{smoothing}
\begin{equation}\label{smooth_conn}
\forall U_\alpha\quad \left(A_n|_{U_\alpha}\right)^{g_\alpha}\in C^\infty\connforms{U_\alpha}\quad\mbox{and}\quad \vertii{A_n-A}_{L^2(U)}\to 0.
\end{equation}
and furthermore if $F_n:=dA_n + A_n\wedge A_n$ then we have
\begin{equation}\label{smooth_curv}
\vertii{F_n-F}_{L^2(U)}\to 0.
\end{equation}
\end{subequations}
\end{lemma}
\begin{proof}
We fix a sequence $\eta_n\downarrow0$ for the rest of the proof. For any connection form $\widetilde{A}\in L^2\connforms{U_\alpha}$, we define its smoothing by
\[
\widetilde{A}_{\eta_n}(x):=\int \rho_{\min\{\eta_n,\op{dist}(x,\partial U_\alpha)\}}(x-y)\widetilde{A}(y)dy.
\]
where $(\rho_\epsilon)_{\epsilon>0}$ is a family of mollifiers with $\rho_1$ supported on the ball of radius 1.

\medskip 

Then we apply the above for $A_n$ by fixing a partition of unity $\{\theta_\alpha\}_\alpha$ subordinated to the cover $\{U_\alpha\}_\alpha$. We define
\[
A_n:=\sum_\alpha\theta_\alpha\left(\left(\left(\left.A\right|_{U_\alpha}\right)^{g_\alpha}\right)_{\eta_n}\right)^{g_\alpha^{-1}}.
\]
We can check directly by the properties of mollification that the smoothness condition in \eqref{smooth_conn} holds. The convergence required in \eqref{smooth_conn} follows then by triangle inequality from the formula $A^g=g^{-1}dg+g^{-1}Ag$ using the fact that $g_\alpha^{-1}$ is in $W^{1,2}\cap L^\infty$ and
\[
\vertii{\left(\left(A|_{U_\alpha}\right)^{g_\alpha}\right)_{\eta_n} - \left(A|_{U_\alpha}\right)^{g_\alpha}}_{L^2(U_\alpha)}\to 0.
\]
We next prove \eqref{smooth_curv}. In this case we may write
\begin{eqnarray*}
F_n&=&\sum_\alpha \theta_\alpha g_\alpha F_{\left(\left(A|_{U_\alpha}\right)^{g_\alpha}\right)_{\eta_n}} g_\alpha^{-1} + \sum_\alpha d\theta_\alpha\left(\left(\left(A|_{U_\alpha}\right)^{g_\alpha}\right)_{\eta_n}\right)^{g_\alpha^{-1}}\\
&&+\sum_\alpha(\theta_\alpha^2-\theta_\alpha) g_\alpha\left(\left(A|_{U_\alpha}\right)^{g_\alpha}\right)_{\eta_n}\wedge\left(\left(A|_{U_\alpha}\right)^{g_\alpha}\right)_{\eta_n} g_\alpha^{-1}\\
&&+\sum_{\alpha\neq\beta}\theta_\alpha\theta_\beta\left(\left(\left(A|_{U_\alpha}\right)^{g_\alpha}\right)_{\eta_n}\right)^{g_\alpha^{-1}}\wedge\left(\left(\left(A|_{U_\beta}\right)^{g_\beta}\right)_{\eta_n}\right)^{g_\beta^{-1}},\\
F&=&\sum_\alpha \theta_\alpha g_\alpha F_{\left(A|_{U_\alpha}\right)^{g_\alpha}} g_\alpha^{-1} + \sum_\alpha d\theta_\alpha\left(\left(A|_{U_\alpha}\right)^{g_\alpha}\right)^{g_\alpha^{-1}}\\
&&+\sum_\alpha(\theta_\alpha^2-\theta_\alpha) g_\alpha\left(A|_{U_\alpha}\right)^{g_\alpha}\wedge\left(A|_{U_\alpha}\right)^{g_\alpha}g_\alpha^{-1}\\
&&+\sum_{\alpha\neq\beta}\theta_\alpha\theta_\beta\left(\left(A|_{U_\alpha}\right)^{g_\alpha}\right)^{g_\alpha^{-1}}\wedge\left(\left(A|_{U_\beta}\right)^{g_\beta}\right)^{g_\beta^{-1}},
\end{eqnarray*}
and we can bound $\vertii{F_n-F}_{L^2(U)}$ by controlling the following quantities.
\begin{subequations}\label{curv_mollif_terms}
From the above first terms, using the fact that $\vertii{F_{A_1^g} - F_{A_2^g}}_{L^2}=\vertii{F_{A_1}+ F_{A_2}}_{L^2}$, summed over $\alpha$, and with the lighter notation $\tilde A:=(A|_{U_\alpha})^{g_\alpha}$ and $\eta:=\eta_n$, we have the terms
\begin{eqnarray}
\lefteqn{\vertii{F_{\tilde A_\eta}-F_{\tilde A}}_{L^2(U_\alpha)}\le \vertii{F_{\tilde A_\eta} - \left(F_{\tilde A}\right)_\eta}_{L^2(U_\alpha)} + \vertii{\left(F_{\tilde A}\right)_\eta - F_A}_{L^2(U_\alpha)}}\nonumber\\
&=& \vertii{\tilde A_\eta\wedge \tilde A_\eta - \left(\tilde A\wedge \tilde A\right)_\eta}_{L^2(U_\alpha)} + \vertii{\left(F_{\tilde A}\right)_\eta - F_{\tilde A}}_{L^2(U_\alpha)}\nonumber\\
&\le&\vertii{\tilde A}_{L^4(U_\alpha)}\vertii{\tilde A_\eta - \tilde A}_{L^4(U_\alpha)} + \vertii{\left(F_{\tilde A}\right)_\eta - F_{\tilde A}}_{L^2(U_\alpha)}
\end{eqnarray}
and the two terms converge to zero as $\eta\to 0$. 

\medskip

Next, we have, again summed over $\alpha$ and with the notations above,
\begin{eqnarray}
\vertii{d\theta_\alpha\left(\left(\tilde A_\eta\right)^{g_\alpha^{-1}}-\left(\tilde A\right)^{g_\alpha^{-1}}\right)}_{L^2(U_\alpha)}&\le&\vertii{d\theta_\alpha}_{L^\infty}\vertii{g_\alpha\left(\tilde A_\eta - \tilde A\right)g_\alpha^{-1}}_{L^2(U_\alpha)}\nonumber\\
&\le&\vertii{d\theta_\alpha}_{L^\infty}\vertii{\tilde A_\eta - \tilde A}_{L^2(U_\alpha)},
\end{eqnarray}
and 
\begin{multline}\vertii{(\theta_\alpha^2-\theta_\alpha)g_\alpha\left(\tilde A_\eta\wedge \tilde A_\eta - \tilde A\wedge \tilde A\right)g_\alpha^{-1}}_{L^2(U_\alpha)}\\
\le\vertii{\theta^2_\alpha-\theta_\alpha}_{L^\infty(U_\alpha)}\vertii{\tilde A}_{L^4(U_\alpha)}\vertii{\tilde A_\eta-\tilde A}_{L^4(U_\alpha)}
\end{multline}
and finally with the further notation $\check A:=(A|_{U_\beta})^{g_\beta}$ and with a required sum also over $\beta$, we have the terms
\begin{eqnarray}
\lefteqn{\vertii{\theta_\alpha\theta_\beta\left(\tilde A_\eta^{g_\alpha^{-1}}\wedge\check A_\eta^{g_\beta^{-1}} -\tilde A^{g_\alpha^{-1}}\wedge\check A^{g_\beta^{-1}}\right) }_{L^2(U_\alpha)}}\nonumber\\
&\le&\vertii{\theta_\alpha\theta_\beta}_{L^\infty(U_\alpha)}\left(\vertii{\tilde A_\eta-\tilde A}_{L^4(U_\alpha)}\vertii{\check A}_{L^4(U_\alpha)}\right.\nonumber\\
&&\quad\quad\left.+\vertii{\tilde A}_{L^4(U_\alpha)}\vertii{\check A_\eta-\check A}_{L^4(U_\alpha)}\right).
\end{eqnarray}
Then summing all the terms \eqref{curv_mollif_terms} over $\alpha$ and $\beta$, and using the convergence proved in the previous step as well, we see that $\vertii{F_n-F}_{L^2(U)}\to 0$ too, as desired.
\end{subequations}
\end{proof}
We now are ready to conclude the proof of Theorem \ref{thm:strongapprox}. By working first on the grid $P_{r,t}$ rescaled to scale $r=1$, we consider the final result of performing the extensions as in Lemma \ref{lem:est_good_kface} on all the good cubes and the ones leading to \eqref{L2_bound_bad} on all the bad cubes. Then we rescale back to scale $r=1$.
We denote by $\widetilde{A}_r$, $\widetilde{F}_r$ the connection and curvature forms obtained in this way. 

\medskip

Next, we consider the piecewise constant forms defined inductively as the averages \eqref{f_k-a_k} scaled back to scale $r$, and thus, with $c_n=2^{n-4}\tfrac{n!}{4!(n-4)!}$, 
\begin{eqnarray}\label{f_n-a_n}
\overline{F}_r(x)&:=&\sum_{C_j^n\in P_{r,t}}1_{C_j^n}(x)\frac1{c_n r^4}\sum_{C\in \left(C_j^n\right)^{(4)}}\int_C i^*_CF,\nonumber\\[3mm]
\overline{A}_r(x)&:=&\sum_{C_j^n\in P_{r,t}}1_{C_j^n}(x)\frac1{c_n r^4}\sum_{C\in \left(C_j^n\right)^{(4)}}\int_C i^*_CA.
\end{eqnarray}
On the set of good cubes we scale to scale $r$ and then sum up the conclusions \eqref{bounds_good_n} of Lemma \ref{lem:est_good_kface}. We find, denoting, like in \eqref{complex_bad}, $P_\mathrm{good}^{(n)}, P_\mathrm{bad}^{(n)}\subset P_{r,t}$ to be the subcomplexes generated by the good cubes and the bad cubes, respectively,
\begin{subequations}\label{final_good}
\begin{eqnarray}
\lefteqn{\vertii{\widetilde{F}_r-F}_{L^2(P_\mathrm{good})}^2\lesssim\vertii{\widetilde{F}_r-\overline{F}_r}_{L^2(P_\mathrm{good})}^2+\vertii{\overline{F}_r-F}_{L^2(P_\mathrm{good})}^2}\nonumber\\
&\lesssim&r^{n-4}\sum_{C\in P_\mathrm{good}^{(4)}}\vertii{i_C^*F - \frac1{\verti{C}}\int_Ci^*_CF}_{L^2(C)}^2\nonumber\\
&&+r^{n-2}\sum_{C\in P_\mathrm{good}^{(4)}}\left(\sum_{C_j^n:\ C\in\left(C_j^n\right)^{(4)}}\verti{\overline{F}_{C_j^n}}^2\right)\vertii{i_C^*A - \frac1{\verti{C}}\int_Ci^*_CA}_{L^2(C)}^2\nonumber\\
&&+r^{n-4}\sum_{C\in P_\mathrm{good}^{(4)}}\vertii{i^*_CF}_{L^2(C)}^4+ \vertii{\overline{F}_r-F}_{L^2([-1,1]^n)}^2,\label{F-good}
\end{eqnarray}
\begin{eqnarray}
\lefteqn{\vertii{\widetilde{A}_r-A}_{L^2(P_\mathrm{good})}^2\lesssim\vertii{\widetilde{A}_r-\overline{A}_r}_{L^2(P_\mathrm{good})}^2+\vertii{\overline{A}_r-A}_{L^2(P_\mathrm{good})}^2}\nonumber\\
&\lesssim&r^{n-4}\sum_{C\in P_\mathrm{good}^{(4)}}\vertii{i_C^*A - \frac1{\verti{C}}\int_Ci^*_CA}_{L^2(C)}^2+ \vertii{\overline{A}_r-A}_{L^2([-1,1]^n)}^2.\label{A-good}
\end{eqnarray}
\end{subequations}
Using the bounds \eqref{kskeletonaverage} valid for our choice of $P_{r,t}$ for $A$ and $F$ contemporarily, we then find $\vertii{\widetilde{F}_r - \overline{F}_r}_{L^2([-1,1]^n)}\to 0$ and $\vertii{\widetilde{A}_r - \overline{A}_r}_{L^2([-1,1]^n)}\to 0$ as $r\to 0$, and the first terms on the right in \eqref{F-good} and \eqref{A-good} tend to zero as well. Concerning the forelast line in \eqref{F-good}, the assumed bound on $\int_{C^\prime}\verti{i_{C^\prime}^*\left(A-\overline A_{C_j^n}\right)}^2$ in \eqref{deltagood4}, together with the inversion of the order of summation and the Cauchy-Schwartz inequality, implies that
\begin{multline*}
r^{n-2}\sum_{C\in P_\mathrm{good}^{(4)}}\left(\sum_{C_j^n:\ C\in\left(C_j^n\right)^{(4)}}\verti{\overline{F}_{C_j^n}}^2\right)\vertii{i_C^*A - \frac1{\verti{C}}\int_Ci^*_CA}_{L^2(C)}^2\\
\lesssim \delta^{(n)}r^n\sum_{C_j^n\in P_\mathrm{good}^{(n)}}\verti{\int_{C_j^n}F}^2\le\delta^{(n)}\int_{P_\mathrm{good}}\verti{F}^2\le\delta^{(n)}\vertii{F}_{L^2([-1,1]^n)}^2,
\end{multline*}
and by a similar estimate using again \eqref{deltagood4}, we find for the remaining term in \eqref{F-good} that
\begin{equation*}
r^{n-4}\sum_{C\in P_\mathrm{good}^{(4)}}\vertii{i^*_CF}_{L^2(C)}^4\le \delta^{(n)}\int_{P_\mathrm{good}}\verti{F}^2\le\delta^{(n)}\vertii{F}_{L^2([-1,1]^n)}^2.
\end{equation*}
Thus we found
\begin{subequations}\label{est_good}
\begin{eqnarray}
\vertii{\widetilde{F}_r-F}_{L^2(P_\mathrm{good})}^2&\lesssim& \delta^{(n)}\vertii{F}_{L^2([-1,1]^n)}^2 + o_{r\to0}(1),\\
\vertii{\widetilde{A}_r-A}_{L^2(P_\mathrm{good})}^2&=& o_{r\to0}(1).
\end{eqnarray}
\end{subequations}
For the bad cubes we use the bounds \eqref{L2_bound_bad} and we may apply \eqref{kskeletonbound} to the forms $1_{P_\mathrm{bad}}A$ and $1_{P_\mathrm{bad}}F$ to obtain
\begin{subequations}\label{est_bad}
\begin{eqnarray}
\int_{P_\mathrm{bad}}\verti{\widetilde{A}_r}^2 &=&\sum_{C_j^n\in P_\mathrm{bad}^{(n)}}\vertii{A_{C_j^n}}_{L^2(C_j^n)}^2\lesssim \sum_{C\in\left(P_\mathrm{bad}\right)^{(4)}}\int_C\verti{i_C^*A}^2\lesssim \int_{P_\mathrm{bad}}\verti{A}^2,\quad\quad\\[3mm]
\int_{P_\mathrm{bad}}\verti{\widetilde{F}_r}^2 &=&\sum_{C_j^n\in P_\mathrm{bad}^{(n)}}\vertii{F_{C_j^n}}_{L^2(C_j^n)}^2\lesssim \sum_{C\in\left(P_\mathrm{bad}\right)^{(4)}}\int_C\verti{i_C^*F}^2\lesssim \int_{P_\mathrm{bad}}\verti{F}^2,\quad\quad
\end{eqnarray}
\end{subequations}
Now note that the bound implied \eqref{boundbadcubes} on the total measure of $P_\mathrm{bad}$:
\begin{equation}\label{volume_Pbad}
\verti{P_\mathrm{bad}}=r^nN_{\delta^{(n)}}\stackrel{\eqref{boundbadcubes}}{\lesssim}\frac{r^4}{\delta^{(n)}}\vertii{F}_{L^2([-1,1]^n)}^2+\frac{r^2}{\left(\delta^{(n)}\right)^2}\vertii{A}_{L^2([-1,1]^n)}^2+\frac{r^4}{\delta^{(n)}},
\end{equation}
thus if $\left(\delta^{(n)}\right)^{-1}r\to 0$ then we find that the right hand sides of the equations \eqref{est_bad} also tend to zero, by dominated convergence. In particular, by summing up \eqref{est_good} and \eqref{est_bad}, we find that given two sequences of positive numbers $\delta^{(n)}_\ell\to 0$ and $r_\ell\to 0$, there holds
\begin{equation}\label{est_all}
\begin{array}{c}\vertii{\widetilde{F}_{r_\ell} -F}_{L^2([-1,1]^n)}\to 0,\quad\mbox{and}\quad\vertii{\widetilde{A}_{r_\ell} -A}_{L^2([-1,1]^n)}\to 0,\\[5mm]
\mbox{provided}\quad \delta^{(n)}_\ell\to 0 \quad\mbox{and}\quad \frac{r_\ell}{\delta^{(n)}_\ell}\to 0.
\end{array}
\end{equation}
We then apply the mollification as in Lemma \ref{lem:mollif} to such choices of $\widetilde{A}_{r_\ell}$, $\widetilde{F}_{r_\ell}$, and the desired smooth approximants are constructed, completing the proof of Theorem \ref{thm:strongapprox}.
%
\section{Strong compactness for weak connections}\label{sec:strcpt}
%
Our aim is to prove the following theorem:
\begin{theorem}[\textbf{sequential weak closure of $\widetilde{\mathcal A}_G$}]\label{thm:wclos}
Let $A_j\in \widetilde{\mathcal A}_{G}([-1,1]^n)$ be a sequence of connections such that the corresponding curvature forms $F_j$ are equibounded in $L^2$ and converge weakly in $L^2$ to a $2$-form $F$. Then $F$ corresponds to $A\in\widetilde{\mathcal A}_{G}([-1,1]^n)$, and furthermore there holds
\begin{equation}
\tilde\delta(A_j,A)\to 0,
\end{equation}
where for $A,B\in\widetilde{\mathcal A}_G([-1,1]^n)$, the pseudo-distance $\tilde\delta(A,B)$ is defined in \eqref{tilde_delta}.
\end{theorem}
We recall that here again like for the pseudo-distance $\delta$ defined in \eqref{def_delta}, for $G=SU(2)$ the pseudo-distance $\tilde\delta$ on $\widetilde{\mathcal A}_G([-1,1]^n)$ induces a distance on gauge-equivalence classes of connections from this space, as a consequence of Corollary \ref{cor:gauge-dist}, due to the fact that $\tilde\delta$ is equivalent to $\delta_1$ as defined in \eqref{bard2}.

\medskip

The above result is mainly due to Uhlenbeck in dimensions $n\le 4$ and it is one of the main results in \cite{PRym} for $n=5$. We aim here at proving it by induction on $n$, and thus we first describe how the proof of the $n=4$ in \cite{PRym} allows our definition of $\widetilde{\mathcal A}_G([-1,1]^4)$ based on the $L^4$ norm rather than the $W^{1,2}$ norm, and then we use the theorem's statement for dimension $n-1$ in order to prove it in dimension $n$.

\medskip

As it will be more befitting to the overall proof, we frame the result in terms of an abstract compactness theorem which is the tool allowing the induction on the dimension. Versions of the same tool were successfully used for proving results in the theory of metric currents and in the one of scans, see \cite{AK}, \cite{HR1}, \cite{DPHscan} and the references therein.
%
%
\subsection{An abstract compactness result}\label{sec:abstractthm}
%
We employ as an abstract tool Proposition \ref{prop:abstractconv} below, which is the multi-dimensional substitute of the abstract compactness result used in \cite[Prop. 3.1]{PRym}. The H\"older continuity which was used in \cite{PRym} now does not hold for more general slicings, and thus we need a different approach. The natural candidate is a metric space valued Sobolev embedding theorem, inspired by \cite[Thm. 1.13]{KS}. The difference between our case and \cite{KS} is that the metric space in which our sliced connections take values is not locally compact, unlike what assumed in \cite{KS}, thus the coercivity of the Yang-Mills energy has to be used, like in \cite{PRym} and \cite[Thm. 9.1]{HR1}.

\medskip

We find it useful to introduce, following the spirit in which in \cite{AK} the notion of metric-$BV$-functions was used in the proof of compactness by slicing, the following notion of metric upper gradient (which extends the definition \cite[\S 3]{ADS} to the case of metric-space-valued maps). 

\medskip 

Recall that for $p>1$ the $p$-modulus $\mathrm{Mod}_p(\Gamma)$ of a family $\Gamma$ of absolutely continuous curves $\gamma:[a,b]\to \R^n$ is defined by
\[
\mathrm{Mod}_p(\Gamma):=\inf\left\{ \int_{\R^n} f^p :\ f:\R^d\to[0,\infty]\mbox{ Borel, }\int_\gamma f \ge 1\mbox{ for all }\gamma\in\Gamma\right\}.
\]
We recall that $\mathrm{Mod}_p$ is an outer measure on absolutely continuous curves, and we say that a property holds for $p$-a.e. curve if the set of curves for which it fails is $\mathrm{Mod}_p$-negligible. We will use the following definition:
\begin{definition}[metric upper gradient structures]\label{uppergr}
Let $(\mathcal Y,\op{dist})$ be a metric space and $p>1$. Consider a measurable map $f:\R^n\to\mathcal Y$ and a map $\mathcal N:\mathcal Y\to \R$. Further, if $v\in AC([0,1],X)$ is an absolutely continuous curve, then let $\verti{\gamma^\prime}(t)=\lim_{s\to t}\tfrac{d_X(\gamma(s),\gamma(t))}{\verti{s-t}}$ be its metric derivative. We say that $\mathcal N$ gives a \emph{$p$-upper gradient structure} for $f$ if for $p$-almost every curve $\gamma$ we have that $\mathcal N\circ \gamma$ is Borel and 
\begin{equation}\label{wugs}
\op{dist}(f\circ \gamma(s), f\circ \gamma(t))\le \int_s^t \mathcal N\circ \gamma(r)\verti{\gamma^\prime}(r)dr, \quad \forall\ 0 <s\le t<1.
\end{equation}
\end{definition}
Next, we state the following abstract compactness result:
\begin{proposition}\label{prop:abstractconv}
Consider a metric space $\mathcal Y$ and let $K:=[-1,1]^n$. Suppose that the function $\mathcal N:\mathcal Y\to \mathbb R^+$ satisfies the following condition:
\begin{equation}\label{hypoth}
\forall C>0\text{ the sublevels }\{\mathcal N\leq C\}\text{ are compact in }\mathcal Y\ .
\end{equation}
Suppose that $f_j:K\to \mathcal Y$ are measurable maps such that $\mathcal N$ gives a $p$-weak upper gradient structure for the $f_j$
and that
\begin{equation}\label{boundednorm}
\sup_j \int_K\left(\mathcal N\circ f_j\right)^p<C\ .
\end{equation}
Then $f_j$ have a subsequence which converges pointwise almost everywhere to a function $f:K\to \mathcal Y$ for which $\mathcal N$ gives a $p$-weak upper gradient structure for $f$, and such that
\begin{equation}\label{bdnorm_limit}
\int_K\left(\mathcal N\circ f\right)^p<C\ .
\end{equation}
Moreover, there holds, up to passing to the above subsequence,
\begin{equation}\label{dist_conv}
\int_K\mathrm{dist}^p\left(f_j(x), f(x)\right) dx\to 0\quad \mbox{as}\quad j\to\infty.
\end{equation}

\end{proposition}
\begin{rmk}
Note that in Proposition \ref{prop:abstractconv}, we don't assume the metric space $\mathcal Y$ to be complete or separable.
\end{rmk}
\begin{proof}
By \eqref{hypoth}, for each $\epsilon>0$ we may find a countable $\epsilon$-net of $\mathcal Y^\prime:=\mathcal Y \cap \{\mathcal N<\infty\}$, i.e. $N_\epsilon\subset \mathcal Y^\prime$ is a countable set such that $\min_{q\in \mathcal Y^\prime,q^\prime\in N_\epsilon}\op{dist}(q,q^\prime)\le \epsilon$. 

\medskip

Next, consider the functions
\[
d_{j,a}(x):=\op{dist}\left(f_j(x),a\right),\quad\mbox{for}\quad a\in \mathcal Y.
\]
We then note that, by triangle inequality,
\begin{equation}\label{dist_lip}
\verti{d_{j,a}(x)-d_{j,a}(y)}\le \op{dist}\left(f_j(x),f_j(y)\right),
\end{equation}
therefore the function $\mathcal N$ also gives a common $p$-weak upper gradient structure for the functions $d_{j,a}:K\to \R^+$. We obtain that by \eqref{boundednorm} and due to the fact that subsevels of $\mathcal N$ are compact, there exists a point $x_0\in K$ such that up to subsequence $f_j(x_0)$ forms a $\mathrm{dist}$-Cauchy sequence, converging to $y\in \mathcal Y$. Then due to \eqref{boundednorm} and \eqref{dist_lip}, we see that all the functions $d_{j,a}$ are bounded in $W^{1,p}(K,\R)$, by the same proof as in \cite[Thm. 7.6]{heinonen}, and thus by Rellich embedding and a diagonal extraction, we find a subsequence (denoted still by $j$, by abuse of notation) and maps $d_a\in W^{1,p}(K, \R)$ for $a\in\cup_{j\ge 1}N_{1/j}$ such that
\begin{equation}\label{converg_distances}
d_{j,a}(x)\to d_a(x), \quad\mbox{for}\quad \forall a\in\cup_{j\ge 1}N_{1/j}.
\end{equation}
Next, we claim that for all $x$ such that \eqref{converg_distances} holds, there exists a unique point $f(x)\in \mathcal Y$ such that 
\begin{equation}\label{find_limit}
\forall a\in\cup_{j\ge 1}N_{1/j}, \quad\mbox{there holds}\quad d_a(x)=\op{dist}(f(x),a).
\end{equation}
To prove \eqref{find_limit}, we note that by definition of $N_{1/j}$, for all $x$ as in \eqref{converg_distances} and all $j$, there exists $a_j\in N_{1/j}$ such that $\op{dist}(f_j(x), a_j)=d_{j,a_j}(x)<1/j$. Due to \eqref{converg_distances}, we also find that $d_{j,a_j}(x)$ is Cauchy. By triangle inequality, $a_j(x)$ forms a Cauchy sequence, thus it has a limit in $\mathcal Y$, and it converges to a point $f(x)\in\hat{\mathcal Y}$, where $\hat{\mathcal Y}$ is the completion of $\mathcal Y$. Now by \eqref{boundednorm} and Fatou's lemma, we find that for a.e. $x\in K$ the sequence $f_j(x)$ has a subsequence $j^\prime(x)$ depending on $x$ so that $\sup_{j^\prime(x)}\mathcal N(f_{j^\prime}(x))<\infty$. Then the hypothesis \eqref{hypoth} implies that $j^\prime(x)$ has a subsequence which converges to a point in $\mathcal Y$. But as we saw, all limit points in $\hat{\mathcal Y}$ coincide with $f(x)$, in particular $f(x)\in \mathcal Y$, proving \eqref{find_limit}.

\medskip

The function $x\mapsto f(x)$ is clearly measurable, and by construction $f_j(x)\to f(x)$ for a.e. $x\in K$. Property \eqref{hypoth} implies the lower semicontinuity of $\mathcal N$ and thus we find that $\mathcal N$ is also a weak upper gradient for $f$ and that \eqref{boundednorm} then gives \eqref{bdnorm_limit}.

\medskip

In order to obtain the property \eqref{dist_conv} we then use the pointwise convergence and conclude by dominated convergence via $\mathrm{dist}\left(f_j(x), f(x)\right)\le d_{j,a}(x)+d_a(x)$, using the fact that $d_{j,a}$ and $d_a$ are bounded in $L^p(K)$, which implies the $L^p$-convergence from \eqref{dist_conv}.
\end{proof}
%
\subsection{Scheme of proof of the closure theorem}\label{sec:schemeproof}
%
For applying the Proposition \ref{prop:abstractconv}, we use the following specializations:
\begin{itemize}
\item The well known \cite{DoKr} \textit{geometric distance} on $1$-forms: for 
\[
A,A'\in L^2\connforms{[-1,1]^k}\quad\mbox{with}\quad 4\le k \le n\ ,
\]
then we define the pseudo-distance 
\begin{equation}\label{dist_k}
  \op{Dist}_k(A,A'):=\min\{\|A-g^{-1}dg - g^{-1}A'g\|_{L^2([-1,1]^k)}:\:g\in W^{1,2}\gauges{[-1,1]^k}\}
\end{equation}
and we define the \emph{equivalence relation} $\sim_k$ on $L^2\connforms{[-1,1]^k}$ according to which $A\sim_k A'$ if $\op{Dist}_k(A,A')=0$. The facts that $\op{Dist}_4$ satisfies reflexivity and triangle inequality (and that as a consequence $\sim_k$ is an equivalence relation) follow from the fact that $W^{1,2}\gauges{[-1,1]^k}$ is a group (for which see \cite[Appendix]{Isobecrit}).
\item On the quotient $L^2\connforms{[-1,1]^k}/\sim_k$ the pseudo-distance $\op{Dist}_k$ induces a distance between $\sim_k$-equivalence classes which we denote by $\op{dist}_k$. We denote the so-obtained metric spaces by
\begin{equation}\label{Y}
 \mathcal Y_k:=(\mathcal A_G([-1,1]^k)/\sim_k, \op{dist}_k)\ .
\end{equation}
\item Let $[A]$ denote the $\sim_k$-equivalence class of a given $A\in \widetilde{\mathcal A}_G([-1,1]^k)$, namely the set of of all $A'\in \widetilde{\mathcal A}_G([-1,1]^k)$ such that $A'=g^{-1}dg+g^{-1}Ag$ for $g\in W^{1,2}\gauges{[-1,1]^k}$.
\item Like in \cite{PRym} we will study the functional 
\begin{equation}\label{N}
 \mathcal N_4:\mathcal Y_4\to\mathbb R^+,\quad \mathcal N_4([A])=\int_{[-1,1]^4}|F_A|^2\ .
\end{equation}
Note that because the curvature satisfies $F_{g^{-1}dg + g^{-1}Ag}=g^{-1}F_Ag$ and since the norm on $\mathfrak g$ is $\op{ad}G$-invariant, we have that $\mathcal N_4([A])$ does not depend on the representative $A$ employed to compute $F_A$.
\item The $f_j:[-1,1]^{n-4}\to \mathcal Y_4$ will be $4$-dimensional sliced connection forms corresponding to a sequence of connection forms $A_j\in\mathcal A_G([-1,1]^n)$, defined as follows, with the notation of \S \ref{sec:goodbad}. We fix a multi-index $I=\{i_1,\ldots,i_{n-4}\}$ and for $T\in[-1,1]^I$ we define
\begin{equation}\label{fj}
\tilde f_j(T):= i_{H(I,T)}^*A_j, \quad f_j(T):=[\tilde f_j(T)]\ .
\end{equation}
Then the $\tilde f_j$ take a.e. values in $\widetilde{\mathcal A}_G([-1,1]^{n-4})$ by the definition \eqref{eqdefag} of $\widetilde{\mathcal A}_G([-1,1]^n)$. If $A\sim_n A'$ then we find that for a.e. $T\in [-1,1]^I$ the $H(I,T)$-trace of the differential of the gauge $g\in W^{1,2}([-1,1]^n, G)$ relating $A, A'$ is defined and $L^2$-integrable, and $g|_{H(I,T)}$ relates $i_{H(I,T)}^*A$ to $i_{H(I,T)}^*A'$, thus $f_j$ is well-defined up to negligible sets.

\medskip

The assertion that the weak limit $A$ of the $A_j$ has $H(I,T)$-slices in $\widetilde{\mathcal A}_G([-1,1]^4)$ for all $i$ and almost every $T$ is equivalent to the thesis of the theorem \ref{thm:wclos}.
\end{itemize} 
We note that the pseudo-distance $\tilde\delta(A,B)$ between local weak connections in $\widetilde{\mathcal A}_G([-1,1]^n)$ defined in \eqref{tilde_delta}, can now be rewritten in terms of the distance \eqref{dist_k} as follows:
\begin{equation}\label{def_bar_delta}
\tilde\delta(A,B)=\max_{\substack{I\subset\{1,\ldots,n\}\\ \#I=n-4}}\int_{[-1,1]^I}\mathrm{Dist}_4\left(i_{H(I,T)}^*A,i_{H(I,T)}B\right)^2\mathrm{d} T.
\end{equation}
%
%
\subsection{The compactness in dimension $4$}
%
For $n\ge 5$ the the compactness in $\mathcal Y_n$ of sublevels of $\mathcal N_n$ is precisely the compactness result which we desire to prove. Since we know that it holds for $n=4$ we may proceed as for the closure theorem for rectifiable chains, and prove it by induction on $n$, assuming that it's true for $n-1$.
\begin{proposition}\label{verifh1}
Let $\mathcal Y_4$ and $\mathcal N_4$ be as above. Then $\mathcal N_4$ has sublevels which are compact with respect to the distance $\op{dist}_4$ defined in \eqref{Y}.
\end{proposition}
\begin{proof}[modification of the proof of \cite{PRym} Prop. 3.3]
The difference between the definition of $\widetilde{\mathcal A}_G([-1,1]^4)$ defined in \eqref{ag4d} and the version used in \cite[Prop. 3.3]{PRym} is that here a local gauges $g$ such that $A^g\in L^4$ are assumed to exist, rather than ones such that $A^g\in W^{1,2}$.

\medskip

The way in which such hypothesis is used in \cite{PRym} Prop. 3.3 is however just via ss theorem in regions where the $L^2$ norm of $F$ is small. Theorem \ref{coulstrange} for $n=4, \pi=0$ however works under the hypothesis that such $A^g\in L^4$ locally and then we obtain 
\[
d^*A^g=0\text{ and }dA^g=F_{A^g}-A^g\wedge A^g\in L^2\ ,
\]
which implies that $A\in W^{1,2}$ by Hodge inequality. This reduces us to the situation of \cite{PRym} Prop. 3.3, and the rest of the proof follows like in that proposition.
\end{proof}
%
\subsection{The Yang-Mills energy gives a weak gradient structure}
%
We provide a new version of \cite{PRym} Lem. 3.4 for the case of parallel slices instead of spherical slices. The main ingredient is a new version of \cite{PRym} Coroll. 1.13 which we now state:
\begin{lemma}[controlled solutions to the gauge fixing ODE]\label{lem:wsolode}
Assume that to $A\in\widetilde{\mathcal A}_G([-1,1]^n)$ and fix $g_0\in W^{1,2}([-1,1]^{n-1}\times\{0\}, G)$. Then there exists a solution $g\in W^{1,2}([-1,1]^n, G)$ to the following ODE, where $A=\sum_iA_idx_i$:
\begin{equation}\label{ODE}
 \left\{\begin{array}{ll}
        \partial_n g= -A_n g&\quad\text{on }[-1,1]^n\ ,\\
        g(x',0)=g_0(x')&\quad\text{for }x'\in [-1,1]^{n-1}\ .
        \end{array}
\right.
\end{equation}
In particular the form $A^g:=g^{-1}dg+g^{-1}Ag$ is $L^2$-integrable and has zero component in the direction $\partial/\partial x_n$. Moreover we have 
\[
\vertii{g}_{W^{1,2}([-1,1]^n}\lesssim \vertii{A}_{L^2([-1,1]^n)}+\vertii{F\llcorner e_n}_{L^2([-1,1]^n)}.
\]

\end{lemma}
\begin{proof}
By Theorem \ref{thm:strongapprox} applied to the cube $[-1,1]^n$ we have a sequence of connections $A_j\in\mathcal R^\infty([-1,1]^n)$ such that
\[
 A_j\to A\text{ in }L^2,\quad F_{A_j}\to F_A\text{ in }L^2\ .
\]
We then solve, with notation $A_j=\sum_i\left(A_j\right)_idx_i$,
\begin{equation}\label{ode0}
 \left\{\begin{array}{ll}
       \partial_n g_j(x',t)=-\left(A_j\right)_n(x',t) g_j(x',t)&\quad\text{for }, t\in[-1,1],\ x'\in[-1,1]^{n-1}\ ,\\[3mm]
       g_j(x',0)=g_0(x')&\quad\text{for }x'\in[-1,1]^{n-1}\ ,
        \end{array}
\right.
\end{equation}
where the solution $g_j$ is well defined on all segments $x'=const$ except for the ones which contain one of the singular set $\Sigma_j$ of $A_j$. The union of all such segments is negligible, therefore $g_j$ is defined almost everywhere. We have the following, with the further notation $F_{A_j}^{g_j}:=F_j:= \sum_{a<b}\left(F_j^{g_j}\right)_{ab}dx_a\wedge dx_b$ and for indices $i\in\{1,\ldots, n-1\}$:
\begin{eqnarray}
\left. \left(A_j^{g_j}\right)_i\right|_{x_n=0} &=& \left(A_j^{g_0}\right)_i|_{x_n=0}, \ ,\label{ode1}\\[3mm]
\left(A_j^{g_j}\right)_n&\stackrel{\eqref{ode0}}{=}&0  \ ,\label{ode2}\\[3mm]
\left(F_j^{g_j}\right)_{ni}&=&\partial_n \left(A_j^{g_j}\right)_i - \partial_i \left(A_j^{g_j}\right)_n + \left[\left(A_j^{g_j}\right)_n, \left(A_j^{g_j}\right)_i\right]\nonumber\\
&\stackrel{\eqref{ode2}}{=}&\partial_n \left(A_j^{g_j}\right)_i\ ,\label{ode3}\\[3mm]
\partial_i g_j&=& g_j\left(A_j^{g_j}\right)_i - \left(A_j\right)_i\ g_j\ .\label{ode4}
\end{eqnarray}
Integrating \eqref{ode1}, \eqref{ode3} we find that $\left(A_j^{g_j}\right)_i, i>1$ are $L^2$-integrable with bounds depending on $\|F\llcorner e_n\|_{L^2}$ only, thus we find 
\begin{equation}\label{estgk}
 \|g_j\|_{W^{1,2}([-1,1]^n)}\lesssim \|A_j\|_{L^2([-1,1]^n)}+\|F_j\llcorner e_n\|_{L^2([-1,1]^n)}\leq C\\ .
\end{equation}
Up to extracting a subsequence we may assume
\[
 g_j\rightharpoonup g\quad\text{weakly in }W^{1,2}
\]
and thus $g_j\to g$ a.e. and strongly in all $L^p,p<\infty$ by interpolation between $L^{2^*}$ and $L^\infty$. 
From this, the rest of the reasoning proceeds precisely like for \cite[Cor. 1.13]{PRym}, as this convergence allows to conclude the proof by approximation.
\end{proof}
The possibility to solve an ODE such as \eqref{ODE} allows to proceed to the proof of the second hypothesis that $\mathcal N$ gives a $2$-weak upper gradient structure for the slices $f_j$, as required for the application of Proposition \ref{prop:abstractconv}. This is done by the following two Corollaries. The first result is obtained by just applying Lemma \ref{lem:wsolode} along a curve:
\begin{corollary}\label{cor:wsolode}
Let $n\ge 5$ and let $I\subset\{1,\ldots,n\}$ of cardinality $4$ and let $J:=\{1,\ldots,n\}\setminus I$. Assume that $A\in\widetilde{\mathcal A}_G([-1,1]^n)$ and an injective rectifiable curve $\gamma:[0,1]\to[-1,1]^J$ such that for almost all $t\in[0,1]$ and for $t=0,1$ the slices $i^*_{H(J,\gamma(t))}A\in L^2\connformsr{H_{J,\gamma(t)}}{\R^I}$ are well-defined and satisfy the curvature bound
\[
\int_0^1\vertii{i_{H(J,\gamma(t))}^*F}^2_{L^2(\{\gamma(t)\}\times[-1,1]^I)}\verti{\dot\gamma(t)}dt< \infty.
\]
Fix $g_0\in W^{1,2}\gauges{\{\gamma(0)\}\times [-1,1]^4}$. Then there exists a solution $g_\gamma\in W^{1,2}\gauges{\gamma([0,1])\times [-1,1]^I}$ to the ODE
\begin{equation}\label{solode_gen}
\left\{\begin{array}{lll}
       \partial_t g_\gamma(\gamma(t),x)=-A[\dot\gamma(t)](\gamma(t),x)\ g_\gamma(\gamma(t),x)& \ \ \mbox{for }t\in[0,1],&\ \ x\in[-1,1]^I,\\[3mm]
       \left.g_\gamma\right|_{H(I,\gamma(0))}= g_0.
       \end{array}
\right.
\end{equation}
Moreover we have that the component $\left(A^{g_\gamma}\right)_{\dot\gamma(t)}=0$ for $t\in [0,1]$ and 
\begin{multline}\label{solode_gen1}
\vertii{g_\gamma}_{W^{1,2}(\gamma([0,1])\times[-1,1]^I)}\\
\lesssim \vertii{A}_{L^2(\gamma([0,1])\times[-1,1]^4)}+\left(\int_0^1\vertii{i^*_{H(J,\gamma(t))}F}_{L^2(\{\gamma(t)\}\times[-1,1]^I)}^2 \verti{\dot\gamma(t)}d t\right)^\frac12.
\end{multline}
\end{corollary}
The next Corollary can be viewed as an adaptation to the current setting (translated now in the language of weak upper gradient structures, for clarity) of the study done in the abelian case in \cite{petrache2014notes}:
\begin{corollary}[The curvature gives a $2$-weak upper gradient structure for the slices]\label{cor:2wugs}
Let $n\ge 5$ and let $I\subset\{1,\ldots,n\}$ of cardinality $n-4$. Assume that $A\in\widetilde{\mathcal A}_G([-1,1]^n)$. Then for 
\[
f:[-1,1]^I\to \mathcal Y_4\quad\mbox{defined by}\quad f_I(T)=\left[i^*_{H(I,T)}A\right],
\]
the function $\mathcal N_4:\mathcal Y_4\to \R^+$ gives a $2$-weak upper gradient structure for $f$.
\end{corollary}
\begin{proof}
We first find $T_0$ such that the slice $i_{H(I,T_0)}^*A$ is well-defined and we may start with $g_0\equiv id\in W^{1,2}\gauges{H(I,T_0)}$, then apply Corollary \ref{cor:wsolode} to extend $g_0$ to $g\in W^{1,2}\gauges{H(I,T_0)}$ such that $A^\prime:= g^{-1}dg + g^{-1}Ag$ satisfies \eqref{solode_gen}. Then we find that $F^\prime:=F_{A^\prime}$ satisfies for $i\in I, j\in J:=\{1,\ldots,n\}\setminus I$, similarly to \eqref{ode3},
\begin{equation}\label{curv_conn}
\left(F_{A^\prime}\right)_{i,j}= \partial_i\left(A^\prime\right)_j.
\end{equation}
Then by the definition of the distance $\op{dist}_4$ of $\mathcal Y_4$, and by the expression of the curvature as the gradient of an $L^2$-function on the space $[-1,1]^J$ given in \eqref{curv_conn}, for a rectifiable curve $\gamma:[a,b]\to [-1,1]^J$ such that all terms below are finite there holds, using definition \eqref{dist_k},
\begin{eqnarray}\label{weak_grad_f}
\op{Dist}_4\left(i_{H(I,\gamma(a))}^*A^\prime, i_{H(I,\gamma(b))}^*A^\prime\right) &\le&\vertii{i_{H(I,\gamma(a))}^*A^\prime - i_{H(I,\gamma(b))}^*A^\prime}_{L^2([-1,1]^4)}\nonumber\\
&\le&\int_a^b\vertii{\nabla_Ti^*_{H(I,\gamma(t))}A^\prime}_{L^2([-1,1]^4)}\verti{\dot\gamma(t)}dt\nonumber\\
&\le&\int_a^b\vertii{i^*_{H(I,\gamma(t))}F^\prime}_{L^2([-1,1]^4)}\verti{\dot\gamma(t)}dt,
\end{eqnarray}
where $\nabla_T$ represents the gradient taken in the $T$-variables, belonging to $[-1,1]^J$. The bound \eqref{weak_grad_f} coincides with the inequality that is required in \eqref{wugs}, for $f, \mathcal N$ as in the statement of the lemma.

\medskip 

Now we can follow the reasoning from \cite[\S 7]{heinonen} valid for gradients of $W^{1,2}$-functions, in order to obtain that the same bound \eqref{weak_grad_f} also holds for $2$-a.e. curve $\gamma:[a,b]\to[-1,1]^J$, concluding the proof. 
\end{proof}
As a direct consequence of Corollary \ref{cor:2wugs}  applied to the $f_j$ defined as in the beginning of the section (see \eqref{fj}), and of Proposition \ref{verifh1}, we have that the hypotheses of Proposition \ref{prop:abstractconv} hold for $\mathcal N\mapsto \mathcal N_4$ and $p\mapsto 2$.
%
\subsection{Proof of the Closure Theorem \ref{thm:wclos}}
%
We first note the following lemma, analogous to \cite[Lem. 3.5]{PRym}:
\begin{lemma}[{cf. \cite[Lem. 3.5]{PRym}}]\label{lem:weaksl}
Let $n\ge 5$, $I\subset\{1,\ldots,n\}$ of cardinality $n-4$. Let $A_j\in\overline{A}_G([-1,1]^n)$, and consider the gauges $g_j(I)$ as given in Corollary \ref{cor:wsolode}. Assume that $\sup_j\vertii{F_{A_j}}_{L^2([-1,1]^n)}\le C$ and that
\begin{equation}\label{weakconv}
 \left(A_j\right)^{g_j(I)}\rightharpoonup A(I)\quad\text{ weakly in }L^2\connforms{[-1,1]^n}\ .
\end{equation}
Then there exists a subsequence $j'$ such that 
\begin{equation}\label{slicewiseweak}
 \text{for a.e. }T\in[-1,1]^I\text{ there holds }i_{H(I,T)}^*\left(A_{j'}\right)^{g_{j'}(I)}\rightharpoonup i_{H(I,T)}^*A(I)\text{ weakly in }L^2\ .
\end{equation}
\end{lemma}
The proof follows roughly the same method as the one of \cite[Lem. 3.5]{PRym}, but with several changes, including the use of weak upper gradient structures, and therefore we present it in full.
\begin{proof}

\medskip

We denote, for $T\in [-1,1]^I$, by $A_j(T):=i_{H(I,T)}^*\left(A_j\right)^{g_j(I)}$. We again consider a test form, now of the form $\beta:=\omega_T\wedge\phi:=\left(i_{H(I,T)}^*\omega\right)\wedge \phi$, with $\omega\in L^2([-1,1]^4, \wedge^3\R^4\otimes\mathfrak{g})$ and $\phi\in C^\infty([-1,1]^I, \wedge^{n-4}\R^I\otimes\mathfrak{g})$, and we define 
\begin{equation}\label{fjomega}
f_j^\omega(T):=\int_{[-1,1]^4}A_j(T)\wedge \omega_T,
\end{equation}
and we find from Corollary \ref{cor:wsolode} that the maps $f_j^\omega:[-1,1]^I\to \R$ have a $2$-weak upper gradient structure given by $A\mapsto \vertii{F_A}_{L^2([-1,1]^4)}\vertii{\omega}_{L^2([-1,1]^4)}$, and due to the assumed bound on $\vertii{F_{A_j}}_{L^2([-1,1]^n)}$, we may apply the abstract result of Proposition \ref{prop:abstractconv} to obtain the thesis.
\end{proof}
The above lemma allows to complete the proof of Theorem \ref{thm:wclos} proceeding precisely like for \cite[Thm. 1.11]{PRym}.

\medskip

\begin{proof}[End of proof of Theorem \ref{thm:wclos}:] We work under the hypothesis of the theorem, and we consider the global weak limit connection of the $A_j$'s, and denote it by $A\in L^2\connforms{[-1,1]^n}$. 

\medskip

Fix first $I\subset\{1,\ldots,n\}$ of cardinality $n-4$ and first apply Proposition \ref{prop:abstractconv}, to the slice functions $f_j$ as defined in \S \ref{sec:schemeproof}: we find that pointwise a.e. $T\in[-1,1]^I$, up to subsequence the sliced connection equivalence classes as defined in \S \ref{sec:schemeproof} $\left[i_{H(I,T)}^*A_j\right]$ converge in $\mathcal Y_4$ and that there holds, due to \eqref{dist_conv}, that for some forms $A(I,T)\in \overline{\mathcal A_G}([-1,1]^4)$ for $T\in[-1,1]^I$, there holds
\begin{equation}\label{directionwise_slice}
\int_{[-1,1]^I}\mathrm{Dist}_4^2\left(i^*_{H(I,T)}A_j, A(I,T)\right)dT\to 0\quad\mbox{as}\quad j\to\infty.
\end{equation}
Now we apply Corollary \ref{cor:wsolode}, and Lemma \ref{lem:weaksl}, and find that in the $g_j(I)$-gauges up to yet another subsequence, the sliced connection forms converge as in \eqref{slicewiseweak}.

\medskip 

We consider a sequence $A_j\in\widetilde{\mathcal A}_{G}([-1,1]^n)$ as in Theorem \ref{thm:wclos}. For $I\subset\{1,\ldots,n\}$ of cardinality $n-4$, we may find a change of gauge $g_j(I)$ as described in Corollary \ref{cor:wsolode}. Then we have in particular, due to \eqref{solode_gen}, \eqref{solode_gen1}, that
\begin{eqnarray}
 \|A_j^{g_j(I)}\|_{L^2([-1,1]^n)}&\leq& C\|F_j\|_{L^2([-1,1]^n)}\ ,\label{bound_agi}\\[3mm] \vertii{g_j(I)}_{W^{1,2}([-1,1]^n)}&\leq& C \left(\verti{A_j}_{L^2([-1,1]^n)}+\|F_j\|_{L^2([-1,1]^n)}\right).\label{bound_gi}
\end{eqnarray}
We thus have that up to extracting a subsequence there holds 
\begin{equation}\label{weakconvan}
 A_j^{g_j(I)}\rightharpoonup A(I)\text{ in }L^2\connforms{[-1,1]^n},\quad g_j(I)\rightharpoonup g(I)\text{ in }W^{1,2}\gauges{[-1,1]^n}.
\end{equation}
We claim that if we denote $A(I),A(J)$ and $g(I), g(J)$ the above limit connection forms and gauges for two sets of coordinates $I\neq J\subset\{1,\ldots,n\}$ of cardinality $n-4$, we have, for
\begin{equation}\label{gij_formula}
g(IJ):=g(I)^{-1}g(J),
\end{equation}
that then 
\begin{equation}\label{indepslice}
\left(A(I)\right)^{g(IJ)}=A(J).
\end{equation}
To see this, we introduce the notation $g_j(IJ):=\left(g_j(I)\right)^{-1}g_j(J)$ and we find a $W^{1,2}$-bound for $g_j(IJ)$ similar to \eqref{bound_gi}, as follows. In order to bound $\partial_\alpha g_j(IJ)$ we separately consider the cases (a) $\alpha\in I\cup J$ -- in which case we assume up to exchanging the roles of $I,J$ that $\alpha \in I$-- and (b) $\alpha\notin I\cup J$. In the case (a) we use
\[
\left(\left(A_j\right)^{g_j(J)}\right)_\alpha = g_j(IJ)^{-1}\left(\left(A_j\right)^{g_j(I)}\right)_\alpha g_j(IJ) + g_j(IJ)^{-1}\partial_\alpha g_j(IJ),
\]
and using the bounds \eqref{bound_agi}, \eqref{bound_gi} and \eqref{gij_formula}, we find that $\verti{\partial_\alpha g_j(IJ)}$ is controlled by $L^2$-integrable quantities. For the case (b), take a third index $\tilde I$ containing $\alpha$ and use the cocycle condition $g_j(IJ)=g_j(I\tilde I)g_j(\tilde IJ)$, valid due to \eqref{gij_formula}:
\[
\partial_\alpha g_j(IJ) = \partial_\alpha\left(g_j(I\tilde I)g_j(\tilde IJ)\right)=\partial_\alpha g_j(I\tilde I)g_j(\tilde IJ)+g_j(I\tilde I)\partial_\alpha g_j(\tilde IJ).
\]
By triangle inequality, we thus reduce to case (a). Thus 
\begin{equation}\label{bound_gij}
\vertii{g_j(IJ)}_{W^{1,2}([-1,1]^n)}\lesssim \vertii{F_j}_{L^2([-1,1]^n)}.
\end{equation}
Since we are assuming that the right-hand side of \eqref{bound_gij} is bounded, we find that $g_j(IJ)$ is bounded in $W^{1,2}$, and therefore we can extract a subsequence that converges weakly in $W^{1,2}$ to a limit $g(IJ)$. The relation \eqref{gij_formula} also passes to the limit, and we find that \eqref{indepslice} holds.  

\medskip

Because, by \eqref{gij_formula} and \eqref{indepslice}, the connection forms $A(I)$ obtained as weak limits for different indices $I$ as above are connected by the gauges $g(IJ)$, we find that these connection forms come from a global connection form, which is gauge-equivalent to the weak limit $A$. 

\medskip

Combining the outcome of the last two paragraphs, we find that for all $I$ for almost all $T\in[-1,1]^I$ the classes of slices $\left[i_{H(T,I)}^*A\right]$ of the weak limit belong to $\mathcal Y_4$, and thus for such $T$ we find $i_{H(T,I)}^*A\in\overline{A}_G([-1,1]^4)$, as desired. Moreover by construction $i^*_{H(T,I)}A$ is gauge equivalent to $A(I,T)$ obtained in \eqref{directionwise_slice} and therefore we have also $\tilde\delta(A_j,A)\to 0$ as $j\to 0$, as desired.
\end{proof}

\section{The case of general base manifolds}\label{sec:global}
%
In this section we extend the strong closure and compactness results of Theorems \ref{thm:strongapprox} and \ref{thm:wclos} to the results stated in the introduction in Theorems \ref{thm:strongapprox_intro} and \ref{thm:weakclosure}, respectively, where the base space is a general Riemannian manifold $(M^n,h)$ rather than the Euclidean cube $[-1,1]^n$ and where the slices we take of our connection forms are by regular levelsets of general functions $f\in C^\infty(M^n,\mathbb R^{n-4})$, like in Definition \ref{def:weakconn}. 
%
\subsection{Locality and $C^1$-invariance of the space of weak connections}
%
We start by noting that our definition of space of weak connections is localizable, and that it is robust under perturbation by regular diffeomorphisms, and even by bi-Lipschitz homeomorphisms.

\medskip

In fact more generally the structures we study are also invariant under perturbation by bilipschitz transformations, but for this paper we concentrate on regular manifolds $M^n$, for which such more general statement is not needed. The question about what is the lowest regularity assumption on $M_n$ which allows to prove the closure mentioned in Theorem \ref{thm:wclos} is left for future work.
\begin{lemma}[Localization of $\mathcal A_G(M^n)$]\label{lem:localization}
Let $U_\alpha, \alpha\in I$ be an atlas of a compact $n$-dimensional Riemannian manifold $(M^n,h)$. Then the following hold:
\begin{enumerate}
\item If a differential form $A\in L^2\connforms{M^n}$ is such that its restriction $A|_{U_\alpha}$ to each $U_\alpha$ is a weak connection, $A|_{U_\alpha}\in\mathcal A_G(U_\alpha)$, then $A\in\mathcal A_G(M^n)$.
\item If $I$ is finite, then there exists a constant $C>0$ depending only on $(M^n,h)$ such that if $A,A'\in L^2\connforms{M^n}$ are such that, with the notation of Definition \ref{def:weakconn}, for each $\alpha\in I$ we have $\delta(A|_{U_\alpha},A'|_{U_\alpha})<\infty$ in $\mathcal A_G(U_\alpha)$, then $\delta(A,A')<\infty$ and there holds
\begin{equation}\label{equiv_chart_dist}
C^{-1}\sum_{\alpha\in I}\delta\left(A|_{U_\alpha}, A'|_{U_\alpha}\right)\le \delta(A,A')\le C\sum_{\alpha\in I}\delta\left(A|_{U_\alpha}, A'|_{U_\alpha}\right).
\end{equation}
\end{enumerate}
\end{lemma}
As the proof reasoning is rather standard we only indicate the overall reasoning, omitting the details.
\begin{proof}[Sketch of proof:]
For the point (i) note that, indeed, if $M^n$ is compact, then the bounds on the distributional curvature forms $dA|_{U_\alpha}+A|_{U_\alpha}\wedge A|_{U_\alpha}$ imply the corresponding bound on $dA+A\wedge A$, whereas the slice condition from Definition \ref{def:weakconn} holding on each $U_\alpha$ implies that it also holds globally on $M^n$.

\medskip

For point (ii), we may proceed by classical compactness methods, and note that one may pass from $f_\alpha\in C^\infty(U_\alpha,\mathbb R^{n-4}), \alpha\in I$ to $f\in C^\infty(M^n,\mathbb R^{n-4})$ by restriction or by using partitions of unity, conserving information about the local structure of the levelsets.
\end{proof}
\begin{lemma}[Invariance under $C^1$-diffeomorphisms]\label{lem:C1perturbation}
If $\Psi:\Omega\to \Omega'$ is a $C^1$-diffeomorphism with $\Omega,\Omega'\subset\mathbb R^n$, then we claim that $\Psi$ establishes a correspondence between $\mathcal A_G(\Omega)$ and $\mathcal A_G(\Omega')$ in the sense that
\begin{enumerate}
\item there holds
\begin{equation}\label{lipinvar}
\mathcal A_G(\Omega)=\lf\{A\in L^2(\Omega,\wedge^1\R^n\otimes{\mathfrak g})\ : \exists A'\in \mathcal A_G(\Omega'), A=\Psi^*A'\rg\},
\end{equation}
\item there exists $C>0$ depending only on $(M^n,h)$ and on the bi-Lipschitz constant of $\Psi$, such that 
\begin{equation}\label{lipinvar_delta}
C^{-1}\delta(\Psi^*A,\Psi^*A')\le \delta(A,A')\le C\delta(\Psi^*A,\Psi^*A').
\end{equation}
\end{enumerate}
\end{lemma}
Note that for an $L^2$-form $A'$ and $\Psi$ Lipschitz, the form $\Psi^*A'$ is well-defined in $L^2\connforms{\Omega}$. 
\begin{proof}
If $S^4$ is a generic embedded submanifold, in $\Omega'$, then the slice $i_{S^4}^*A'$ of $A'$ are transferred to slices of $A$ by the $C^1$-submanifold $\Psi(S^4)$, defined by $\Psi^*i_{S^4}^*A'$. If $\Psi$ is a $C^1$-diffeomorphism, then these slices are along $C^1$ submanifolds, as the ball boundaries appearing in Definition \ref{def:weakconn}. We consider the case of $S^4$ from now on, the other case being treated similarly.

\medskip

We may use $\Psi^*$ and composition with $\Psi$ applied to $\tilde A$ and $g$, respectively, to transfer the equations $\tilde A^g=g^{-1}dg+ g^{-1}\tilde Ag$ to $\Omega$ in the case of $\tilde A$ equal to $i_{S^4}^*A'$. The fact that $\vertii{D\Psi}_{L^\infty}, \vertii{D\Psi^{-1}}_{L^\infty}<C$ shows that bounds on $g\in W^{1,2}\gauges{U_\alpha}$ defined locally on elements of a good cover $\{U_\alpha\}$ of such slices $S^4$ can, by chain rule, be transferred to $g\in W^{1,2}\gauges{\Psi^{-1}(U_\alpha)}$, which form a good cover of $\Psi^{-1}(S^4)$. Thus the version of Definition \ref{def:weakconn} as indicated in the discussion following that definition, holds for $\mathcal A_G(\Omega)$ as defined by the right-hand side in \eqref{lipinvar}, as claimed in point (i).

\medskip

For proving point (ii), we compose $f$ from Definition \ref{def:weakconn} with $\Psi$ or $\Psi^{-1}$, and use the fact that taking the infimum in \eqref{def_delta} over $f\in C^\infty$ or over $f\in C^1$ does not change its value.
\end{proof}
%
%
\subsection{Proof of the compactness theorem for $\mathcal A_G(M^n)$}
%
In this section, we indicate how to extend the proof of Theorem \ref{thm:wclos} from $\widetilde{\mathcal A}_G([-1,1]^n)$ to prove Theorem \ref{thm:weakclosure}.
\begin{proof}[Proof of Theorem \ref{thm:weakclosure}, given Theorem \ref{thm:wclos}:]
We consider separately every family of slicing submanifolds $S^4$ as described in the statement of Definition \ref{def:weakconn} Theorem \ref{thm:weakclosure}. We will find that the weak limit of the $A_j$ coincides on such family with a connection which has, on almost all slices that form a neighborhood $U_{S^4}$ of a given slice $S^4$, local gauges in which it becomes $L^4$-integrable.

\medskip

\textbf{Step 1.}\textit{ Weak closure in $\widetilde{\mathcal A}_G([-1,1]^n)$ with a tame background metric.}   We first note that the proof of Theorem \ref{thm:wclos} holds as well when the base manifold $[-1,1]^n$ is endowed with a $C^1$-regular Riemannian metric $h$ such that $\vertii{h-id}_{C^1([-1,1]^n)}$ is small enough. Indeed, the only changes to be applied are in the computation of integrals, in which the volume form $\op{Vol}_h$ replaces the volume element, and in the computation of norms, where $\verti{\cdot}$ has to be replaced by $\verti{\cdot}_h$. This still allows to find good cubeulations such as in Proposition \ref{prop:centergridball}. In the proof of the approximation theorem \ref{thm:strongapprox}, the hypothesis that $h$ is close to the identity allows to still obtain the needed bounds \eqref{psipresnorm} for domains making up the given cubeulation. The rest of the proofs are easily adaptable to the present case.

\medskip

\textbf{Step 2.} \textit{Deformation and localization.} We note, that up to perturbing the $f$ appearing in Definition \ref{def:weakconn}, we may assume that for a.e. $y\in \mathbb R^{n-4}$ with corresponding levelset $S^4=f^{-1}(y)$ corresponding to a regular value $y\in \mathsf{Reg}(f)$, we have for $r>0$ small enough, that a neighborhood $U_{S^4}=f^{-1}(B_r(y))$ is foliated by levelsets corresponding to regular values of $f$ as well. Then $U_{S^4}$ is $C^1$-diffeomorphic to $S^4\times B_r(0)$ and is thus the union of finitely many charts $U_\alpha$ which are $C^1$-diffeomorphic to $[-1,1]^n$ with a Riemannian metric close to the Euclidean one. In these charts the slices by $f^{-1}(y')\cap U_\alpha$ with $y'\in B_r(y)$ which we need to consider are sent to the sets $[-1,1]^4\times\{T\}$, for $T\in[-1,1]^{n-4}$. By using Lemmas \ref{lem:C1perturbation} and \ref{lem:localization}, we then reduce to the case considered in Step 1, and this concludes the proof.
\end{proof}
\appendix
%
%

%
%
\section{Distances and equivalence relations on connection and curvature forms}\label{app:distances}
%
%
In this section we use the notation from \eqref{def_delta}, \eqref{def_donaldson_conn}, \eqref{def_donaldson_curv} and \eqref{def_delta_curv}, but for simplicity of notations we drop the subscripts ``$\mathrm{conn}$'' and ``$\mathrm{curv}$''.
%
%
\subsection{Geometric distances on $2$-forms}
%
%
Below we use the notation $F=\sum_{i<j}F_{ij}dx_i\wedge dx_j$ for a $\mathfrak g$-valued $2$-form, where $F_{ij}\in \mathfrak{g}$. We then define the following pointwise distances between such forms:
\begin{subequations}\label{dist_pw_2form}
\begin{eqnarray}
d_\mathrm{pw}(F, F')^2&:=&\min_{g\in G}\left|g^{-1}Fg-F'\right|^2=\sum_{i<j}\left|g^{-1}F_{ij}^{(1)}g - F_{ij}^{(2)}\right|^2\ ,\\[3mm]
\delta_\mathrm{pw}(F, F')^2&:=&\frac{2}{(n-2)(n-3)}\sum_{\substack{J\subset\{1,\ldots,n\}\\ \#J=4}}\min_{g(J)\in G}\sum_{\substack{i<j\\i,j\in J}}\left|g(J)^{-1}F_{ij}g(J)-F'_{ij}\right|^2\ .\quad\quad
\end{eqnarray}
\end{subequations}
We see easily that $\delta_{\mathrm{pw}}\le d_{\mathrm{pw}}$, keeping in mind that each pair $ij$ belongs to the $4$-ple $J$ for $\tfrac{(n-2)(n-3)}{2}$ distinct $4$-ples $J\subset\{1,\ldots,n\}$. The above pointwise definitions directly extend by integration to distances $d,\delta_1$ on $L^2$-forms $F,F'\in L^2(M^n, \wedge^2 TM\otimes\mathfrak g)$. In the case of $d_{\mathrm{pw}}$, we find again the definition \eqref{def_donaldson_curv} 
\begin{subequations}\label{distforms_app}
\begin{equation}\label{dist_2form}
d(F,F')^2=\int_{M^n}d_{\mathrm{pw}}(F(x), F'(x))^2d\mathrm{vol}_h(x)=\inf_{g: M^n\to G}\int_{M^n}\left|g^{-1}Fg-F'\right|^2d\mathrm{vol}_h,
\end{equation}
and from $\delta_\mathrm{pw}$ we define
\begin{equation}\label{dist_2form2}
\delta(F,F')^2:=\int_{M^n}\delta_\mathrm{pw}(F(x),F'(x))^2\mathrm{d}vol_h.
\end{equation}
\end{subequations}
In the case $M^n=[-1,1]^n$ we may re-express the above directly via \eqref{f-coord} and find a distance which is equivalent to $\tilde\delta$ defined like \eqref{tilde_delta} and to $\delta$ as defined in \eqref{def_delta_curv}:
\begin{subequations}\label{equivalences_cube}
\begin{eqnarray}\label{dist_2form3}
\delta_1(F,F')^2&=&\sum_{f\in \mathcal C_{n,n-4}}\inf_{g:[-1,1]^n\to G}\int_{[-1,1]^n}\left|(g^{-1}Fg - F')\wedge f^*\omega\right|^2\frac{\mathrm{d}vol}{|f^*\omega|}\nonumber\\[3mm]
&\asymp&\tilde \delta(F,F')^2 \label{tilde_delta_eq}\\[3mm]
&\asymp&\sup_{f\in\mathrm{Lip}([-1,1]^n,\mathbb R^{n-4})}\inf_{g:[-1,1]^n\to G}\int_{[-1,1]^n}\left|(g^{-1}Fg - F')\wedge f^*\omega\right|^2\frac{\mathrm{d}vol}{|f^*\omega|}\quad\quad\label{delta_eq}\\[3mm]
&=&\delta(F,F')^2.\label{delta_eq2}
\end{eqnarray}
\end{subequations}
In the above, the equivalence \eqref{tilde_delta_eq} follows by comparison between the supremum and the sum, with implict constant depending only on $n$, and the equivalence \eqref{delta_eq},\eqref{delta_eq2} follows by localizing the pointwise distance equivalence
\begin{equation}\label{pw_eq_4planes}
\delta_\mathrm{pw}(F(x),F'(x))^2\asymp\sup_{H\in\mathrm{Gr}(n,n-4)}\inf_{g_H\in G}\left|g_H^{-1}i^*_HFg_H-i^*_HF'\right|^2.
\end{equation}
From the equivalences \eqref{equivalences_cube} we can find the equivalence between the distances defined in terms of all the intermediate clases $\mathcal C$ of slicing functions $f$ such that $\mathcal C_{n,n-4}\subseteq \mathcal C\subseteq \mathrm{Lip}([-1,1]^n,\mathbb R^{n-4})$.

\medskip

While as a direct consequence of the definition $d(F,F')=0$ if and only if $F,F'$ are gauge-equivalent by a measurable gauge transformation, on the other hand, we couldn't prove that for general $G$ the same is true under the a priori weaker equivalent conditions that $\delta_1(F,F')=0\Leftrightarrow\tilde\delta(F,F')=0\Leftrightarrow\delta(F,F')=0$. In the next subsection however, we prove this in the case of $G=SU(2)$.

\subsubsection{The case of $SU(2)$}
We recall a series of very well-known identifications concerning the groups $SU(2)$, $Sp(1)$ and $SO(3)$, that unfold as follows. Recall the bijective maps
\[
Sp(1)\ni w+\mathbf{i}x+\mathbf{j}y+\mathbf{k}z = \alpha+\mathbf{j}\beta\simeq\left(\begin{array}{cc}\alpha&-\bar\beta\\\beta&\bar\alpha\end{array}\right)\in SU(2)
\]
and 
\[
\R^3\ni(a_1,a_2,a_3)\simeq \mathbf{i}a_1+\mathbf{j}a_2+\mathbf{k}a_3\in \op{Im}\mathbb H\simeq\left(\begin{array}{cc}ia_1&a_2+ia_3\\-a_2+ia_3& -ia_1\end{array}\right)\in\mathfrak{su}(2)\ .
\]
Then one directly verifies that the actions 
\[
SU(2)\times\mathfrak{su}(2)\ni(g,A)\mapsto g^{-1}Ag\in\mathfrak{su}(2)\ ,\quad\quad Sp(1)\times\op{Im}\mathbb H\ni (q,v)\mapsto q^{-1}vq\in\op{Im}\mathbb H
\]
are in fact the same action, if viewed under the above identifications. Moreover if $q=w+\mathbf{i}x+\mathbf{j}y+\mathbf{k}z\in Sp(1)$, $\op{Im}\mathbb H\ni v\simeq \vec a\in \R^3$ as above, then 
\[
q^{-1}vq=(-q)^{-1}vq=R_q\vec a,
\]
where the map $Sp(1)\ni q \mapsto R_q\in SO(3)$ is a $2:1$ covering of $SO(3)$ by $Sp(1)$, and $R_q=R_{w+\mathbf{i}x+\mathbf{j}y+\mathbf{k}z}$ is the rotation by $2\theta$ around $(x,y,z)\in \R^3$, where $\cos\theta=w$. 

\medskip

We then find that if a set of vectors in $\R^3$ are identified under $SO(3)$-rotation, then the corresponding matrices in $\mathfrak{su}(2)$ are identified under $SU(2)$-conjugation action, where the identification is uniquely determined modulo a $\mathbb Z/2\mathbb Z$-action.
\begin{proposition}\label{prop:su2-kr}
If $G=SU(2)$ then the pseudo-distances $\delta_{\mathrm{pw}}$ and $d_{\mathrm{pw}}$  are equivalent.
\end{proposition}
\begin{proof}
We note that since the space $\wedge^2\R^n\otimes\mathfrak{g}$ is finite-dimensional, thus it suffices to show that $\delta_{\mathrm{pw}}(F, G)=0$ if and only if $d_{\mathrm{pw}}(F,G)=0$. One implication follows directly from the previous observation of the stronger fact $\delta_{\mathrm{pw}}\le d_{\mathrm{pw}}$. The implication $\delta_{\mathrm{pw}}(F,G)=0\Rightarrow d_{\mathrm{pw}}(F,G)=0$ is based on an argument already present in \cite[Lem. 3.7]{kessel}, which we slightly extend. If $A,B\in\mathfrak{su}(2)$ are represented by vectors $\vec a, \vec b\in \R^3$ under the identification $\mathfrak{su}(2)\simeq \R^3$, then $\op{tr}(AB)=2\vec a\cdot\vec b$. We denote $\vec f_{ij}, \vec g_{ij}\in\R^3$ vectors identified to the $2$-form coefficients $F_{ij}, G_{ij}$. If $\delta_{\mathrm{pw}}(F,G)=0$ then for any $4$-ple $J\subset\{1,\ldots,n\}$ there exist $g_J\in SU(2)$ such that $g^{-1}F_{ij}g=G_{ij}$ for all $i<j\in J$, and by thus there exists $R_J\in SO(3)$ such that $R_J\vec f_{ij}=\vec g_{ij}$ for all $i<j\in J$. In particular, $\vec f_{ij}\cdot \vec f_{kl}=\vec g_{ij}\cdot \vec g_{kl}$ for any two pairs $i<j, k<l\in \{1,\ldots,n\}$. Due to the general fact that the set of pairwise scalar products form a complete $SO(\ell)$-invariant for $\kappa$-ples of vectors in $\R^\ell$ (as can be easily seed by Gram-Schmidt orthonormalization, see also \cite[Ch.14]{weyl}), in the case at hand we find that there exists $R\in SO(3)$ for which $R\vec f_{ij}=\vec g_{ij}$ for all $i<j\in\{1,\ldots,n\}$. By the above identifications, such $R$ determines a $g\in SU(2)$ and up to a factor $\varepsilon\in\{\pm1\}$, such that $g^{-\varepsilon}F_{ij}g^{\varepsilon}=G_{ij}$ for all $i<j\in\{1,\ldots,n\}$. This shows that $d_{\mathrm{pw}}(F,G)=0$, completing the proof of the proposition.
\end{proof}
From the above proposition and the definition \eqref{dist_2form} (and its analogue for the distance $\delta$ on connection forms) we directly have the following:
\begin{corollary}
If $G=SU(2)$ then with the above notations $d\asymp\delta$, with implicit constant depending only on $n$.
\end{corollary}
Proving the same result for $G$ other than $SU(2)$ would answer Question \ref{q:equiv_general} from the introduction, and we leave this as an open question:
\begin{question}\label{quest:main_equiv}
Under which conditions on $G,n$ are the pseudo-distances $d_\mathrm{pw}$ and $\delta_\mathrm{pw}$ equivalent over $\mathfrak{g}$-valued $2$-forms in $\mathbb R^n$?
\end{question}
This type of question seems to be related to the theory of invariants on Lie groups. Indeed, a seemingly closely related question is what are \emph{minimal conditions} that allow to infer that for two $m$-ples of matrices $(M_1,\ldots,M_m)$ and $(\bar M_1,\ldots,\bar M_m)$ with $M_j, \bar M_j\in \mathfrak{su}(n)$ for $j=1,\ldots,m$, there exists $g\in SU(n)$ such that
\[
(\bar M_1,\bar M_2,\ldots,\bar M_m)=(g^{-1}M_1g,g^{-1}M_2g,\ldots,g^{-1}M_mg).\ .
\]
This type of question appears in the theory of polynomial invariants, see e.g. \cite{procesi} or \cite{shapiro}, in which we have to replace the role of $O(n)$ by $SU(n)$. However the tools which connect such study to Question \ref{quest:main_equiv} seem to not be sufficiently developed yet.
%
%
%
\subsection{Distances on connection $1$-forms and on curvature $2$-forms}
%
We start by proving that our alternative definitions of Donaldson-type distances between connection forms are actually equivalent:
\begin{lemma}\label{lem:equiv}
Let $(M^n,h)$ be a compact Riemannian manifold. For $A,A'\in L^2(\wedge^1M^n,\mathfrak g)$ and for $k\ge 1$, the following holds, with an implicit constant depending only on $(M^n,h)$:
\begin{multline}\label{equiv_donald}
\inf_{g\in W^{1,2}(M^n,G)}\int_{M^n}\left|g^{-1}dg+g^{-1}Ag-A'\right|_h^2\mathrm{d}vol_h\\
\asymp \inf_{\substack{g:M^n\to G\\\mathrm{measurable}}}\sup_{f\in C^\infty(M^n,\mathbb R^k)}\int_{M^n}\left|(dg+Ag-gA')\wedge f^*\omega\right|_h^2\frac{\mathrm{d}vol_h}{\left|f^*\omega\right|_h}.
\end{multline}
\end{lemma}
\begin{proof}
We first show, for the case $g\in W^{1,2}(M^n,G)$, the equivalence
\begin{multline}\label{equiv_donald_1}
\int_{M^n}\left|g^{-1}dg+g^{-1}Ag-A'\right|_h^2\mathrm{d}vol_h\\
\asymp \sup_{f\in C^\infty(M^n,\mathbb R^k)}\int_{M^n}\left|(g^{-1}dg+g^{-1}Ag-A')\wedge f^*\omega\right|_h^2\frac{\mathrm{d}vol_h}{\left|f^*\omega\right|_h}.
\end{multline}
After establishing \eqref{equiv_donald_1} for $M^n=[-1,1]^n$, we can pass to the case of general compact manifolds $M^n$ by the covering argument of Section \ref{sec:global}. For $M^n=[-1,1]^n$ \eqref{equiv_donald_1} follows by estimating the supremum above and below by a finite sum, as done for two-forms in \eqref{equivalences_cube}. Together with the co-area formula, this completes the proof of \eqref{equiv_donald_1} for $g\in W^{1,2}(M^n,G)$. 

\medskip

Again, reducing without loss of generality to the case $M^n=[-1,1]^n$, we next note that if the distributionally defined form $dg + Ag-gA'$ is represented by an $L^2$-form, then due to the fact that $A,A'\in L^2(f^{-1}(y),\wedge^1\mathbb R^n\otimes\mathfrak g)$ we have $Ag,gA'\in L^2(f^{-1}(y),\wedge^1\mathbb R^n\otimes\mathfrak g)$ as well, and thus by triangle inequality $dg\in L^2$ and thus $g\in W^{1,2}(f^{-1}(y),G)$. Also for $g\in G$ and $a\in \mathfrak g$, our norm satisfies $|a|=|ga|=|ag|$, and in particular
\[
\left|(g^{-1}dg+g^{-1}Ag-A')\wedge f^*\omega\right|=\left|(dg+Ag-gA')\wedge f^*\omega\right|.
\]
By testing the second line of \eqref{equiv_donald} against coordinate functions $f\in \mathcal C_{n,n-4}$, we then find that $g\in W^{1,2}(H)$ contemporarily for all coordinate hyperplanes $H$, and thus $g\in W^{1,2}([-1,1]^n,G)$ like in the first line of \eqref{equiv_donald}, and we are then justified to use interchangeably \eqref{equiv_donald_1} for $g\in W^{1,2}$ only, and this completes the proof.
\end{proof}
In order to define the analogues of $d, \delta$ of from \eqref{dist_2form} and \eqref{dist_2form2} for connection forms, due to the non-pointwise dependence on $g$ of the gauge-transformed connection forms $g^{-1}dg+g^{-1}Ag$, we can only use the integral formulations directly, and we find again the Donaldson distance and, respectively, a distance equivalent to $\tilde\delta$ on $[-1,1]^n$ and to $\delta$ on general manifolds $M^n$:
\begin{subequations}\label{dist_1form}
\begin{equation}\label{bard1}
d(A,A')^2:=\min\left\{\left\lVert  g^{-1}dg + g^{-1}Ag-A'\right\rVert_{L^2([-1,1]^n)}^2:\ g:[-1,1]^n\to G\mbox{ measurable}\right\}.
\end{equation}
and denoting 
\begin{equation}\label{integrand_4d}
(*):=\min_{g:[-1,1]^4\to G}\left\lVert g^{-1}dg + g^{-1}i_{H(J,T)}^*Ag - i_{H(J,T)}^*B\right\rVert_{L^2([-1,1]^4)}^2,
\end{equation}
we have 
\begin{eqnarray}\label{bard2}
\delta^2(A,B)&:=&\frac{2}{(n-2)(n-3)}\sum_{\substack{J\subset \{1,\ldots,n\}\\ \#J=n-4}}\int_{[-1,1]^J}(*)\ \mathrm{d}T\\
\end{eqnarray}
\end{subequations}
We can directly see by comparing definitions, that $\tilde\delta$ from \eqref{tilde_delta} is equivalent to $\tilde\delta_2$, thus making a link to the study from the previous sections, and to the distance $\delta$ described in the introduction in \eqref{def_delta}.

\medskip

The following useful approximation result will be proved in a forthcoming work \cite{PPR}:

\begin{lemma}\label{lem:approx_lem}
If $A,B\in \widetilde{\mathcal A}_G([-1,1]^n,\R^n\otimes\mathfrak{g})$ are weak connection forms such that $F_A=F_B$, then for almost all $2$-dimensional surfaces $S^2\subset[-1,1]^n$ there exist smooth forms $A_k,B_k\in \Omega_1(S^2,\mathfrak{g})$ such that $A_k\to A$ and $B_k\to B$ in $L^2$ and furthermore $F_{A_k}=F_{B_k}$.
\end{lemma}
Using the above approximation result we can prove the following:
\begin{proposition}\label{prop:2-planes_eq}
Let $A,B\in \widetilde{\mathcal A}_G([-1,1]^n,\R^n\otimes\mathfrak{g})$ be two weak connection $1$-forms with curvature forms $F_A, F_B$, respectively. There exists a measurable function $h:[-1,1]^n\to G$ such that $h^{-1}F_Ah=F_B$ if and only if there exists a measurable function $g:[-1,1]^n\to G$ such that $g^{-1}dg + g^{-1}Ag=B$.
\end{proposition}
\begin{proof}
The existence of $g$ implies the existence of $h$ as above, because 
\[
F_{g^{-1}dg+g^{-1}Ag}=g^{-1}F_Ag\ ,
\]
and we can then take $h:=g$.

\medskip

We now concentrate on the opposite implication: assuming that there exists $h$ such that $h^{-1}F_Ah=F_B$, we prove that there exists $g$ such that $g^{-1}dg+g^{-1}Ag=B$. We may assume that $h\equiv id$ without loss of generality, up to replacing $A$, $g$ by $h^{-1}dh+h^{-1}Ah$, $gh^{-1}$, respectively.

\medskip

We first note that for any Lipschitz injective curve $\gamma$, if the $\mathfrak{g}$-valued $1$-forms $A$, $B$ are integrable along $\gamma$ then we can always explicitly solve the equation
\begin{equation}\label{g_gamma_sol}
g^{-1}\partial_{\dot\gamma}g +g^{-1}A_{\dot\gamma}g=B_{\dot\gamma}\ ,\quad g(\gamma(0))=id\ .
\end{equation}
Indeed, the solution is explicitly expressed as
\begin{equation}\label{sol_g_gamma}
g(\gamma(t))=P\left(\gamma|_{[0,t]}, B\right)^{-1}P\left({\gamma|_{[0,t]}}, A\right)\ ,
\end{equation}
where the time-ordered path integrals $P(\gamma,A)$ appearing in \eqref{sol_g_gamma} are defined as follows. For a curve $\gamma$, in order to define $P(\gamma, A)$ we associate to each Riemann sum $R_N:=\sum_{j=1}^N\int_{\gamma_j}A\in\mathfrak{g}$ corresponding to a partition of $\gamma$ into a concatenation of injective curves $\gamma_j$, the parameter-ordered product
\begin{equation}\label{exp_g_riemsum}
\op{exp}(R_N):=\prod_{j=1}^N\op{exp}_G\int_{\gamma_j}A\ ,
\end{equation}
where now $\op{exp}_G$ equals the usual exponential map of $G$, which is well-defined for $\int_{\gamma_j}A\in \mathfrak{g}$ small enough. Then taking the limit of the expressions \eqref{exp_g_riemsum} along any sequence of refining Riemann sums $R_N\to\int_{\gamma}A$, we obtain the definition $P(\gamma,A):=\lim_{R_N\to\int_\gamma A}\op{exp}(R_N)$. The fact that $g$ as defined in \eqref{sol_g_gamma} solves \eqref{g_gamma_sol} follows directly by differentiation. The fact that the solution to \eqref{g_gamma_sol} is unique follows from the classical theory of ODEs.

\medskip

If we consider two different injective paths $\gamma^{(1)}, \gamma^{(2)}: [0,1]\to [-1,1]^n$ along which $A$ and $B$ are integrable and such that $\gamma^{(1)}(0)=\gamma^{(2)}(0)=0$ and $\gamma^{(1)}(1)=\gamma^{(2)}(1)=p\in[-1,1]^n$ that meet only in $0$ and $p$, then the condition for the solutions to the corresponding equations \eqref{g_gamma_sol} to coincide at the common point $p$ is
\begin{equation}\label{cond_coincide_1}
P\left({\gamma^{(1)}}, B\right)^{-1}P\left({\gamma^{(1)}}, A\right) = P\left({\gamma^{(2)}}, B\right)^{-1}P\left({\gamma^{(2)}}, A\right),
\end{equation}
which is equivalent to
\begin{equation}\label{cond_coincide_2}
P\left({\gamma^{(2)}}, B\right)P\left({\gamma^{(1)}}, B\right)^{-1} = P\left({\gamma^{(2)}}, A\right)P\left({\gamma^{(1)}}, A\right)^{-1}\ .
\end{equation}
By coming back to the expressions as limits of \eqref{exp_g_riemsum}, we see that \eqref{cond_coincide_2} is directly re-expressed in terms of the solutions along the loop $\gamma:=\gamma^{(1)}*(\gamma^{(2)})^{-1}$, where by $*$ we denote the concatenation of paths, and $\gamma^{-1}$ represents the path $\gamma$ parameterized backwards, i.e. $\gamma^{-1}(t)=\gamma(1-t)$. In this notation, equation \eqref{cond_coincide_2} becomes the following:
\begin{equation}\label{cond_coincide_3}
P(\gamma,B)=P(\gamma,A)\ ,
\end{equation}
which in geometric terms is nothing else but the condition that the holonomies of $B$ and $A$ coincide along the loop $\gamma$ starting from $0$. By considering a surface $S^2\subset[-1,1]^n$ such that $\partial S^2$ is parameterized by $\gamma$, and along which $F_A$ and $F_B$ are integrable, we claim that 
\begin{equation}\label{cond_coincide_4}
 F_A= F_B\quad\Rightarrow \quad  P(\gamma,A)=P(\gamma,B)\quad\ .
\end{equation}
To prove the above we may first use Lemma \ref{lem:approx_lem} and for the purposes of \eqref{cond_coincide_4} we may assume that $A,B$ are smooth, and up to reparameterization we assume $S^2=[0,1]^2$. In this case we subdivide $S^2=[0,1]^2$ into small squares of size $\epsilon$ and consider the discrete homotopy between the loop $\gamma$ based at $0$ and with image $\partial [0,1]^2$, and the trivial loop. We note that the homotopy can be subdivided into steps each of which consists in applying the inverse of the holonomy along the polygonal loop along a square of size $\epsilon$. We denote this loop by $\gamma_p$ and let the corresponding square be $\{p,p+(\epsilon,0),p+(\epsilon,\epsilon),p+(0,\epsilon)\}$. Then there holds
\begin{equation}\label{smalloop_p}
P(\gamma_p,A)=1_G+\epsilon^2 F_A(p)[e_1\wedge e_2]+o(\epsilon^2)\ .
\end{equation}
As $F_A=F_B$, we find that the error between the compositions of all the above elementary homotopies for $A,B$ differs by a quantity bounded by 
\[
o_{\epsilon\to 0}(1)\int_{[0,1]^2} |F_A|d\mathcal H_2\ ,
\]
Therefore as $\epsilon\to 0$ this error tends to zero, thus \eqref{cond_coincide_4} holds.

\medskip

As \eqref{cond_coincide_4} allows to prove $P(\gamma, A)=P(\gamma,B)$ for almost all $\gamma$ we conclude the proof that \eqref{cond_coincide_1} also holds, and thus for any two paths $\gamma^{(1)}$, $\gamma^{(2)}$ along which $A$ and $B$ are integrable we have, with the notation \eqref{sol_g_gamma} for the solution of \eqref{g_gamma_sol}, 
\begin{equation}\label{equal_g}
\gamma^{(1)}(t)=\gamma^{(2)}(t)\Rightarrow g(\gamma^{(1)}(t))=g(\gamma^{(2)}(t))\ .
\end{equation}
This means that the solutions of \eqref{g_gamma_sol} uniquely define a global $g$ over $[-1,1]^n$ on a full measure set. The fact that such $g$ satisfies \eqref{g_gamma_sol} along all paths implies in particular (by taking $\gamma=\gamma_{p,j}$ such that $\gamma(t)=p, \dot\gamma(t)=e_j$ for arbitrary $p\in[-1,1]^n$, $j\in\{1,\ldots,n\}$) there holds
\begin{equation}\label{gaugechange}
\forall j\in\{1,\ldots,n\}\ ,\quad g^{-1}\partial_jg + g^{-1}A_jg=B_j\ .
\end{equation}
Thus $g^{-1}dg+g^{-1}Ag=B$, and the proof is complete.
\end{proof}
From the above we directly have the following:
\begin{corollary}\label{cor:gauge-dist}
Let $A,B\in\mathcal A_G([-1,1]^n)$, and consider the pseudo-distances $\delta, d$ be defined over curvature forms as in \eqref{distforms_app} and over connection forms as in \eqref{dist_1form}. Then there holds $d(F_A,F_B)=0$ if and only if $d(A,B)=0$, and $\delta(F_A,F_B)=0$ if and only if $\delta(A,B)=0$.
\end{corollary}

%
%
%
%
%
%
%
%
%
%
\bibliographystyle{plain}

\end{document}